\documentclass{amsart}
\usepackage{amsfonts,amscd,amsthm,amsgen,amsmath,amssymb}
\usepackage[all]{xy}
\usepackage{epsfig,color}
\usepackage[vcentermath]{youngtab}
\newtheorem{theorem}{Theorem}[section]
\newtheorem{lemma}[theorem]{Lemma}

\newtheorem{proposition}[theorem]{Proposition}

\theoremstyle{definition}
\newtheorem{definition}[theorem]{Definition}
\newtheorem{example}[theorem]{Example}

\theoremstyle{remark}
\newtheorem{remark}[theorem]{Remark}

\numberwithin{equation}{section}

\begin{document}

\setlength\parskip{0.5em plus 0.1em minus 0.2em}

\title[A Fourier--Mukai transform for $KR$-theory]{A Fourier--Mukai transform for $KR$-theory}
\author{David Baraglia}

\address{School of Mathematical Sciences, The University of Adelaide, Adelaide SA 5005, Australia}

\email{david.baraglia@adelaide.edu.au}


\date{\today}

\begin{abstract}

In complex K-theory, the Fourier--Mukai transform is an isomorphism between K-theory groups of a torus and its dual torus which is defined by pullback, tensoring by the Poincar\'e line bundle and pushforward. The Fourier--Mukai transform extends to families of dual tori provided one works with twisted K-theory. The Fourier--Mukai transform is then an isomorphism between twisted K-theory groups of $T$-dual torus bundles. In this paper we prove an extension of these results to twisted KR-theory. We introduce a notion of Real $T$-duality for torus bundles with Real structures and prove the existence of Real $T$-duals. We then define a Real Fourier--Mukai transform for Real $T$-dual torus bundles and prove that it is an isomorphism. Lastly, we consider an application of these results to the families index of Real Dirac operators which is relevant to Real Seiberg--Witten theory.

\end{abstract}

\maketitle




\section{Introduction}

The Fourier--Mukai transform in complex $K$-theory is a topological counterpart of the Fourier--Mukai transform in algebraic geometry. In its simplest form, the Fourier--Mukai transform is an isomorphism $\Phi : K^*(X) \to K^{*-dim(X)}( \widehat{X} )$ between the $K$-theories of dual tori and is given by pulling back to the product $X \times \widehat{X}$, tensoring by the Poincar\'e line bundle, and pushing forward to $\widehat{X}$. The Fourier--Mukai transform can be extended to families of dual tori $X \to B$, $\widehat{X} \to B$ over a common base $B$, except now it becomes necessary to make use of twisted $K$-theory. The necessity of using twisted $K$-theory arises because a twist on one side corresponds to a non-trivial Chern class on the dual side. The families $X$, $\widehat{X}$ are {\em topologically $T$-dual} in the sense of \cite{bem,bhm,bs,brs,bar1,bar2}. $T$-duality is a duality in string theory which roughly speaking, replaces a family of tori with a family of dual tori. D-branes in string theory have topological charges which are valued in complex $K$-theory \cite{mm,wit}, or more generally in twisted $K$-theory in the presence of a non-trivial $B$-field \cite{bm}. From this point of view, the Fourier---Mukai transform maps the charge of a D-brane to the charge of the corresponding $T$-dual D-brane.

For orientifold string theories, D-brane charges are valued in twisted $KR$-theory \cite{guk,hor}, hence it is natural to expect that both $T$-duality and the Fourier--Mukai transform extend to the setting of $KR$-theory. The main goal of this paper is to prove that such extensions exist and that the $KR$-theoretic Fourier--Mukai transform is an isomorphism. Some related results can be found in \cite{ddr1,ddr2,bar4}. Aside from the connection to string theory, a separate motivation for the $KR$-theoretic Fourier--Mukai transform comes from Seiberg--Witten theory. In \cite{bar4}, we made use of the $KR$-theoretic Fourer--Mukai transform in order to compute the mod $2$ Seiberg--Witten invariants of spin structures on $4$-manifolds. $KR$-theory on tori (or torus bundles) is also relevant for Real Seiberg--Witten theory \cite{bar5}.

Recall that to define the Real K-theory of a space $X$, one must equip $X$ with a {\em Real structure}, that is a homeomorphism $\sigma : X \to X$ satisfying $\sigma^2 = id_X$. To define twisted $KR$-theory, we must also equip $X$ with a (graded) {\em Real gerbe} $\mathcal{G}$. We recall the definition in Section \ref{sec:grg}. If $X$ is a torus, we say that a Real structure $\sigma$ on $X$ is {\em affine} if it is induced by an affine transformation on the universal cover of $X$. More generally, we consider Real affine torus bundles $X \to B$ (see \textsection \ref{sec:ratb}). This means that $X$ is a torus bundle over a base space $B$ whose transition maps are affine transformations and furthermore $X$ is equipped with a Real structure $\sigma_X$ which covers a Real structure $\sigma_B$ on $B$ and such that $\sigma_X$ acts affinely on the fibres. We also require that the gerbe $\mathcal{G}$ is an {\em affine gerbe} (see \textsection \ref{sec:ag}). This somewhat technical condition ensures that the pair $(X , \mathcal{G})$ admits a $T$-dual pair.

In order to define the $KR$-theoretic Fourier--Mukai transform we must first define what it means for two Real affine torus bundles $X$, $\widehat{X}$ to be Real $T$-duals. The precise definition will be given in Section \ref{sec:rtd}. Here we sketch the definition. Let $B$ be a topological space with a Real structure $\sigma_B$ ($B$ is always taken to be locally contractible and paracompact). Let $\pi : X \to B$, $\widehat{\pi} : \widehat{X} \to B$ be affine torus bundles over $B$ of the same rank, equipped with Real structures $\sigma_X$, $\sigma_{\widehat{X}}$ covering $\sigma_B$. Let $\mathcal{G}, \widehat{\mathcal{G}}$ be affine graded Real gerbes on $X, \widehat{X}$. Let $C = X \times_B \widehat{X}$ be the fibre product and $p : C \to X$, $\widehat{p} : C \to \widehat{X}$ the projections to $X$ and $\widehat{X}$. The pairs $(X , \mathcal{G})$, $(\widehat{X} , \widehat{\mathcal{G}})$ are said to be {\em Real $T$-duals} if there is an isomorphism 
\[
\gamma : p^*(\mathcal{G}) \to \widehat{p}^*(\widehat{\mathcal{G}})
\]
of graded Real gerbes which satisfies a condition which we call the {\em Poincar\'e property}. This property ensures that in a certain sense, the isomorphism $\gamma$ locally resembles tensoring by the Poincar\'e line bundle. In particular this property implies that monodromy representations of $X$ and $\widehat{X}$ are dual, as one would expect. Our first main result is the existence of Real $T$-duals. We give a simplified statement here, the more precise result is Theorem \ref{thm:rtd}.

\begin{theorem}\label{thm:1}
Let $X \to B$ be a Real affine torus bundle on $B$ and let $\mathcal{G}$ be an affine graded Real gerbe on $X$. Then there exists a Real $T$-dual pair $(\widehat{X} , \widehat{\mathcal{G}})$.
\end{theorem}

Now let $(X , \mathcal{G})$, $(\widehat{X} , \widehat{\mathcal{G}})$ be Real $T$-dual pairs over $B$ and let $\gamma : p^*(\mathcal{G}) \to \widehat{p}^*(\widehat{\mathcal{G}})$ be an isomorphism satisfying the Poincar\'e property. The {\em Real Fourier--Mukai transform} (with respect to $(X , \mathcal{G} , \widehat{X} , \widehat{\mathcal{G}} , \gamma)$) is the map
\[
\Phi : KR^j(X , \mathcal{G} \otimes L(V) ) \to KR^{j-n}(\widehat{X} , \widehat{\mathcal{G}})
\]
(where $n$ is the dimension of the fibres of $X$ and $L(V)$ is the lifting gerbe of the vertical tangent bundle of $X \to B$, see \textsection \ref{sec:lg}) given by pulling back to $X \times_B \widehat{X}$, applying $\gamma$ and then pushing forward to $\widehat{X}$:
\[
\Phi = \widehat{p}_* \circ \gamma \circ p^*.
\]

Our second main result is that the Real Fourier--Mukai transform is an isomorphism.

\begin{theorem}\label{thm:2}
Assume that $B$ is a compact smooth manifold and $\sigma_B$ is a smooth involution. Then the Real Fourier--Mukai transform $\Phi : KR^*(X , \mathcal{G} \otimes L(V_\rho) ) \to KR^{*-n}(\widehat{X} , \widehat{\mathcal{G}})$ is an isomorphism for any Real $T$-dual pair $(X , \mathcal{G}), (\widehat{X} , \widehat{\mathcal{G}})$ and any isomorphism $\gamma : p^*(\mathcal{G}) \to \widehat{p}^*(\widehat{\mathcal{G}})$ satisfying the Poincar\'e property.
\end{theorem}

The proof of Theorem \ref{thm:2} involves several steps. First a Mayer--Vietoris argument is used to reduce to the case that $B$ is a point. However for this to work one needs a compatibility between the Mayer--Vietorus sequence and the Fourier--Mukai transform. The key technical result used here is Proposition \ref{prop:mv}. Having reduced to the case where $B$ is a point we now have that $X$ and $\widehat{X}$ are dual tori. In contrast to complex $K$-theory, it is still not straightforward to show that the Fourier--Mukai transform is an isomorphism even in this case. In the complex case, the complex $K$-theory of a torus has no torsion, so one can pass to cohomology and deduce that the Fourier--Mukai transform is an isomorphism by a direct calculation. In the case of $KR$-theory there is torsion, so such a simplification is not possible. Another difference between the complex case is that there are non-trivial Real affine gerbes even when $B$ is a point. To prove Theorem \ref{thm:2} in the case that $B$ is a point, we first show that any Real affine torus equipped with a Real affine gerbe can be factored into a product of indecomposable tori. We further show that the Fourier--Mukai transform factors into a product of partial Fourier--Mukai transforms, one for each indecomposable factor. Then by another Mayer--Vietoris argument we are able to reduce to the case that $X$ is an indecomposable Real affine torus. It turns out that there are exactly five isomorphism classes of such tori. We show by a direct calculation that the Real Fourier--Mukai transform is an isomorphism in each of the five cases and from this Theorem \ref{thm:2} follows.

In the final section of the paper, we consider an application of the Real Fourier--Mukai transform to the families index of Real Dirac operators. Suppose $X$ is a compact smooth $n$-manifold and $\sigma_X$ is a smooth, orientation preserving Real structure on $X$. Let $\mathfrak{s}$ be a spin$^c$-structure on $X$ with corresponding $Spin^c(n)$-bundle $P \to X$. A {\em Real structure} on $\mathfrak{s}$ is a lift of $\sigma_X$ to an involution $\sigma_P : P \to P$ such that $\sigma_P(pg) = \sigma_P(p) c(g)$ for all $p \in P$, $g \in Spin^c(n)$, where $c$ is the involution on $Spin^c(n) = S^1 \times_{\pm 1} Spin(n)$ which is trivial on $Spin(n)$ and is complex conjugation on $S^1$. In such a situation the Real structure on $\mathfrak{s}$ promotes the index of the spin$^c$-Dirac operator $D$ from a class $ind(D) \in K^{-n}(pt)$ in complex $K$-theory to a class $ind_R(D) \in KR^{-n}(pt)$ in $KR$-theory. Extending this, we define the notion of a {\em Real spin$^c$-structure of type $k$}, where $k \in \mathbb{Z}/4\mathbb{Z}$. A Real spin$^c$-structure in the sense described above is the case $k = 0$. The cases $k \neq 0$ correspond to allowing $\sigma_X$ to be orientation reversing and also by considering lifts which square to $+1$ or $-1$. More precisely, if $\mathfrak{s}$ is a spin$^c$-structure on $X$, then a Real structure on $\mathfrak{s}$ of type $k$ is defined as follows. First we require that $\sigma_X$ changes the orientation on $X$ by $(-1)^k$. Second, we require that $\sigma_X$ lifts to a map $\sigma_{P'} : P' \to P'$, where $P'$ is the principal $Pin(n)$-bundle corresponding to $\mathfrak{s}$, for which $\sigma_{P'}(pg) = \sigma_{P'}(p) c(g)$ for all $p \in P'$, $g \in Pin(n)$ and $\sigma_{P'}^2 = (-1)^{\binom{k+1}{2}}$ (note that these two conditions fix the value of $k$ mod $4$). We then show that for a Real spin$^c$-structure of type $k$, the index of the Dirac operator lifts to a class $ind_R(D) \in KR^{2k-n}(pt)$. This leads to the following conclusion:

\begin{proposition}
Let $X$ be a compact smooth $n$-manifold with smooth Real structure $\sigma_X$. Let $\mathfrak{s}$ be a spin$^c$-structure and let $ind(D) \in \mathbb{Z}$ denote the index of the spin$^c$-Dirac operator associated to $\mathfrak{s}$. Suppose that $\mathfrak{s}$ admits the structure of a Real spin$^c$-structure of type $k \in \mathbb{Z}/4\mathbb{Z}$.
\begin{itemize}
\item[(1)]{If $n = 2k+4 \; ({\rm mod} \; 8)$ then $ind(D)$ is even.}
\item[(2)]{If $n = 2k \pm 2 \; ({\rm mod} \; 8)$ then $ind(D) = 0$.}
\end{itemize}
\end{proposition}

So far we only considered the index of a single Dirac operator. However, by coupling the Dirac operator to flat line bundles, we get a family of Dirac operators parametrised by the Jacobian $Jac(X) = H^1(X ; \mathbb{R})/H^1(X ; \mathbb{Z})$. The families index of the Dirac operator is now a $K$-theory class on the Jacobian $ind(D) \in K^{-n}(Jac(X))$. This families index is of particular interest for $3$- and $4$-manifolds because it appears in the context of Seiberg--Witten theory. If $\mathfrak{s}$ is a Real spin$^c$-structure on $X$ of type $k$, then we show that the families index can be promoted to a $KR$-theory class $ind_R(D) \in KR^{2k-n}(Jac(X))$. The dual torus to $Jac(X)$ is called the {\em Albanese torus} of $X$, denoted $Alb(X)$. It comes with a natural map $a : X \to Alb(X)$ (defined up to homotopy) with the property that $a^* : H^1(Alb(X) ; \mathbb{Z}) \to H^1(X ; \mathbb{Z})$ is an isomorphism (see Section \ref{sec:indr} for details of its construction). Assuming $\sigma_X$ has a fixed point, it is possible to define a linear Real structure on $Alb(X)$ for which $a$ is equivariant. Our last result expresses the families index $ind_R(D)$ in terms of the Real--Fourier Mukai transform
\[
\Phi : KR^*( Alb(X) ) \to KR^{*-\beta}( Jac(X) )
\]
where $\beta = b_+ - b_+$ and $b_{\pm}$ is the dimension of the $\pm 1$-eigenspace of $\sigma_X$ acting on $H^1(X ; \mathbb{R})$.

\begin{proposition}
Assume $\sigma_X$ has a fixed point. Then we have an equality
\[
ind_R(D) = \Phi( \alpha )
\]
where
\[
\alpha = a_*(1) \in KR^{2k-n-\beta}(Alb(X))
\]
is the pushforward of $1 \in KR^0(X)$ under the Albanese map $a : X \to Alb(X)$.
\end{proposition}

\subsection{Structure of the paper}

The structure is the paper is as follows. In Section \ref{sec:grg} we recall the definition and main properties of graded Real gerbes. In Section \ref{sec:tkr} we recall the main properties of twisted $KR$-theory that will be needed. In particular we prove a compatibillity bewtween the Mayer--Vietoris sequence and push-forward maps (Proposition \ref{prop:mv}). Section \ref{sec:rtd} is concerned with Real $T$-duality. We introduce the notion of Real affine torus bundles and affine Real gerbes. We then define Real $T$-duality and prove the existence of Real $T$-duals. In Section \ref{sec:rfmt} we introduce the Real Fourier--Mukai transform and prove that it is an isomorphism. The proof is broken into several steps. In Section \ref{sec:rtdpt} we reduce to the case that the base is a point. In Section \ref{sec:gentor} we reduce to the case of indecomposable tori. Finally in Section \ref{sec:indtor} we prove that the Real Fourer--Mukai transform is an isomorphism for indecomposable tori. In Section \ref{sec:indr} we consider an application of the Real Fourier--Mukai transform to the families index of Real Dirac operators.

\noindent{\bf Acknowledgments.} The author was financially supported by an Australian Research Council Future Fellowship, FT230100092.

\section{Graded Real gerbes}\label{sec:grg}

In this section we recall the basic definitions and properties of graded Real gerbes. We use the language of bundle gerbes, following \cite{hmsv}, see also \cite{mur, ste, bar2} for the case of complex bundle gerbes. Useful additional references are \cite{fht1} for graded complex gerbes and \cite{fm, gom} for the Real case.

Let $X$ be a topological space, always assumed to be locally contractible and paracompact. A {\em Real structure} on $X$ is a continuous involution $\sigma : X \to X$. By a {\em covering} of $X$ we mean a topological space $Y$ and a continuous map $\pi : Y \to X$ such that $\pi$ admits local sections. Suppose $\sigma_X$ is a Real structure on $X$. A {\em Real covering} is a covering $\pi : Y \to X$ and a Real structure $\sigma_Y$ on $Y$ such that $\pi \circ \sigma_Y = \sigma_X \circ \pi$. A {\em refinement} of a covering $\pi : Y \to X$ is a covering $\pi' : Y' \to X$ together with a continuous map $r : Y' \to Y$ such that $\pi' = \pi \circ r$. There is an obvious notion of Real refinements of Real coverings.

\begin{example}
If $\{ U_i \}_{i \in I}$ is an open cover of $X$ then $Y = \coprod_{i \in I} U_i$ with the obvious map $Y \to X$ is a covering. Let $\sigma$ be a Real structure on $X$. An {\em equivariant open cover} on $X$ is an open cover $\mathcal{U} = \{ U_i \}_{i \in I}$ together with an involution $\sigma : I \to I$ on the indexing set $I$ such that $\sigma(U_i) = U_{\sigma i}$ for all $i \in I$. Any equivariant open cover can be refined to one for which $\sigma$ acts freely on $I$. We will usually assume that our open covers are chosen to satisfy this property. We use the notation $U_{i_1 i_2 \cdots i_k} = U_{i_1} \cap \cdots \cap U_{i_k}$ for multiple intersections. Notice that $\sigma(U_{i_1 \cdots i_k}) = U_{\sigma(i_1) \cdots \sigma(i_k)}$.
\end{example}

Given a covering $\pi : Y \to X$, let $Y^{[p]}$ denote the $p$-fold fibre product $Y \times_X Y \times_X \cdots \times_X Y$. The collection $\{ Y^{[p]} \}_{p \ge 1}$ defines a simplicial space whose face maps $\pi_i : Y^{[p]} \to Y^{[p-1]}$ are given by omitting the $i$-th entry. If $\sigma$ is a Real structure on $X$ and $\pi : Y \to X$ is a Real covering, then each of the spaces $Y^{[p]}$ inherits a Real structure and the face maps respect the Real structures.

We use the term ``line bundle" to refer to unitary complex line bundles and the term ``isomorphism" of line bundles will refer to $\mathbb{C}$-linear unitary isomorphism. An ``antilinear isomorphism" means a $\mathbb{C}$-antilinear unitary isomorphism. A {\em Real structure} on a line bundle $L \to X$ is a lift $\sigma : L \to L$ of $\sigma$ to an involution on $L$ which is an antilinear involution. A Real line bundle is a line bundle equipped with a Real structure. Isomorphism classes of Real line bundles are classified by the equivariant cohomology group $H^2_{\mathbb{Z}_2}(X ; \mathbb{Z}_-)$, where $\mathbb{Z}_-$ denotes the equivariant local system which is just the trivial local system $\mathbb{Z}$, but where $\sigma$ acts as multiplication by $-1$. A {\em graded Real line bundle} is a Real line bundle $L \to X$ together with a continuously varying $\mathbb{Z}_2$-grading on the fibres of $L$ such that the Real structure is grading preserving. Such a grading is specified by a $\sigma$-invariant function $\epsilon : X \to \mathbb{Z}_2$, hence graded Real line bundles are classified by the group $H^0_{\mathbb{Z}_2}(X ; \mathbb{Z}_2) \oplus H^2_{\mathbb{Z}_2}(X ; \mathbb{Z}_-)$. We define the {\em graded Real Picard group} of $X$, $GrPic_R(X)$ to be the group of isomorphism class of graded Real line bundles on $X$, where the group operation is the graded tensor product. Hence
\[
GrPic_R(X) \cong H^0_{\mathbb{Z}_2}(X ; \mathbb{Z}_2) \oplus H^2_{\mathbb{Z}_2}(X  ; \mathbb{Z}_-).
\]

Let $\pi : Y \to X$ be a Real cover of $X$. A {\em graded Real gerbe} on $X$ (with respect to the cover $Y$) is a pair $\mathcal{G} = (L , \lambda)$, where $L \to Y^{[2]}$ is a graded Real line bundle and $\lambda$ is an isomorphism $\lambda : \pi_3^*(L) \otimes \pi_1^*(L) \to \pi_2^*(L)$ which is required to be associative in the sense that the following diagram commutes:
\[
\xymatrix{
L_{12} \otimes L_{23} \otimes L_{34} \ar[r]^-{\lambda_{123} } \ar[d]^-{\lambda_{234}} & L_{13} \otimes L_{34} \ar[d]^-{\lambda_{134}} \\
L_{12} \otimes L_{24} \ar[r]^-{\lambda_{124}} & L_{14}
}
\]
In this diagram $L_{ij}$ denotes the pullback of $L \to Y^{[2]}$ under the map $Y^{[4]} \to Y^{[2]}$ which sends $(y_1,y_2,y_3,y_4)$ to $(y_i,y_j)$ and $\lambda_{ijk}$ is defined as the pullback of $\lambda$ under the map $Y^{[4]} \to Y^{[3]}$ which sends $(y_1,y_2,y_3,y_4)$ to $(y_i,y_j,y_k)$.

\begin{remark}
In this definition we have chosen to describe gerbes in terms of line bundles. One could just as easily define gerbes in terms of principal circle bundles and sometimes it more convenient to do so. We will switch between the two points of view without further mention.
\end{remark}

A {\em strict isomorphism} $\varphi : \mathcal{G} \to \mathcal{G}'$ of graded Real bundle gerbes $\mathcal{G} = (L , \lambda)$, $\mathcal{G}' = (L' , \lambda')$ (both with respect to a cover $\pi : Y \to X$) is an isomorphism $\varphi : L \to L'$ of graded Real line bundles on $Y^{[2]}$ which intertwines $\lambda$ and $\lambda'$.

The dual of a graded Real bundle gerbe $\mathcal{G} = (L , \lambda)$ is given by $\mathcal{G}^* = (L^* , \lambda^{-1})$ where $L^*$ is the dual of $L$ and $\lambda^{-1}$ is the inverse of $\lambda$. Note that since $L$ is unitary, $L^* \cong L^{-1}$ and so $\lambda^{-1}$ has the form $\lambda^{-1} : L_{12}^{*} \otimes L_{23}^{*} \to L_{13}^{*}$.

Let $N \to Y$ be a graded Real line bundle on $Y$. Define a graded Real gerbe $\delta N = ( L , \lambda )$ where $L = \pi_2^*(N) \otimes \pi_1^*(N^*) = N_1 \otimes N_2^*$ and $\lambda : L_{12} \otimes L_{23} \to L_{13}$ is the map
\[
L_{12} \otimes L_{23} = N_1 \otimes N_2^* \otimes N_2 \otimes N_3^* \to N_1 \otimes N_3^* = L_{13}
\]
given by the dual pairing of $N_2$ and $N_2^*$.

Given graded Real gerbes $\mathcal{G} = (L , \lambda)$, $\mathcal{G}' = (L' , \lambda')$ (both with respect to $Y$) we define their tensor product $\mathcal{G} \otimes \mathcal{G}' = (L'' , \lambda'')$ as follows. We take $L''$ to be the graded tensor product $L'' = L \otimes L'$ and $\lambda''$ is defined by the composition
\[
L_{12} \otimes L'_{12} \otimes L_{23} \otimes L'_{23} \buildrel sw \over \longrightarrow L_{12} \otimes L_{23} \otimes L'_{12} \otimes L'_{23} \buildrel \lambda_{123} \otimes \lambda'_{123} \over \longrightarrow L_{13} \otimes L'_{13}
\]
where $sw$ denotes the graded swapping map 
\[
sw( a \otimes b \otimes c \otimes d ) = (-1)^{\epsilon'_{12}\epsilon_{23}} a \otimes c \otimes b \otimes d,
\]
which swaps $b$ and $c$ together with a sign factor $(-1)^{\epsilon'_{12}\epsilon_{23}}$. Here $\epsilon'_{12}$, $\epsilon_{23}$ denote the gradings of $L'_{12}$ and $L_{23}$.

Let $\mathcal{G}, \mathcal{G}'$ be graded Real gerbes (with respect to $L$). An {\em isomorphism} $(N ,\varphi ) : \mathcal{G} \to \mathcal{G}'$ consists of a graded Real line bundle $N \to Y$ and a strict isomorphism $\varphi : \mathcal{G} \otimes \delta N \to \mathcal{G}'$. Two isomorphisms $(N,\varphi) , (N' , \varphi') : \mathcal{G} \to \mathcal{G}'$ are said to be {\em equivalent} if there is an isomorphism $\psi : N \to N'$ of graded Real line bundles such that the diagram
\[
\xymatrix{
L_{12} \otimes N_1 \otimes N_2^* \ar[rr]^-{ \psi_1 \otimes (\psi_2)^* } \ar[drr]^-{\varphi_{12}} & & L_{12} \otimes N'_1 \otimes (N'_2)^* \ar[d]^-{\varphi'_{12}} \\
& & L'_{12}
}
\]
commutes. Define the category $GrGb_R(X ; Y)$ to be the category whose objects are graded Real bundle gerbes on $X$ defined with respect to the cover $Y$ and whose morphisms $GrGb_R(X ; Y)(\mathcal{G} , \mathcal{G}')$ are equivalence classes of isomorphisms $\mathcal{G} \to \mathcal{G}'$. More generally, we wish to have a notion of isomorphism of graded Real gerbes defined for different covers. We use a similar approach to \cite[\textsection 2.3]{fht1}. Let $\pi : Y \to X$, $\pi' : Y' \to X$ be two Real covers and $r : Y' \to Y$ a Real refinement. If $\mathcal{G} = (L , \lambda)$ is a graded Real gerbe defined with respect to $Y$, then we can define the pullback $r^*(\mathcal{G})$ of $\mathcal{G}$ with respect to $r$ to be $r^*(\mathcal{G}) = ( r^*(L) , r^*(\lambda))$. Pullback defines a functor $r^* : GrGb_R(X ; Y) \to GrGb_R(X ; Y)$. Now we define the category $GrGb_R(X)$ of graded Real gerbes on $X$ as follows. Objects are pairs $(\mathcal{G} , Y)$, where $Y \to X$ is a Real cover and $\mathcal{G}$ is a graded Real bundle gerbe on $X$ defined with respect to the cover $Y$. For objects $a = (\mathcal{G} , Y)$, $b = (\mathcal{G}' , Y')$, the set of morphisms is defined to be the direct limit
\[
GrGb_R(X)(a,b) = \lim\limits_{\underset{Y''\to Y \times_X Y'}{\longrightarrow}} GrGb_R(X ; Y'')( p_1^*(a) , p_2^*(b))
\]
where the direct limit is taken over all refinements $p : Y'' \to Y \times_X Y'$ and where $p_1 : Y'' \to Y$, $p_2 : Y'' \to Y'$ are the projections of $p$ to $Y$ and $Y'$.

Each graded Real gerbe $\mathcal{G}$ on $X$ has an invariant called the {\em (graded) Dixmier--Douady class}, as we now explain. Suppose $\mathcal{G} = (L , \lambda)$ is defined with respect to a Real open cover $\pi : Y \to X$. Since $\pi$ admits local sections, we can find an equivariant open cover $\{ U_i \}_{i \in I}$ and local sections $s_i : U_i \to Y$ such that $\sigma \circ s_i = s_{\sigma(i)} \circ \sigma$ (this can always be done since any open cover $\{ U_i \}$ can be replaced by its double $\{ U_i \} \cup \{ \sigma(U_i) \}$ ). Let $Z = \coprod_i U_i$ and let $s : Z \to Y$ be the map which equals $s_i$ on $U_i$. Then $s : Z \to Y$ is a refinement and we may consider the pullback gerbe $s^*(\mathcal{G}) = ( s^*(L) , s^*(\lambda))$. Refining the cover if necessary, we may assume that $(s_i , s_j)^*( L )$ is trivial on each $U_{ij}$. Choose a local section $t_{ij}$ of $(s_i , s_j)^*( L )$. We assume that $\sigma$ acts freely on the indexing set $I$ and so it is possible to choose the local sections $t_{ij}$ such that $t_{\sigma(i) \sigma(j)} \circ \sigma = \sigma \circ t_{ij}$. Then on $U_{ijk}$, we get an $S^1$-valued function $m_{ijk}$ defined by 
\[
(s^*_i , s^*_j , s^*_k)(\lambda)( t_{ij} , t_{jk} ) = m_{ijk} t_{ik}.
\]
It follows that $m_{ijk}$ is an equivariant \v{C}ech $2$-cocycle and so defines a class in $H^2_{\mathbb{Z}_2}(X ; \mathcal{C}_X(S^1)_-) \cong H^3_{\mathbb{Z}_2}(X ; \mathbb{Z}_-)$ (here $\mathcal{C}_X(S^1)$ denotes the sheaf of continuous $S^1$-valued functions and $ \mathcal{C}_X(S^1)_-$ denotes $ \mathcal{C}_X(S^1) \otimes_{\mathbb{Z}} \mathbb{Z}_-$). We also get a $\mathbb{Z}_2$-valued \v{C}ech $1$-cocycle $\epsilon_{ij}$ by taking $\epsilon_{ij}$ to be the grading of $(s_i^* , s_j^*)(L)$. One finds that the \v{C}ech cohomology classes of $\epsilon_{ij}$ and $m_{ijk}$ depend only on the isomorphism class of $\mathcal{G}$ and not on the choice of open cover $\{ U_i \}$ or local sections $s_i, t_{ij}$. The graded Dixmier--Douady class of $\mathcal{G}$ is then defined to be
\[
DD(\mathcal{G}) = ( [\epsilon_{ij}] , [m_{ijk}] ) \in H^1_{\mathbb{Z}_2}( X ; \mathbb{Z}_2) \times H^3_{\mathbb{Z}_2}(X ; \mathbb{Z}_-).
\]
We will sometimes refer to $[m_{ijk}]$ as the ungraded Dixmier--Douady class of $\mathcal{G}$. Let $\pi_0 GrGb_R(X)$ denote the set of isomorphism classes of objects in $GrGb_R(X)$. The graded Dixmier--Douady class defines a bijection of sets
\[
\pi_0 GrGb_R(X) \to H^1_{\mathbb{Z}_2}( X ; \mathbb{Z}_2) \times H^3_{\mathbb{Z}_2}(X ; \mathbb{Z}_-).
\]
The operation of graded tensor product makes $\pi_0 GrGb_R(X)$ into a group. Under the Dixmier--Douady class, the corresponding group operation on $H^1_{\mathbb{Z}_2}( X ; \mathbb{Z}_2) \times H^3_{\mathbb{Z}_2}(X ; \mathbb{Z}_-)$ is not the direct product, but is instead given by:
\[
( e , m ) \otimes (e' , m') = (e+e' , m + m' + \delta(e \smallsmile e'))
\]
See \cite[Corollary 2.25]{fht1} for the proof in the case of graded gerbes without Real structure. The cases of graded Real gerbes is a straightforward extension of this result. Thus it is better to think of $\pi_0 GrGb_R(X)$ as an extension of groups:
\[
0 \to H^3_{\mathbb{Z}_2}(X ; \mathbb{Z}_-) \to \pi_0 GrGb_R(X) \to H^1_{\mathbb{Z}_2}(X ; \mathbb{Z}_2) \to 0.
\]

Let $1$ denote the trivial graded Real gerbe on $X$ ($Y = X$, $L = \mathbb{C}$ with trivial grading and $\lambda = 1$). Then the set of morphisms $GrGb_R(X)(\mathcal{G} , \mathcal{G}')$ can be identified with $GrPic_R(X)$, the group of graded Real line bundles on $X$. Indeed, if $N \to X$ is any graded Real line bundle on $X$ then there is an evident strict isomorphism $\varphi : \delta N \to \mathbb{C}$ and the pair $(N , \varphi)$ defines a morphism $1 \to 1$. Up to equivalence every morphism $1 \to 1$ can be shown to be of this form. More generally, if $\mathcal{G} , \mathcal{G}'$ are isomorphic graded Real gerbes on $X$, then the set of morphisms $GrGb_R(X)(\mathcal{G} , \mathcal{G}')$ has the structure of a torsor for $GrPic_R(X)$. This follows because if $f_0 : \mathcal{G} \to \mathcal{G}'$ is a fixed choice of isomorphism, then any other isomorphism $f : \mathcal{G} \to \mathcal{G}'$ can be written as $f = f_0 \otimes \psi$ for a uniquely determined isomorphism $\psi : 1 \to 1$.

To each graded Real gerbe $\mathcal{G}$ on $X$, we can associate a real line bundle $det(\mathcal{G}) \to X$ as follows. Suppose $\mathcal{G} = (L , \lambda)$ is defined with respect to a cover $Y$. On $Y$ we take the trivial line bundle $\mathbb{R}_Y = Y \times \mathbb{R}$. On $Y^{[2]}$, the grading function $\epsilon : Y^{[2]} \to \mathbb{Z}_2$ can be regarded as an isomorphism $\epsilon : (\pi_2)^*( \mathbb{R}_Y) \to (\pi_1)^*( \mathbb{R}_Y )$. Since $\epsilon$ is a \v{C}ech cocycle, this isomorphism satisfies the descent condition and so $\mathbb{R}_Y$ descends to a line bundle on $X$, which we denote by $det(\mathcal{G})$. Since $\epsilon$ is $\sigma$ invariant, the involution $\sigma : \mathbb{R}_Y \to \mathbb{R}_Y$ given by $\sigma( y , r ) = (\sigma(y) , r)$ descends to an involution $\sigma : det(\mathcal{G}) \to det(\mathcal{G})$. Thus $det(\mathcal{G})$ is a $\mathbb{Z}_2$-equivariant real line bundle on $X$. Of course the isomorphism class of $det(\mathcal{G})$ is precisely the grading class $[\epsilon] \in H^1_{\mathbb{Z}_2}(X ; \mathbb{Z}_2)$ of $\mathcal{G}$.

Let $\mathcal{G}' = (L' , \lambda')$ be another graded Real gerbe (defined with respect to $Y$) and with grading function $\epsilon' : Y^{[2]} \to \mathbb{Z}_2$. Suppose that $(N , \varphi) : \mathcal{G} \to \mathcal{G}'$ is a morphism. Let $n : Y \to \mathbb{R}^*$ be given by $n = (-1)^{\epsilon(N)}$, where $\epsilon(N)$ is the grading on $N$. We can regard $n$ as a line bundle isomorphism $n : \mathbb{R}_Y \to \mathbb{R}_Y$. Then since $\varphi : L \otimes \delta N \to L'$ is grading preserving, it follows that $(-1)^{\epsilon} n = (-1)^{\epsilon'}$. This means that $n$ descends to a line bundle isomorphism $n : det(\mathcal{G}) \to det(\mathcal{G}')$. Moreover, since the grading on $N$ is $\sigma$-invariant, we have that $n : det(\mathcal{G}) \to det(\mathcal{G}')$ is an isomorphism of equivariant line bundles. Observe also that $n$ depends only on the equivalence class of the isomorphism $(N , \varphi)$. In particular, a trivialisation of $\mathcal{G}$ determines a trivialisation of $det(\mathcal{G})$ in a canonical way.

\subsection{Lifting gerbes}\label{sec:lg}

In this section we consider an important special class of graded Real gerbes known as lifting gerbes. Let $G$ be a Lie group. Let $\epsilon : G \to \mathbb{Z}_2$ be a homomorphism and let
\[
1 \to S^1 \to \widetilde{G} \buildrel \pi \over \longrightarrow G \to 1
\]
be a central extension of $G$ by $S^1$. Define $\widetilde{\epsilon} : \widetilde{G} \to \mathbb{Z}_2$ to be the composition $\widetilde{\epsilon} = \pi \circ \epsilon$. Let $c : \widetilde{G} \to \widetilde{G}$ by an involutive automorphism on $\widetilde{G}$ which covers the identity on $G$ and satisfying $c(z) = \overline{z}$ for all $z \in S^1$.

Let $X$ be a topological space with a Real structure $\sigma$. Let $\rho : P \to X$ be a principal $G$-bundle on $X$ and suppose that $P$ is equipped with a Real structure, by which we mean a lift of $\sigma$ to an involution on $P$ that satisfies $\sigma(pg) = \sigma(p)g$ for all $p \in P$, $g \in G$. The {\em lifting gerbe of $P$} is defined as follows. We take the covering $Y \to X$ to be given by $Y = P$. Then 
\[
Y^{[2]} = \{ (p_1 , p_2) \; | p_2 = p_1 g \text{ for some } g \in G \} \cong Y \times G
\]
where the isomorphism $Y \times G \to Y^{[2]}$ is given by $(p,g) \mapsto (p,pg)$. Over $Y^{[2]}$ we have a Real circle bundle $\mathcal{C} \to Y^{[2]}$ given by $\mathcal{C} = Y \times \widetilde{G}$. The projection $\mathcal{C} \to Y^{[2]}$ is given by $( p , h ) \mapsto ( p , p \pi(h) )$. The Real structure on $\mathcal{C}$ is given by $\sigma( p , h ) = (\sigma(p) , c(h) )$ and the grading $\epsilon : \mathcal{C} \to \mathbb{Z}_2$ is given by $\epsilon(p,h) = \widetilde{\epsilon}(h)$. The multiplication $\lambda : \mathcal{C}_{12} \times \mathcal{C}_{23} \to \mathcal{C}_{13}$ is given by
\[
\lambda( (p_1 , h_1 ) , (p_2 , h_2 ) ) = ( p_1 , h_1 h_2 ).
\]

It is easily checked that this data defines a graded Real gerbe on $X$. We denote it by $L(P)$. By the definition of the grading function, it follows that the determinant line $det(L(P))$ is isomorphic to the line bundle $P \times_G \mathbb{R}$, where $G$ acts on $\mathbb{R}$ via $\epsilon : G \to \mathbb{Z}_2$.

\begin{remark}\label{rem:redstr}
If $H \to G$ is a subgroup, then we can restrict the homomorphism $\epsilon$ and the central extension $\widetilde{G} \to G$ to $H$. If $r : P' \to P$ is a reduction of structure from $P$ to $P'$ (compatible with Real structures), then $r$ induces an isomorphism of lifting gerbes by sending $(p' , h') \in P' \times \widetilde{H}$ (where $\widetilde{H} = \pi^{-1}(H)$) to $(r(p') , h ) \in P \times \widetilde{G}$.
\end{remark}

We will be interested in the particular case of this construction where $G = O(n)$ is an orthogonal group, $\epsilon : O(n) \to \mathbb{Z}_2$ is the determinant homomorphism and the central extension is
\[
1 \to S^1 \to Pin^c(n) \to O(n) \to 1.
\]
Let $V \to X$ be a rank $n$ orthogonal vector bundle. Suppose that $V$ has a Real structure in the sense that there is a lift of $\sigma$ to an orthogonal involution on $V$. Let $P_V \to X$ be the $O(n)$-frame bundle of $V$. Then the Real structure on $V$ induces a Real structure on $P_V$. In this case, we denote the lifting gerbe of $P_V$ by $L(V)$ and call it {\em the lifting gerbe of $V$}. The isomorphism class of $L(V)$ measures the failure of $V$ to admit a Real spin$^c$-structure. Observe that the determinant line of $L(V)$ is canonically isomorphic to $det(V)$, the determinant line of $V$.

Let $V,W$ be rank $n$ orthogonal vector bundles and let $f : V \to W$ be a vector bundle isomorphism. Then $f$ induces a principal bundle isomorphism $f : P_V \to P_W$, where $P_V,P_W$ are the frame bundles of $V,W$. We also have that $f$ induces an isomorphism of lifting gerbes as follows. Let $Y = P_V \times_X P_W$ and let $p_V : Y \to P_V$, $p_W : Y \to P_W$ be the projections. Write $L(V) = ( C_V , \lambda_V)$, $L(W) = (C_W , \lambda_W)$ where 
\[
C_V = \{ (v_1,v_2,h) \in P_V \times_X P_V \times Pin^c(n) \; | \; v_2 = v_1 \pi(h) \},
\]
\[
\lambda_V( (v_1,v_2,h_1) , (v_2,v_3 , h_2) ) = (v_1 , v_3 , h_1 h_2)
\]
and $C_W, \lambda_V$ are defined similarly. We will construct an isomorphism $(N , \varphi) : p_V^*(L(V)) \to p_W^*(L(W))$ as follows. We will take $N$ to be the trivial Real line bundle but with grading $n : Y \to \mathbb{Z}_2$ given by $n( v , w ) = \epsilon(t)$, where $t \in O(n)$ is given by $f(v) = wt$. The map $\varphi : p_V^*(C_V) \otimes \delta N \to p_W^*(C_W)$ is defined as follows. First note that
\[
p_V^*(C_V) = \{ (v_1,v_2,w_1,w_2,h) \in P_V^{[2]} \times_X P_W^{[2]} \times Pin^c(n) \; | \; v_2 = v_1 \pi(h) \}
\]
and similarly
\[
p_W^*(C_W) = \{ (v_1,v_2,w_1,w_2,h) \in P_V^{[2]} \times_X P_W^{[2]} \times Pin^c(n) \; | \; w_2 = w_1 \pi(h) \}.
\]
Then we define $\varphi(v_1,v_2,w_1,w_2,h) = (v_1,v_2,w_1,w_2, t_1 h t_2^{-1})$, where $t_1,t_2 \in O(n)$ are defined by $f(v_i) = w_i t_i$ for $i=1,2$. It is easily checked that this defines an isomorphism of gerbes $p_V^*( L(V)) \to p_W^*( L(W) )$ defined over $Y = P_V \times_X P_W$, and hence an isomorphism of gerbes $L(V) \to L(W)$. Moreover, the isomorphism $det(L(f)) : det( L(V) ) \to det( L(W) )$ of determinant lines is precisely the map given by the determinant of $f$: $det(f) : det(V) \to det(W)$.

Suppose $P \to X$ is a principal $G$-bundle with a Real structure. Suppose that $\widetilde{\pi} : \widetilde{P} \to P$ is a lift of $P$ to a principal $\widetilde{G}$-bundle and suppose $\epsilon : \widetilde{P} \to \mathbb{Z}_2$ is a continuous function such that $\epsilon( q h) = \epsilon(q) + \widetilde{\epsilon}(h)$ for all $q \in \widetilde{P}$ and $h \in \widetilde{G}$. Suppose also that $\sigma$ lifts to an involution on $\widetilde{P}$ which is Real in the sense that $\sigma(qh) = \sigma(q) c(h)$ for all $q \in \widetilde{P}$, $h \in \widetilde{H}$. We claim that this data determines a canonical trivialisation of the lifting gerbe of $P$. To see this, first note that $\widetilde{P} \to P$ is a graded Real line bundle over $Y = P$, where the grading is given by $\epsilon : \widetilde{P} \to \mathbb{Z}_2$. One checks that $(\delta \widetilde{P})$ is given by the set of pairs $(q_1,q_2)$ such that $\widetilde{\pi}(q_2) = \widetilde{\pi}(q_1)g$ for some $g \in G$, modulo the equivalence relation $(q_1,q_2) \sim (q_1 z , q_2 z)$, where $z \in S^1$. But $\widetilde{\pi}(q_2) = \widetilde{\pi}(q_1)g$ for some $g \in G$ if and only if $q_2 = q_1 h$ for some $h \in \widetilde{H}$. It follows that $\delta \widetilde{P}$ can be identified with the set of pairs $(q_1 , h) \in \widetilde{P} \times \widetilde{H}$ modulo the equivalence relation $(q_1,h) \sim (q_1 z , h)$, where $z \in S^1$. This shows that there is a canonical isomorphism $\varphi : \delta \widetilde{P} \cong P \times \widetilde{G} = \mathcal{C}$. Explicitly, $\varphi(q_1,q_2) = ( \widetilde{q_1} , h)$, where $q_2 = q_1 h$. The Real structure on $\delta \widetilde{P}$ is $\sigma(q_1 , q_2) = (\sigma(q_1) , \sigma(q_2)$. Then $\varphi ( \sigma(q_1 , q_2) ) = \varphi( \sigma(q_1) , \sigma(q_2) ) = ( \widetilde{\pi}(\sigma(q_1)) , c(h) ) = (\sigma( \widetilde{\pi}(q_1) ) , c(h) ) = \sigma( \varphi( q_1 , q_2) )$. So $\varphi$ respects Real structures. The grading on $\delta \widetilde{P}$ is given by $\epsilon(q_1,q_2) = \epsilon(q_1) + \epsilon(q_2)$. But if $q_2 = q_1 h$, then $\epsilon(q_1) + \epsilon(q_2) = \epsilon(q_1) + \epsilon(q_1) + \widetilde{\epsilon}(h) = \widetilde{\epsilon}(h)$. Thus $\varphi : \delta \widetilde{P} \to \mathcal{C}$ is an isomorphism of graded Real line bundles. Lastly, it is easily checked that $\varphi$ respects multiplication and so $\varphi$ is a strict isomorphism from the trivial gerbe $\delta \widetilde{P}$ to the lifting gerbe $L(P) = (\mathcal{C} , \lambda)$. Thus the pair $( \widetilde{P} , \varphi)$ defines a trivialisation $(\widetilde{P} , \varphi ) : 1 \to L(P)$, as claimed. The trivialisation $(\widetilde{P} , \varphi) : 1 \to L(P)$ induces a trivialisation $1 \to det( L(P) )$ of the determinant line. Recall that $det(L(P))$ can be identified with the line bundle $P \times_G \mathbb{R}$, where $G$ acts on $\mathbb{R}$ through $\epsilon : G \to \mathbb{Z}_2$. This trivialisation of $det(L(P))$ is the trivialisation defined by the grading $\epsilon : \widetilde{P} \to \mathbb{Z}_2$ (note first that $\epsilon : \widetilde{P} \to \mathbb{Z}_2$ descends to a map $\epsilon : P \to \mathbb{Z}_2$. Viewing $\mathbb{Z}_2 = \{ \pm 1\} \subset \mathbb{R}^*$, the map $\epsilon : P \to \mathbb{Z}_2$ defines a non-vanishing section of $P \times_G \mathbb{R}$).

Let $V$ be an $O(n)$-vector bundle with a Real structure. Let $P_V$ be the principal $O(n)$-frame bundle. In the case, a pair $( \widetilde{P} , \epsilon )$ where $\widetilde{P}$ is a lift of $P_V$ to $Pin^c(n)$ and $\epsilon : \widetilde{P} \to \mathbb{Z}_2$ satisfies $\epsilon(qh) = \epsilon(q) + \widetilde{\epsilon}(h)$. This data is (by definition) a Real spin$^c$-structure on $V$. Thus we have  a canonical bijection between Real spin$^c$-structures on $V$ and trivialisations of the lifting gerbe $L(V)$.

If $V = V_1 \oplus V_2$ is a direct sum of orthogonal vector bundles, then there is a canonical isomorphism $L(V_1 \oplus V_2) \cong L(V_1) \otimes L(V_2)$. To see this, let $P_1,P_2$ and $P$ be the orthogonal frame bundles of $V_1,V_2$ and $V$. From $V = V_1 \times V_2$ we get a reduction of structure map $r : P_1 \times_X P_2 \to P$. This is a reduction of structure from $G = O(n)$ to $H = O(n_1) \times O(n_2)$, where $n_i$ is the rank of $V_i$. Then as seen in Remark \ref{rem:redstr}, the lifting gerbe of $P$ is canonically isomorphic to the lifting gerbe of $P_1 \times_X P_2$. Set $Y = P_1 \times_X P_2$. Then the lifting gerbe for $P_1 \times_X P_2$ is given by $\mathcal{C} = Y \times \widetilde{H}$, where $\widetilde{H} = \{ h \in Pin^c(n) \; | \; \pi(h) \in O(n_1) \times O(n_2) \}$. As $S^1$-bundles, $\widetilde{H} \cong Pin^c(n_1) \times_{S^1} Pin^c(n_2)$. Moreover the group structure is given by
\[
(h_1 , h_2) ( h'_1 , h'_2) = (-1)^{\widetilde{\epsilon}(h'_1) \widetilde{\epsilon}(h_2)} ( h_1 h'_1 , h_2 h'_2).
\]
The sign factor $(-1)^{\widetilde{\epsilon}(h'_1) \widetilde{\epsilon}(h_2)}$ comes from the Clifford algebra relations $vw = -wv$ if $v \in \mathbb{R}^{n_1} \oplus 0$, $w \in 0 \oplus \mathbb{R}^{n_2}$. This means that as graded $S^1$ bundles over $O(n_1) \times O(n_2)$ we have $\widetilde{H}$ is the {\em graded} tensor product $Pin^c(n_1) \otimes Pin^c(n_2)$. From this, it follows easily that the lifting gerbe for $P_1 \times_X P_2$ is the graded tensor product of the lifting gerbes for $P_1$ and $P_2$, hence $L(V_1 \oplus V_2) \cong L(V_1) \otimes L(V_2)$.

In the isomorphism $L(V_1 \oplus V_2) \cong L(V_1) \otimes L(V_2)$, the order of the summands $V_1$, $V_2$ matters: the composition
\[
L(V_1) \otimes L(V_2) \cong L(V_1 \oplus V_2) \cong L(V_2 \oplus V_1) \cong L(V_2) \otimes L(V_1) \cong L(V_1) \otimes L(V_2)
\]
is an automorphism of $L(V_1) \otimes L(V_2)$ and hence corresponds to a graded Real line bundle on $X$. This graded line bundle is the trivial Real line bundle, but with grading equal to $(-1)^{dim(V_1)dim(V_2)}$. This corresponds to the fact that the composition
\[
det(V_1) \otimes det(V_2) \cong det(V_1 \oplus V_2) \cong det(V_2 \oplus V_1) \cong det(V_2) \otimes det(V_1) \cong det(V_1) \otimes det(V_2)
\]
is multiplication by $(-1)^{dim(V_1)dim(V_2)}$.

\subsection{Graded Real gerbes over a point}\label{sec:ggbpt}

We consider graded Real gerbes over point. In this case $X = \{pt\}$ and the group is isomorphism classes of graded Real gerbes is an extension
\[
0 \to H^3_{\mathbb{Z}_2}( pt ; \mathbb{Z}_-) \to \pi_0 GrGb_R(pt) \to H^1_{\mathbb{Z}_2}(pt ; \mathbb{Z}_2) \to 0.
\]
Since $H^3_{\mathbb{Z}_2}(pt ; \mathbb{Z}_-) \cong H^1_{\mathbb{Z}_2}(pt ; \mathbb{Z}_2) \cong \mathbb{Z}_2$, this is an extension of $\mathbb{Z}_2$ by $\mathbb{Z}_2$. In particular, there are four isomorphism classes. Up to isomorphism, any graded Real gerbe on $pt$ can be represented using the covering $Y = \{0,1\}$ with $\sigma$ the involution that swaps $0$ and $1$. Let $\mathcal{G} = (L , \lambda)$ be a graded Real gerbe on $X$ defined with respect to the covering $Y$. Then $L = L_{00} \cup L_{01} \cup L_{10} \cup L_{11}$ where $L_{ij} = L|_{ \{i \} \times \{j\}}$. Choose non-vanishing sections $t_{ij} \in L_{ij}$ such that $\sigma(t_{ij}) = t_{\sigma(i) \sigma(j)}$. We can use $\{ t_{ij} \}$ identify $L$ with the trivial Real line bundle $\mathbb{C}$, but equipped with a possibly non-trivial grading $\epsilon_{ij} \in \mathbb{Z}_2$. Since $\sigma : L_{ij} \to L_{ \sigma(i) \sigma(j)}$ is grading preserving, we must have $\epsilon_{00} = \epsilon_{11}$ and $\epsilon_{01} = \epsilon_{10}$. The gerbe multiplication $\lambda$ is a collection of isomorphisms $\lambda_{ijk} : L_{ij} \otimes L_{jk} \to L_{ik}$. Since $\lambda_{ijk}$ is required to be grading preserving, we must have that $\epsilon_{ij} + \epsilon_{jk} = \epsilon_{ik}$. This implies that $\epsilon_{00} = \epsilon_{11} = 0$, so $\epsilon_{ij}$ is completely determined by $\epsilon_{01} \in \mathbb{Z}_2$. Set $m_{ijk} \in S^1$ to be given by $\lambda_{ijk}(1,1) = m_{ijk}$. Then $m_{ijk}$ is an equivariant \v{C}ech $2$-cocycle, that is, $m_{\sigma(i) \sigma(j) \sigma(k)} = \overline{m_{ijk}}$ and $m_{jkl} m_{ikl}^{-1} m_{ijl} m_{ijk}^{-1} = 1$. Replacing $t_{00}$ by $m_{000}^{-1}t_{00}$ and $t_{11}$ by $m_{111}^{-1}t_{11}$, we can assume that $\lambda_{iii}(t_{ii} , t_{ii}) = t_{ii}$ for $i=0,1$ and hence $m_{000} = m_{111} = 1$. From the cocycle condition we also get $m_{001} = m_{110} = m_{100} = m_{011} = 1$. Since $m_{010} = \overline{m_{101}}$, this means $m_{ijk}$ is completely determined by $m_{010} \in S^1$. Furthermore the $2$-cocycle condition implies that $m_{010}^2 = 1$, so $m_{010} = \pm 1$. Write $m_{010} = (-1)^{\mu_{010}}$ where $\mu_{010} \in \mathbb{Z}_2 = \{0,1\}$. The graded Dixmier--Douady class of $\mathcal{G}$ can be identified with the pair $( \epsilon_{01} , \mu_{010} ) \in \mathbb{Z}_2^2$. The group law is easily seen to be given by
\[
(e_1 , \mu_1) (e_2 , \mu_2) = (e_1 + e_2 , \mu_1 + \mu_2 + e_1 e_2).
\]
This is a non-trivial extension of $\mathbb{Z}_2$ by $\mathbb{Z}_2$, hence $\pi_0 Gr Gb_R(pt)$ is isomorphic to a cyclic group of order $4$.

Let $\mathbb{R}_-$ denote the $1$-dimensional real representation of $\mathbb{Z}_2 = \langle \sigma \rangle$ where $\sigma$ acts as multiplication by $-1$. We can view $\mathbb{R}_-$ as an equivariant vector bundle over $pt$. The point gerbes can be identified with the lifting gerbes $L( \mathbb{R}^k_-)$ as follows. Since $L(\mathbb{R}^k_-) \cong L(\mathbb{R}_-)^k$, it suffices to show that $L(\mathbb{R}_-)$ is a generator for $\pi_0 Gr Gb_R(pt)$. The frame bundle of $\mathbb{R}_-$ is given by a two-element set $Y = \{0,1\}$ and $\sigma : Y \to Y$ acts by swapping $0$ and $1$. Here we regard $i \in Y$ as corresponding to the frame $(-1)^i \in \mathbb{R}_-$. The group $Pin^c(1)$ equals $S^1 \cup S^1 e$ where $e^2 = - 1$, $\epsilon(e) = 1 \in \mathbb{Z}/2\mathbb{Z}$. The line bundle $C \to Y^{[2]} = \{ 00 , 01 , 10 , 11 \}$ is given by $C = C_{00} \cup C_{01} \cup C_{10} \cup C_{11}$ where $C_{ij} = S^1$ if $i=j$ and $C_{ij} = S^1 e$ if $i \neq j$. The product $\lambda_{ijk} : C_{ij} \times C_{jk} \to C_{ik}$ is given by group multiplication in $Pin^c(1)$. Let $t_{ij} \in C_{ij}$ be given by $t_{ij} = 1$ if $i = j$, $t_{ij} = e$ if $i \neq j$. Then $\lambda_{ijk}(t_{ij} , t_{jk}) = m_{ijk} t_{ik}$, where $m_{010} = m_{101} = -1$, $m_{ijk} = 1$ otherwise. The grading is given by $e_{ij} = \epsilon(t_{ij})$. Thus $e_{ij} = 0$ if $i=j$ and $e_{ij} = 1$ if $i \neq j$. This shows that the graded Dixmier--Douady class of $L(\mathbb{R}_-)$ equals $( e_{ij} , m_{ijk} ) \in H^1_{\mathbb{Z}_2}(pt ; \mathbb{Z}_2) \oplus H^3_{\mathbb{Z}_2}(pt ; \mathbb{Z}_-)$, where $e_{ij}, m_{ijk}$ are the generators of $H^1_{\mathbb{Z}_2}(pt ; \mathbb{Z}_2) \cong H^3_{\mathbb{Z}_2}(pt ; \mathbb{Z}_-) \cong \mathbb{Z}_2$. In particular, $L(\mathbb{R}_-)$ is a generator of $\pi_0 Gr Gb_R(pt)$, as claimed.

\section{Twisted $KR$-theory}\label{sec:tkr}

In this section we will recall some basic properties of pushforward maps in twisted $KR$-theory that will be needed for the Real Fourier--Mukai transforms. For more details on twisted $K$-theory we refer the reader to \cite{fht1} and \cite{as} for complex twisted $K$-theory and \cite{hmsv,fok,fm,gom} for twisted $KR$-theory.

Let $X$ be a topological space with a Real structure $\sigma$ and let $\mathcal{G}$ be a graded Real gerbe on $X$. The twisted $KR$-theory of $(X , \mathcal{G})$ will be denoted by $KR^*(X,\mathcal{G})$. If $A \subseteq X$ is a subspace then the relative twisted $KR$-theory of $(X,A,\mathcal{G})$ will be denoted by $KR^*(X,A,\mathcal{G})$.

Let $X,Y$ be compact smooth manifolds equipped with Real structures $\sigma_X, \sigma_Y$ and let $\mathcal{G}_Y$ be a graded Real gerbe on $Y$. Then to any continuous map $f : X \to Y$ we have a pushforward morphism \cite{hmsv}
\[
f_* : KR^*(X , L(TX) \otimes f^*(\mathcal{G}_Y) ) \to KR^{*-dim(X)+dim(Y)}(Y , L(TY) \otimes \mathcal{G}_Y)
\]
defined so that the following square commutes
\[
\xymatrix{
KR^*(X , L(TX) \otimes f^*(\mathcal{G}_Y)) \ar[r]^-{f_*} \ar[d]^-{PD_X} & KR^{*-dim(X)+dim(Y)}(Y , L(TY) \otimes \mathcal{G}_Y) \ar[d]^-{PD_Y} \\
KR_{dim(X)-*}(X , f^*(\mathcal{G}_Y) ) \ar[r]^-{f_*} & KR^{dim(X)-*}(Y , \mathcal{G}_Y)
}
\]
where the vertical maps are Poincar\'e duality isomorphisms. More generally, if $X$ and $Y$ are smooth manifolds which are not necessarily compact and $f : X \to Y$ is a continuous map, we still have a pushforward morphism 
\[
f_* : KR^*_c(X , L(TX) \otimes f^*(\mathcal{G}_Y) ) \to KR^{*-dim(X)+dim(Y)}_c(Y , L(TY) \otimes \mathcal{G}_Y)
\]
provided we use compactly supported $KR$-theory: $KR^*_c(M , \mathcal{G}) = KR^*( M^+ , + , \mathcal{G})$, where $M^+$ denotes the one-point compactification of $M$. If $f : X \to Y$ is a fibre bundle, then $TX \cong f^*(TY) \oplus V$ where $V$ is the vertical tangent bundle. Hence $L(TX) \cong f^*( L(TY) ) \oplus L(V)$ and in such cases the pushforward takes the form
\[
f_* : KR^*(X , L(V) \otimes f^*(\mathcal{G}_Y) ) \to KR^{*-m}(Y , \mathcal{G}_Y)
\]
where $m = dim(X) - dim(Y)$ is the dimension of the fibres.

The main properties of pushforward morphisms we need are:
\begin{itemize}
\item[(1)]{Composition: $(f \circ g)_* = f_* \circ g_*$.}
\item[(2)]{Projection formula: $f_*(  f^*(y)x ) = y f_*(x)$.}
\item[(3)]{Thom isomorphism: let $\sigma : X \to X$ be an involution and let $\pi : V \to X$ be a Real vector bundle over $X$. That is, $V$ is a complex vector bundle together with an antilinear lift of $\sigma$ to $V$. Then there exists a class $\tau_V \in KR^m( D(V) , S(V) , \pi^*(L(V)) )$ (where $m$ is the rank of $V$, $D(V)$ is the closed unit disc bundle, $S(V)$ is the unit sphere bundle) called the {\em Thom class of $V$} with the property that for any graded Real gerbe $\mathcal{G}$ on $X$, the map $KR^*(X , \mathcal{G}) \to KR^{*+m}( D(V) , S(V) , \pi^*( \mathcal{G} \otimes L(V) )$ given by $\alpha \mapsto \pi^*(\alpha) \tau_V$ is an isomorphism. Furthermore, we have that $\pi_*( \tau_V) = 1$.}
\item[(4)]{Embeddings: if $f : X \to Y$ is an embedding then $f_*$ is the composition
\begin{align*}
KR^*_c( X , L(TX) \otimes \mathcal{G}_Y|_X) &\to KR_c^{*-dim(X)+dim(Y)}( \nu X , L(TY) \otimes \mathcal{G}_Y|_{\nu X}) \\
& \to KR_c^{*-dim(X)+dim(Y)}(Y , L(TY) \otimes \mathcal{G}_Y)
\end{align*}
where the first map is the Thom isomorphism, $\nu$ is an open tubular neighbourhood of $X$ in $Y$ and the second map is extension by zero (i.e. pullback under the map $Y^+ \to (\nu X)^+$ which sends everything outside of $\nu X$ to infinity). In particular, for open embeddings $f_*$ is just extension by zero.}
\item[(5)]{Base change: let $X,Y,E$ be compact smooth manifolds. Let $\pi : E \to X$ be a fibre bundle over $Y$ with compact fibres and let $f : X \to Y$ be a proper map. Let $\pi_X : E_X \to X$ be the pullback of $E$ along $f$ bundle giving a commutative diagram
\[
\xymatrix{
E_X \ar[r]^-{\widetilde{f}} \ar[d]^-{\pi_X} & E \ar[d]^-{\pi} \\
X \ar[r]^-{f} & Y
}
\]
Assume $X,Y,E$ are given Real structures which are respected by $\pi$ and $f$. Then $E_X$ inherits a Real structure. Let $\mathcal{G}$ be a graded Real gerbe on $Y$. Then the base change formula is that the following square commutes
\[
\xymatrix{
KR^*_c( E , L(V) \otimes \pi^*(\mathcal{G})) \ar[d]^-{\pi_*} \ar[r]^-{\widetilde{f}^*} & KR^*_c( E_X , \widetilde{f}^*(L(V)) \otimes \pi_X^* f^* (\mathcal{G})) \ar[d]^-{(\pi_X)_*} \\
KR^{*-m}_c( Y , \mathcal{G} ) \ar[r]^-{f^*} & KR^{*-m}_c( X , f^{*}(\mathcal{G}))
}
\]
where the fibres of $E \to Y$ are $m$-dimensional and $V$ is the vertical tangent bundle of $E \to Y$.}
\end{itemize}

We also need that the pushforward commutes with the Mayer--Vietoris sequence in an appropriate sense. In general this seems difficult to verify since it is not clear that the pushforward will commute with the coboundary maps. However we can show that this does hold for a special class of Mayer--Vietoris sequences. 

Let $X$ be a smooth manifold. Suppose that $X = X_- \cup_Y X_+$ where $X_+,X_-$ are smooth manifolds with boundary ($Y$ is regarded as an outgoing boundary of $X_-$ and an ingoing boundary of $X_+$) and assume that the inclusion $Y \to X$ is proper. Assume that $X_-$ contains a collar neighbourhood $(-2 , 0] \times Y$ and $X_+$ contains a collar neighbourhood $[0 , 2) \times Y$. Let $U = X_- \cup_Y [0,1) \times Y$, $V = (-1,0] \times Y \cup_Y X_+$. Then $U,V$ are open subsets covering $X$ and $U \cap V = (-1,1) \times Y$.

Note that $X_- - Y$ is diffeomorphic to $U$. In fact we can choose a diffeomorphism $f : X_- - Y \to U$ which is the identity outside of $(-2,0) \times Y$ and which sends $(-2,0) \times Y$ to $(-2,1) \times Y$ via a map of the form $(t,y) \mapsto ( \varphi(t) , y)$, where $\varphi : (-2,0) \to (-2,1)$ is a diffeomorphism which is the identity near $-2$. Similarly $X_+ - Y$ is diffeomorphic to $V$.

Assume $X$ has a Real structure $\sigma$. Assume that $X_+,X_-$ and $Y$ are preserved by $\sigma$. It follows that the collar neighbourhood $(-2,0] \times Y$ of $X_-$ can be chosen so that $\sigma$ is given in this neighbourhood by $\sigma(t,y) = (t , \sigma(y))$. The same is true for the collar neighbourhood $[0 , 2) \times Y$ of $Y$ in $X_+$.

Let $i : U \cap V \to U$, $j : U \cap V \to V$, $k : U \to X$, $l : V \to X$ be the inclusions. Let $i_Y : Y \to X$ be the inclusion of $Y$ in $X$ and let $\iota : Y \to U \cap V$ be the inclusion of $Y$ in $U \cap V$. Let $\mathcal{G}$ be a graded Real gerbe on $X$.

\begin{proposition}\label{prop:mv}
We have a long exact sequence
\begin{small}
\[
\xymatrix@C-1pc{
\cdots \ar[r] &  KR^a_c( U \cap V , \mathcal{G}) \ar[rr]^-{(i_* , -j_*)} & &{\begin{matrix} KR^a_c(U , \mathcal{G}) \\ \oplus \\ KR^a_c(V , \mathcal{G}) \end{matrix}} \ar[rr]^-{k_* + l_*} & & KR^a_c(X , \mathcal{G}) \ar[r]^-{\delta} & KR^{a+1}_c( U \cap V , \mathcal{G}) \ar[r] & \cdots
}
\]
\end{small}
where $\delta = \iota_* \circ i_Y^*$.
\end{proposition}
\begin{proof}
Consider the following diagram:
\begin{small}
\[
\xymatrix@C-1pc{
\cdots \ar[r] & KR^a_c( U \cap V , \mathcal{G}) \ar[r]^-{(i_*,-j_*)} & {\begin{matrix} KR^a_c(U, \mathcal{G}) \\ \oplus \\ KR^a_c(V, \mathcal{G}) \end{matrix}} \ar[r]^-{k_* + l_*} & KR^a(X, \mathcal{G}) \ar[r]^-{\delta} & KR^{a+1}_c(U \cap V) \ar[r] & \cdots \\
\cdots \ar[r] & KR^{a-1}_c(Y,\mathcal{G}) \ar[r]^-{\delta_{X,Y}} \ar[u]^-{\iota_*} & KR^a_c(X - Y , \mathcal{G}) \ar[u]^-{\psi} \ar[r] & KR^a_c(X , \mathcal{G}) \ar[r]^-{i_Y^*} \ar@{=}[u] & KR^a_c(Y , \mathcal{G}) \ar[r] \ar[u]^-{\iota_*} & \cdots
}
\]
\end{small}
In this diagram the bottom row is the long exact sequence of the pair $(X^+,Y^+)$ and the map $\psi$ is the composition
\[
KR^a_c(X - Y, \mathcal{G}) \cong \begin{matrix}KR^a_c(X_-  - Y , \mathcal{G}) \\ \oplus \\ KR^a_c(X_+ - Y , \mathcal{G}) \end{matrix}\cong \begin{matrix} KR^a_c(U , \mathcal{G}) \\ \oplus \\ KR^a_c(V , \mathcal{G}) \end{matrix}
\]
where the first isomorphism follows from $X-Y = (X_- - Y) \cup (X_+ - Y)$ and the second isomorphism uses the diffeomorphisms $X_- - Y \cong U$, $X_+ - Y \cong V$. We claim that the above diagram is commutative, which implies the top row is exact (the vertical maps are all isomorphisms). Commutativity of the rightmost square is immediate from the definition $\delta = \iota_* \circ i_Y^*$. Commutativity of the middle square follows easily from the definition of $\psi$. It remains to show commutativity of the leftmost square. This requires an understanding of the coboundary map $\delta_{X,Y}$.

Since $KR^a_c(X - Y,\mathcal{G}) \cong KR^a_c(X_- - Y,\mathcal{G}) \oplus KR^a_c(X_+ - Y , \mathcal{G})$, it suffices to determine the compositions $KR^{a-1}(Y,\mathcal{G}) \buildrel \delta_{X,Y} \over \longrightarrow KR^a(X - Y,\mathcal{G}) \to KR^a(X_{\pm} - Y , \mathcal{G})$, which we denote by $\delta_{X_{\pm} , Y}$. Compatibility of the long exact sequences for the inclusion of pairs $(X_{\pm}^+ , Y^+) \to (X^+,Y^+)$ implies that $\delta_{X_{\pm} , Y}$ is the coboundary map for the pair $(X_{\pm}^+ , Y^+)$. Commutativity of the leftmost square is then equivalent to commutativity of the following two diagrams:
\[
\xymatrix{
KR^{a-1}_c( U \cap V , \mathcal{G}) \ar[r]^-{i_*} & KR^a_c(U , \mathcal{G}) \\
KR^{a-1}_c(Y,\mathcal{G}) \ar[u]^-{\iota_*} \ar[r]^-{\delta_{X_-,Y}} & KR^a(X_- - Y , \mathcal{G}) \ar[u]^-{\cong}
}
\]
and
\[
\xymatrix{
KR^{a-1}_c( U \cap V , \mathcal{G}) \ar[r]^-{-j_*} & KR^a_c(V,\mathcal{G}) \\
KR^{a-1}_c(Y,\mathcal{G}) \ar[u]^-{\iota_*} \ar[r]^-{\delta_{X_+,Y}} & KR^a(X_+ - Y,\mathcal{G}) \ar[u]^-{\cong}
}
\]
Let $X'_- = X_- \setminus (-1,0] \times Y$. We have maps of pairs
\[
(X_- , Y) \to (X_-  , X'_- \cup Y) \leftarrow ( I \times Y ,  \partial I \times Y)
\]
where $I = [-1,0]$. The long exact sequences for these pairs then give a commutative diagram
\[
\xymatrix{
KR^{a-1}_c(Y, \mathcal{G}) \ar[rr]^-{\delta_{X_-,Y}} & & KR^a_c( X_- - Y, \mathcal{G}) \\
KR^{a-1}_c(X'_- , \mathcal{G}) \oplus KR^{a-1}_c(Y , \mathcal{G}) \ar[u]^{pr_2} \ar[d] \ar[rr]^-{\delta_{X_-,X'_- \cup Y}} & & KR^a_c( X_-  -( X'_- \cup Y) , \mathcal{G}) \ar[u] \ar[d]^-{\cong} \\
KR^{a-1}_c( \partial I \times Y , \mathcal{G}) \ar[rr]^-{\delta_{I \times Y , Y}} & & KR^a_c( (-1,0) \times Y , \mathcal{G}) \\
KR^{a-1}_c(Y, \mathcal{G}) \oplus KR^{a-1}_c(Y , \mathcal{G}) \ar[u]^-{\cong} & &
}
\]
The coboundary map $\delta_{I \times Y , Y}$ is of the form $(u,v) \mapsto (v-u)\tau$, where $\tau$ is the Thom class for $\mathbb{R} \times Y \to Y$. It follows that $\delta_{X_-,Y}$ is given by $v \mapsto v \tau$, followed by the extension by zero map $KR^a_c((-1,0) \times Y , \mathcal{G}) \to KR^a_c(X_- - Y, \mathcal{G}) \cong KR^a_c(U , \mathcal{G})$. This is the same as the map $v \mapsto i_* \iota_*(v)$.

The case $\delta_{X_+ , Y} : KR^{a-1}_c(Y, \mathcal{G}) \to KR^a_c(X_+ - Y , \mathcal{G})$ is almost identical, except that now since $Y$ is an ingoing boundary one finds that $\delta_{X_+,Y}$ is given by $v \mapsto - v \tau$, which is the same as the map $v \mapsto -j_*\iota_*(v)$. This proves commutativity of the leftmost square of the original diagram.
\end{proof}

Recall from Section \ref{sec:ggbpt} that the group of isomorphism classes of graded Real gerbes over a point is isomorphic to $\mathbb{Z}_4$ and is generated by $\mathcal{U} = L(\mathbb{R}_-)$, the lifting gerbe of $\mathbb{R}_-$. If $X$ is any space with an involution $\sigma$, then any point gerbe can be pulled back to $X$ via the map $X \to \{ pt \}$. Let $\mathcal{G}$ be any graded Real gerbe on $X$. The effect on twisted $KR$-theory of twisting by powers of $\mathcal{U}$ is a degree shift. This follows from a combination of the Thom isomorphism and $(1,1)$-periodicity of bigraded $KR$-theory:
\begin{align*}
KR^*(X , \mathcal{G} \otimes \mathcal{U}^{-p} ) &\cong KR^*(X , \mathcal{G} \otimes L(\mathbb{R}_-^p )^{-1} ) \\
& \cong KR^{*+p}( X \times D(\mathbb{R}_-^p) , X \times S(\mathbb{R}_-^p) , \mathcal{G} ) \\
& \cong KR^{*+2p}(X , \mathcal{G}).
\end{align*}

\section{Real $T$-duality}\label{sec:rtd}

\subsection{Real affine torus bundles}\label{sec:ratb}

Let $\Lambda$ be a lattice, i.e. a finitely generated free abelian group. Set $V = \Lambda \otimes_{\mathbb{Z}} \mathbb{R}$ and let $T$ be the torus $T = V/\Lambda$. Conversely given a torus $T$ we have $T \cong V/\Lambda$ where $V = H_1(T ; \mathbb{R}), \Lambda = H_1(T ; \mathbb{Z})$. A map $f : T_1 \to T_2$ of tori is said to be {\em affine} if $f^{-1}(0)f$ is a group homomorphism. Equivalently, writing $T_i = V_i/\Lambda_i$ then $f$ has the form $f(t) = \lambda( \widehat{t}) + u \; ({\rm mod} \; \Lambda_2)$ where $t \in T_i$, $\widehat{t}$ is a lift of $t$ to $V_1$, $\lambda \in Hom(\Lambda_1 , \Lambda_2)$ and $u \in T_2$. We usually write the group structure on tori using additive notation.

An {\em affine torus bundle} over $B$ with fibre $T = V/\Lambda$ is a torus bundle $\pi : X \to B$ whose fibres are $T$ and whose structure group is the group of affine automorphisms of $T$. The {\em monodromy local system} of $X$ is the local system $R^1 \pi_* \mathbb{Z}$ whose fibres are the first homology groups of the fibres of $X$. Fixing a basepoint $b \in B$ and an identification $T_b \cong T$ of the fibre of $X$ over $b$ with $T$, the monodromy local system is determined by the monodromy representation $\rho : \pi_1(B,b) \to GL(\Lambda)$. Given such a representation we write $\Lambda_\rho$ for the corresponding local system, thus if $\rho$ is the monodromy of $X$, then $R^1 \pi_* \mathbb{Z} \cong \Lambda_\rho$. In a similar fashion we have local systems $V_\rho, T_\rho$ corresponding to $R^1 \pi_* \mathbb{R}$ and $R^1 \pi_* (\mathbb{R}/\mathbb{Z})$.

Affine torus bundles are classified up to affine isomorphism by their monodromy $\rho$ and their first Chern class $c_1(X) \in H^2(B ; \Lambda_\rho)$ \cite{bar1}.

Let $\sigma$ be a Real structure on $B$. By a {\em Real structure} on an affine torus bundle $\pi : X \to B$, we mean a lift of $\sigma$ to an involution on $X$ which is affine on the fibres. Associated to a Real affine torus bundle is the equivariant local system $R^1 \pi_* \mathbb{Z}$. Let $\Gamma$ denote the orbifold fundamental group of $B/\!/ \langle \sigma \rangle$ based at $b$. This can be defined as the fundamental group of the Borel construction $X_{\mathbb{Z}_2} = X \times_{\mathbb{Z}_2} B\mathbb{Z}_2$. The Borel fibration $X \to X_{\mathbb{Z}_2} \to B$ yields a short exact sequence
\[
1 \to \pi_1(B) \to \Gamma \buildrel \epsilon \over \longrightarrow \mathbb{Z}_2 \to 1.
\]
The equivariant local system $R^1 \pi_* \mathbb{Z}$ is determined up to isomorphism by its monodromy representation $\rho : \Gamma \to GL(\Lambda)$. By a straightforward extension of \cite[\textsection 3]{bar1} to the Real setting, one has that Real affine torus bundles are classified up to isomorphism by their monodromy $\rho : \Gamma \to GL(\Lambda)$ and (equivariant) first Chern class $c_1(X) \in H^2_{\mathbb{Z}_2}( B ; \Lambda_\rho)$.

\begin{example}\label{ex:Bpt}
If $B = \{ pt \}$ is a point, then a Real affine torus bundle with fibre $T$ is equivalent to giving an affine involution $\sigma : T \to T$. The monodromy representation $\rho : \mathbb{Z}_2 \to GL(\Lambda)$ is the linear part of $\sigma$. Let us denote this by $\sigma_0 : \Lambda \to \Lambda$. The first Chern class has the form
\[
c \in H^2_{\mathbb{Z}_2}(pt ; \Lambda_\rho) \cong H^1_{\mathbb{Z}_2}( pt ; T_\rho) \cong \frac{ \{ c \in T \; | \; c + \sigma_0(c) = 0  \} }{ \{ c = u - \sigma_0(u) \; | \; u \in T \} }.
\]
Given a class in $H^2_{\mathbb{Z}_2}(pt ; \Lambda_\rho) \cong H^1_{\mathbb{Z}_2}(pt ; T_\rho)$ represented by $c \in T$ satisfying $c + \sigma_0(c) = 0$, the corresponding affine involution is $\sigma(t) = \sigma_0(t) + c$. This is an involution because $\sigma_0(c) + c = 0$.

Any action of $\mathbb{Z}_2 = \langle \sigma \rangle$ on a finite rank free abelian group $\Lambda$ is isomorphic to a direct sum of indecomposable modules, each of which is isomorphic to one of three types: the trivial representation $\mathbb{Z}$, the cyclotomic representation $\mathbb{Z}_-$ (where $\sigma$ acts as $-1$) of the regular representation $R = \mathbb{Z}^2$ where $\sigma(x,y) = (y,x)$. Thus $\Lambda_\rho$ can be decomposed into a direct sum of these three types. If $\Lambda_\rho = \bigoplus_i \Lambda_{\rho_i}$ is such a decomposition, then $H^2_{\mathbb{Z}_2}(pt ; \Lambda_\rho) \cong \bigoplus_i H^2_{\mathbb{Z}_2}(pt ; \Lambda_{\rho_i})$, hence the first Chern class similarly decomposes. It follows that $(T , \sigma)$ can be decomposed into a product $(T , \sigma) = \prod_i (T_i , \sigma_i)$. Therefore any affine involution on a torus decomposes into a product of affine involutions $(T_i , \sigma_i)$, where the linear part of $\sigma_i$ is either the trivial, cyclotomic, or regular representation. Let us consider the possible Chern classes in these cases. We have
\[
H^2_{\mathbb{Z}_2}(pt ; \mathbb{Z}) \cong \mathbb{Z}_2, \quad H^2_{\mathbb{Z}_2}(pt ; \mathbb{Z}_-) \cong H^2_{\mathbb{Z}_2}(pt ; R) \cong 0.
\]
Thus in the trivial case, there are two possibilities depending on whether $c$ is zero or non-zero, while in the cyclotomic and regular cases we always have $c=0$. These four cases can be explicitly described as follows:
\begin{itemize}
\item[(1)]{$\Lambda_\rho = \mathbb{Z}$, $c=0$. Then $T = \mathbb{R}/\mathbb{Z}$ and $\sigma(t) = t$ is the identity map.}
\item[(2)]{$\Lambda_\rho = \mathbb{Z}$, $c \neq 0$. Then $T = \mathbb{R}/\mathbb{Z}$ and $\sigma(t) = t + 1/2$.}
\item[(3)]{$\Lambda_\rho = \mathbb{Z}_-$. Then $T = \mathbb{R}/\mathbb{Z}$ and $\sigma(t) = -t$.}
\item[(4)]{$\Lambda_\rho = R$. Then $T = \mathbb{R}^2/\mathbb{Z}^2$ and $\sigma(t_1,t_2) = (t_2,t_1)$.}
\end{itemize}

\end{example}

\subsection{Affine gerbes}\label{sec:ag}

Let $\pi : X \to B$ be a Real affine torus bundle. Associated to $X$ we define a sheaf $\mathcal{A}(X)$ on $B$ by letting $(\mathcal{A}(X))(U)$ be the set of fibrewise affine functions $f : \pi^{-1}(U) \to S^1$. This is an abelian group under pointwise multiplication and it follows that $\mathcal{A}(X)$ is a sheaf of abelian groups. From its definition there is an evident inclusion map $i : \mathcal{A}(X) \to \pi_*( \mathcal{C}_X(S^1) )$, where $\mathcal{C}_X(S^1)$ denotes the sheaf of continuous $S^1$-valued functions on $X$.

Let $\mathcal{U} = \{ U_i \}_{i \in I}$ be an equivariant open cover of $B$. We get an induced equivariant cover $\pi^{-1}(\mathcal{U}) = \{ \pi^{-1}(U_i) \}_{i \in I}$ of $X$ by taking preimages. An {\em affine graded Real gerbe} on $X$ (with respect to $\mathcal{U}$) is a gerbe $\mathcal{G} = (L , \lambda)$ such that $L \to \coprod_{i,j} \pi^{-1}(U_{ij})$ is the trivial line bundle $L = \mathbb{C}$ (possibly with a non-trivial grading) and $\lambda_{ijk} : \pi^{-1}(U_{ijk}) \to S^1$ is fibrewise affine. Note that compatibility of $\lambda$ with the Real structure $\sigma$ implies that $\sigma^*( \lambda_{ijk}) = \overline{\lambda}_{ijk}$. Thus $\lambda_{ijk}$ can be regarded as a \v{C}ech $2$-cocycle valued in $\mathcal{A}(X)_-$.

We will say that a graded Real gerbe on $X$ is {\em affine} if it is isomorphic to an affine graded Real gerbe with respect to some pullback cover $\pi^{-1}(\mathcal{U})$. It follows easily that a graded Real gerbe $\mathcal{G}$ is affine if and only if its Dixmier--Douady class lies in the image of the natural map
\[
\pi^* : H^1_{\mathbb{Z}_2}( B ; \mathbb{Z}_2) \oplus H^2_{\mathbb{Z}_2}(B ; \mathcal{A}(X)_-) \to H^1_{\mathbb{Z}_2}(X ; \mathbb{Z}_2) \oplus H^3(X ; \mathbb{Z}_-)
\]
where on the first summand, $\pi^*$ is pullback and on the second summand $\pi^*$ is the composition
\[
H^2_{\mathbb{Z}_2}(B ; \mathcal{A}(X)_-) \buildrel i \over \longrightarrow H^2_{\mathbb{Z}_2}(B ; \pi_*(\mathcal{C}_X(S^1))_- ) \to H^2_{\mathbb{Z}_2}(X ; (\mathcal{C}_X(S^1))_-) \cong H^3_{\mathbb{Z}_2}(X ; \mathbb{Z}_-).
\]

Affine functions can be decomposed into a linear and constant term, giving rise to a short exact sequence
\[
0 \to \mathcal{C}_B(S^1) \to \mathcal{A}(X) \buildrel \lambda \over \longrightarrow \Lambda_\rho^* \to 0
\]
and similarly
\[
0 \to \mathcal{C}_B(S^1)_- \to \mathcal{A}(X)_- \buildrel \lambda \over \longrightarrow (\Lambda_{\rho^\vee}^*) \to 0
\]
where we define $\rho^\vee : \Gamma \to GL(\Lambda^*)$  by $\rho^{\vee}(g) = \epsilon(g) \rho^*(g)$, or equivalently, $\Lambda^*_{\rho^\vee} = (\Lambda_\rho)^*_-$. This gives a long exact sequence in cohomology, part of which is as follows
\[
H^2_{\mathbb{Z}_2}( B ; \mathcal{C}_B(S^1)_-) \buildrel \pi^* \over \longrightarrow H^2_{\mathbb{Z}_2}( B ; \mathcal{A}(X)_-) \buildrel \lambda \over \longrightarrow H^2_{\mathbb{Z}_2}(B ; \Lambda_{\rho^\vee}^* ).
\]

As mentioned above, there is a natural inclusion $i : \mathcal{A}(X) \to \pi_*( \mathcal{C}_X(S^1) )$. Moreover, we have:
\begin{proposition}\label{prop:affgrbiso}
The inclusion $i : \mathcal{A}(X) \to \pi_*( \mathcal{C}_X(S^1) )$ induces isomorphisms $H^j_{\mathbb{Z}_2}( B ; \mathcal{A}(X) ) \to H^j_{\mathbb{Z}_2}( B ; \pi_*(\mathcal{C}_X(S^1)))$ and $H^j_{\mathbb{Z}_2}( B ; \mathcal{A}(X)_- ) \to H^j_{\mathbb{Z}_2}( B ; \pi_*(\mathcal{C}_X(S^1))_-)$ for all $j > 0$.
\end{proposition}
\begin{proof}
We give the proof for the case $\mathcal{A}(X) \to \pi_*( \mathcal{C}_X(S^1) )$. The case $\mathcal{A}(X)_- \to \pi_*( \mathcal{C}_X(S^1) )_-$ is similar.

Start with the short exact sequence of equivariant sheaves on $X$ given by 
\[
0 \to \mathbb{Z} \to \mathcal{C}_X(\mathbb{R}) \to \mathcal{C}_X(S^1) \to 0
\]
where $\mathcal{C}_X(\mathbb{R})$ denotes the sheaf of continuous $\mathbb{R}$-valued functions. Apply the pushforward $\pi_*$ and its derived functors to get a long exact sequence
\[
0 \to \mathbb{Z} \to \pi_*(\mathcal{C}_X(\mathbb{R})) \to \pi_*(\mathcal{C}_X(S^1)) \to R^1 \pi_* \mathbb{Z} \to R^1 \pi_* \mathcal{C}_X(\mathbb{R}) \to \cdots
\]
We have that $R^1 \pi_* \mathcal{C}_X(\mathbb{R}) = 0$ because $\mathcal{C}_X(\mathbb{R})$ is a fine sheaf. We also have that $R^1 \pi_* \mathbb{Z} = \Lambda^*_\rho$, so the above long exact sequence gives a four-term exact sequence
\[
0 \to \mathbb{Z} \to \pi_*(\mathcal{C}_X(\mathbb{R})) \to \pi_*(\mathcal{C}_X(S^1)) \to \Lambda^*_\rho \to 0.
\]
The sheaf $\mathcal{A}(X)$ fits into a similar four term exact sequence
\[
0 \to \mathbb{Z} \to \mathcal{C}_B(\mathbb{R}) \to \mathcal{A}(X) \to \Lambda^*_\rho \to 0
\]
and the inclusion map $i : \mathcal{A}(X) \to \pi_*(\mathcal{C}_X(S^1))$ can be completed to a commutative diagram
\[
\xymatrix{
0 \ar[r] & \mathbb{Z} \ar[r]^-{\pi^*} & \pi_*(\mathcal{C}_X(\mathbb{R})) \ar[r] & \pi_*(\mathcal{C}_X(S^1)) \ar[r]^-{\lambda} & \Lambda^*_\rho \ar[r] & 0 \\
0 \ar[r] & \mathbb{Z} \ar@{=}[u] \ar[r] & \mathcal{C}_B(\mathbb{R}) \ar[u]^-{i} \ar[r] & \mathcal{A}(X) \ar[r]^-{\lambda} \ar[u]^-{i} & \Lambda^*_\rho \ar[r] \ar@{=}[u] & 0 
}
\]
Letting $Q$ be the quotient sheaf $Q = \pi_*(\mathbb{C}_X(\mathbb{R}))/\pi^*(\mathbb{Z})$, the above diagram induces a commutative diagram with exact rows
\[
\xymatrix{
0 \ar[r] & Q \ar[r] & \pi_*(\mathcal{C}_X(S^1)) \ar[r]^-{\lambda} & \Lambda^*_\rho \ar[r] & 0 \\
0 \ar[r] & \mathcal{C}_B(S^1) \ar[u]^i \ar[r] & \mathcal{A}(X) \ar[u]^-{i} \ar[r]^-{\lambda} & \Lambda^*_\rho \ar[r] \ar@{=}[u] & 0
}
\]
Taking cohomology gives a commutative diagram with exact rows:
\[
\xymatrix{
\cdots \ar[r] & H^j_{\mathbb{Z}_2}(B ; Q) \ar[r] & H^j_{\mathbb{Z}_2}(B ; \pi_*(\mathcal{C}_X(\mathbb{R}))) \ar[r]^-{\lambda} & H^j_{\mathbb{Z}_2}(B ; \Lambda^*_\rho) \ar[r] & \cdots \\
\cdots \ar[r] & H^j_{\mathbb{Z}_2}(B ; \mathcal{C}_B(S^1)) \ar[u]^-{i} \ar[r] & H^j_{\mathbb{Z}_2}( B ; \mathcal{A}(X) ) \ar[r]^-{\lambda} \ar[u]^-{i} & H^j_{\mathbb{Z}_2}(B ; \Lambda^*_\rho) \ar@{=}[u] \ar[r] & \cdots 
}
\]
We claim that the inclusion map $i : \mathcal{C}_B(S^1) \to Q$ induces an isomorphism $H^j_{\mathbb{Z}_2}(B ; \mathcal{C}_B(S^1) ) \to H^j_{\mathbb{Z}_2}(B ; Q)$ for each $j > 0$. Given this, it follows from the above diagram and the five-lemma that $i : H^j_{\mathbb{Z}_2}( B ; \mathcal{A}(X) ) \to H^j_{\mathbb{Z}_2}( B ; \pi_*(\mathcal{C}_X(S^1)))$ is an isomorphism for each $j > 0$.

To prove the claim, consider the commutative diagram with exact rows:
\[
\xymatrix{
0 \ar[r] & \mathbb{Z} \ar[r] & \pi_*(\mathcal{C}_X(\mathbb{R})) \ar[r] & Q \ar[r] & 0 \\
0 \ar[r] & \mathbb{Z} \ar@{=}[u] \ar[r] & \mathcal{C}_B(\mathbb{R}) \ar[r] \ar[u]^-{i} & \mathcal{C}_B(S^1) \ar[u]^-{i} \ar[r] & 0
}
\]
The claim now follows by taking long exact sequences in cohomology and using that $\pi_*(\mathcal{C}_X(\mathbb{R}))$ and $\mathcal{C}_B(\mathbb{R})$ are fine sheaves.
\end{proof}

According to Proposition \ref{prop:affgrbiso}, we could have equivalently defined Real affine gerbes to be those Real gerbes whose Dixmier--Douady class lies in the image of the natural map 
\[
\pi^* : H^1_{\mathbb{Z}_2}(B ; \mathbb{Z}_2) \oplus H^2_{\mathbb{Z}_2}( B ; \pi_*(\mathcal{C}_X(S^1))_-) \to H^1_{\mathbb{Z}_2}(X ; \mathbb{Z}_2) \oplus H^2_{\mathbb{Z}_2}( X ; \mathcal{C}_X(S^1)_-).
\]

Let $\pi : X \to B$ be a Real affine torus bundle. Recall again the short exact sequence $0 \to \mathcal{C}_B(S^1)_- \to \mathcal{A}(X)_- \to (\Lambda_\rho)^*_- \to 0$. Noting that $H^i_{\mathbb{Z}_2}(B ; \mathcal{C}_B(S^1)_-) \cong H^{i+1}_{\mathbb{Z}_2}(B ; \mathbb{Z}_-)$ for $i>0$, the associated long exact sequence has the form:
\[
\cdots H^3_{\mathbb{Z}_2}(B ; \mathbb{Z}_-) \to H^2_{\mathbb{Z}_2}( B ; \mathcal{A}(X)_-) \buildrel \lambda \over \longrightarrow H^2_{\mathbb{Z}_2}( B ; (\Lambda_\rho)^*_-) \buildrel \delta \over \longrightarrow H^4_{\mathbb{Z}_2}(B ; \mathbb{Z}_-) \to \cdots
\]

\begin{proposition}\label{prop:cup}
The coboundary map $\delta : H^2_{\mathbb{Z}_2}(B ; (\Lambda_\rho)^*_- ) \to H^4_{\mathbb{Z}_2}(B ; \mathbb{Z}_-)$ is given by $\delta(\alpha) = \langle c , \alpha \rangle$, the composition of the cup product $\alpha \mapsto c \smallsmile \alpha \in H^4_{\mathbb{Z}_2}(B ; (\Lambda_\rho)^*_- \otimes \Lambda_\rho)$ with the pairing $\langle \; , \; \rangle : (\Lambda_\rho)^*_- \otimes \Lambda_\rho \to \mathbb{Z}_-$.
\end{proposition}
\begin{proof}
First note that $\mathcal{H}om( (\Lambda_\rho)^*_- , \mathcal{C}_B(S^1)_- ) \cong \mathcal{C}_B(T_\rho) \cong (\mathcal{C}_B(V_\rho))/\Lambda_\rho$. The sheaf $\mathcal{A}(X)_-$ is an extension of $(\Lambda_\rho)^*_-$ by $\mathcal{C}_B(S^1)_-$ which is locally split, hence is classified by a class in $H^1_{\mathbb{Z}_2}( B ; \mathcal{C}_B(T_\rho) )$. In fact it follows easily from the definition of $\mathcal{A}(X)$ that the extension class is given by the first Chern class $c \in H^2_{\mathbb{Z}_2}( B ; \Lambda_\rho) \cong H^1_{\mathbb{Z}_2}( B ; \mathcal{C}_B(T_\rho))$. Let $t$ be the class in $H^1_{\mathbb{Z}_2}(B ; \mathcal{C}_B(T_\rho))$ such that $c = \delta t$. Therefore the coboundary map $\delta : H^2_{\mathbb{Z}_2}(B ; (\Lambda_\rho)^*_-) \to H^3_{\mathbb{Z}_2}( B ; \mathcal{C}_B(S^1)_-)$ is given by the cup product $\alpha \mapsto t \smallsmile \alpha$. More explicitly, if $t = [ t_{ij} ]$, where $t_{ij}$ is valued in $\mathcal{H}om( (\Lambda_\rho)^*_- , \mathcal{C}_B(S^1)_- )$ and $\alpha = [\alpha_{ijk}]$ where $\alpha_{ijk}$ is valued in $(\Lambda_\rho)^*_-$, then $t \smallsmile \alpha$ is given by the cocycle $t_{ij} \alpha_{jkl}$.

Let $\widetilde{t}_{ij}$ be lifts of $t_{ij}$ to $\mathcal{C}_B(V_\rho)$. Then $c = \delta t$ is given by $c_{ijk} = \widetilde{t}_{jk} - \widetilde{t}_{ik} + \widetilde{t}_{ij}$. Under $H^3_{\mathbb{Z}_2}(B ; \mathcal{C}_B(S^1)_-) \cong H^4_{\mathbb{Z}_2}(B ; \mathbb{Z}_-)$ we have that $t_{ij}\alpha_{jkl}$ is sent to
\begin{align*}
(\delta(\widetilde{t}\alpha))_{ijklm} &= \widetilde{t}_{jk}\alpha_{klm} - \widetilde{t}_{ik}\alpha_{klm} + \widetilde{t}_{ij}\alpha_{jlm} - \widetilde{t}_{ij}\alpha_{jkm} + \widetilde{t}_{ij}\alpha_{jkl} \\
&= (\widetilde{t}_{jk} - \widetilde{t}_{ik} + \widetilde{t}_{ij})\alpha_{klm} - \widetilde{t}_{ij}( \alpha_{klm} - \alpha_{jlm} + \alpha_{jkm} - \alpha_{jkl} ) \\
&= c_{ijk}\alpha_{klm}.
\end{align*}
Hence $\delta(\alpha) = \langle c , \alpha \rangle$.
\end{proof}

\subsection{Existence of $T$-duals}\label{sec:etd}

Let $\Lambda, \widehat{\Lambda}$ be free abelian groups of equal rank. Let $V = \Lambda \otimes_{\mathbb{Z}} \mathbb{R}$, $\widehat{V} = \widehat{\Lambda} \otimes_{\mathbb{Z}} \mathbb{R}$, let $T = V/\Lambda$ and $\widehat{T} = \widehat{V}/\widehat{\Lambda}$. Let $\pi : X \to B$, $\widehat{\pi} : \widehat{X} \to B$ be affine torus bundles where the fibres of $X$ are modelled on the torus $T = V/\Lambda$ and the fibres of $\widehat{X}$ are modelled on the dual torus $\widehat{T}$. Let $C = X \times_B \widehat{X}$ be the fibre product. Observe that $C$ is itself and affine torus bundle over $B$ with fibres modelled on $T \times \widehat{T}$. Let $p : C \to X$, $\widehat{p} : C \to \widehat{X}$ be the projection maps.

Let $\mathcal{G}, \widehat{\mathcal{G}}$ be affine graded Real gerbes on $X$ and $\widehat{X}$. We say that the pairs $(X , \mathcal{G})$, $(\widehat{X} , \widehat{\mathcal{G}})$ are {\em Real $T$-dual} if there is an isomorphism $\gamma : p^*(\mathcal{G}) \to \widehat{p}^*(\widehat{\mathcal{G}})$ which satisfies the {\em Poincar\'e property}, which we now describe. Since $\mathcal{G}, \widehat{\mathcal{G}}$ are affine, they are necessarily trivialisable along the fibres of $X$ and $\widehat{X}$ ({\em without} Real structure). Given $b \in B$, let $T_b, \widehat{T}_b$ denote the fibres of $X$ and $\widehat{X}$ over $b$. Note that the fibre of $C$ over $b$ is $T_b \times \widehat{T}_b$. Choose trivialisations
\[
\psi_b : 1 \to \mathcal{G}|_{T_b}, \quad \widehat{\psi}_b : 1 \to \widehat{\mathcal{G}}|_{\widehat{T}_b}
\] 
of $\mathcal{G}, \widehat{\mathcal{G}}$ over $b$ (here $1$ denotes the trivial gerbe). On $T_b \times \widehat{T}_b$ we have an automorphism of the trivial gerbe
\[
\xymatrix{
1 \ar[rr]^-{p^*(\psi_b)} & & p^*(\mathcal{G})|_{T_b \times \widehat{T}_b} \ar[rr]^-{\gamma|_{T_b \times \widehat{T}_b}} & & \widehat{p}^*(\widehat{\mathcal{G}})|_{T_b \times \widehat{T}_b} \ar[rr]^-{\widehat{p}^*(\widehat{\psi}_b^{-1})} & & 1
}
\]
Any such automorphism corresponds to a graded line bundle, hence to a class 
\[
\delta_b(\gamma) \in GrPic( T_b \times \widehat{T}_b) \cong H^0(T_b \times \widehat{T}_b ; \mathbb{Z}_2) \oplus H^2( T_b \times \widehat{T}_b ; \mathbb{Z}).
\]
Changing the trivialisations $\psi_b$ changes $\delta_b$ by an element of the pullback $p^* : GrPic(T_b) \to GrPic(T_b \times \widehat{T}_b)$ and similarly for $\widehat{\psi_b}$. Thus we have an invariantly defined class
\[
[\delta_b(\gamma)] \in \frac{GrPic( T_b \times \widehat{T}_b) }{p^*(GrPic(T_b)) + \widehat{p}^*(GrPic(\widehat{T}_b))} \cong H^1(T_b ; \mathbb{Z}) \otimes H^1(\widehat{T}_b ; \mathbb{Z}) \cong (\Lambda_\rho)_b^* \otimes (\Lambda_{\widehat{\rho}})_b^*,
\]
where $\rho, \widehat{\rho}$ are the monodromy representations of $X$ and $\widehat{X}$. This class varies continuously with $b$, so defines a global section
\[
\delta(\gamma) \in H^0(B ; \Lambda_\rho^* \otimes \Lambda_{\widehat{\rho}}^* ) \cong H^0(B ; Hom( \Lambda_{\widehat{\rho}} , \Lambda_{\rho}^* )).
\]

The Poincar\'e property is that $\delta$ is an isomorphism $\Lambda_{\widehat{\rho}} \cong \Lambda_{\rho}^*$. Note that up to this point, we have only considered $\rho,\widehat{\rho}$ as representations of the ordinary fundamental group of $B$ and not the orbifold fundamental group. However, using the Real structures on $\mathcal{G}$ and $\widehat{\mathcal{G}}$ and the fact that $\gamma$ is an isomorphism of Real gerbes, it follows easily that $\delta_{\sigma(b)}(\gamma) = -\sigma^*(\delta_b(\gamma))$. Thus $\delta(\gamma)$ can be regarded as an equivariant isomorphism
\[
\Lambda_{\widehat{\rho}} \cong \Lambda_{\rho}^* \otimes_{\mathbb{Z}} \mathbb{Z}_- = \Lambda^*_{\rho^{\vee}}.
\]
That is, the monodromy representations $\rho, \widehat{\rho}$ are related by $\widehat{\rho} = \rho^{\vee}$.

Let us define a Real $T$-dual $5$-tuple $(X , \mathcal{G} , \widehat{X} , \widehat{\mathcal{G}} , \gamma)$ to be a Real $T$-dual pair $(X , \mathcal{G}) , (\widehat{X} , \widehat{\mathcal{G}})$ together with an isomorphism $\gamma : p^*(\mathcal{G}) \to \widehat{p}^*(\widehat{\mathcal{G}})$ satisfying the Poincar\'e property.

\begin{proposition}\label{prop:tdualc}
Let $(X , \mathcal{G} , \widehat{X} , \widehat{\mathcal{G}} , \gamma)$ be a Real $T$-dual $5$-tuple. Let $\delta = \delta(\gamma) : \Lambda_{\widehat{\rho}} \to (\Lambda^*_\rho)_-$ be the isomorphism induced by $\gamma$. Let $c \in H^2_{\mathbb{Z}_2}( B ; \Lambda_\rho)$, $\widehat{c} \in H^2_{\mathbb{Z}_2}(B ; \Lambda_{\widehat{\rho}})$ denote the Chern classes of $X$ and $\widehat{X}$. Then the ungraded Dixmier--Douady classes of $\mathcal{G}$ and $\widehat{\mathcal{G}}$ in $H^3_{\mathbb{Z}_2}(X ; \mathbb{Z}_-)$ and $H^3_{\mathbb{Z}_2}(\widehat{X} ; \mathbb{Z}_-)$ admit lifts to $H^2_{\mathbb{Z}_2}( B ; \mathcal{A}(X)_-)$ and $H^2_{\mathbb{Z}_2}(B ; \mathcal{A}(\widehat{X})_-)$ such that $\lambda(\mathcal{G}) = \delta \widehat{c}$, $\lambda(\widehat{\mathcal{G}}) = \delta^{t} c$ (where $\delta^t : \Lambda_\rho \to (\Lambda_{\widehat{\rho}})^*_-$ denotes the transpose of $\delta$).
\end{proposition}
\begin{proof}
Let $C = X \times_B \widehat{X}$ and $q : C \to B$ the projection to $B$. We will make use of the Leray--Serre spectral sequence for the map $q$ and sheaf $\mathcal{C}_C(S^1)_-$. This is a spectral sequence $\{ E_r^{p,q} , d_r \}_{r \ge 2}$ abutting to $H^*_{\mathbb{Z}_2}( C ; \mathcal{C}_C(S^1)_-)$ and with
\[
E_2^{p,q} = H^p_{\mathbb{Z}_2}( B ; R^q q_* \mathcal{C}_C(S^1)_- ).
\]
In particular, from the spectral sequence we get an exact sequence
\[
E_2^{0,1} \buildrel d_2 \over \longrightarrow E_2^{2,0} \to H^2_{\mathbb{Z}_2}( C ; \mathcal{C}(S^1)_-).
\]
Observe that $E_2^{2,0} = H^2_{\mathbb{Z}_2}( B ; q_* \mathcal{C}_C(S^1)_-)$, which by Proposition \ref{prop:affgrbiso} is isomorphic to $H^2_{\mathbb{Z}_2}(B ; \mathcal{A}(C)_-)$. Since $\mathcal{G}$ and $\widehat{\mathcal{G}}$ are affine, there are lifts $k \in H^2_{\mathbb{Z}_2}(B ; \mathcal{A}(X)_-)$ and $\widehat{k} \in H^2_{\mathbb{Z}_2}(B ; \mathcal{A}(\widehat{X})_-)$ of the ungraded Dixmier--Douady classes of $\mathcal{G}$ and $\widehat{\mathcal{G}}$. Then $p^*(k), \widehat{p}^*(\widehat{k}) \in E_2^{2,0}$ are lifts of $p^*(\mathcal{G})$ and $\widehat{p}^*(\widetilde{\mathcal{G}})$.

Choose an open cover $\{ U_i \}_{i \in I}$ of $B$ where $\sigma$ acts freely on $I$ and $\sigma(U_i) = U_{\sigma(i)}$. Refining the cover if necessary, we have that $\mathcal{G}$, $\widehat{\mathcal{G}}$ can be represented by \v{C}ech $2$-cocycles $\theta_{ijk} \in \mathcal{A}(X)_-(U_{ijk})$ and $\widehat{\theta}_{ijk} \in \mathcal{A}(\widehat{X})_-(U_{ijk})$. Choosing such representations determines trivialisations $\psi_i : 1 \to \mathcal{G}|_{\pi^{-1}(U_i)}$, $\widehat{\psi}_i : 1 \to \widehat{\mathcal{G}}|_{\widehat{\pi}^{-1}(U_i)}$ (as gerbes without Real structures). Indeed, we have $(\theta|_{U_i})_{abc} = (\delta \mu_i)_{abc}$ where $(\mu_i)_{ab} = \theta_{iab}$ and similarly $(\widehat{\theta}|_{U_i})_{abc} = (\delta \widehat{\mu}_i)_{abc}$ where $(\widehat{\mu}_i)_{ab} = \widehat{\theta}_{iab}$. A trivialisation $\mu_i$ of the cocycle $\theta_{ijk}$ defines a trivialisation of the corresponding gerbe $\mathcal{G}$ in an obvious way. By construction, these trivialisations satisfy $\sigma^*(\psi_i) = \psi_{\sigma(i)}^{-1}$, $\sigma^*(\widehat{\psi}_i) = \widehat{\psi}_{\sigma(i)}^{-1}$. On the overlap $\pi^{-1}(U_{ij}) = \pi^{-1}(U_i) \cap \pi^{-1}(U_j)$ we have two trivialisations $\psi_i|_{U_{ij}}$ and $\psi_j|_{U_{ij}}$. They differ by an automorphism of the trivial gerbe on $\pi^{-1}(U_{ij})$, hence by a line bundle $A_{ij}$ on $\pi^{-1}(U_{ij})$. The transition functions for this line bundle are given by the cocycle $g_{ab} = (\mu_j)_{ab} (\mu_i)_{ab}^{-1} = \theta_{jab}\theta_{iab}^{-1}$. Since $\theta$ is a $2$-cocycle, we have $\theta_{jab}\theta_{iab}^{-1} = \theta_{ijb}^{-1} \theta_{ija}$, so $g_{ab} = h_b h_a^{-1}$, where $h_a = \theta_{ija}^{-1}$. So the line bundle $A_{ij}$ is trivial. Similar reasoning applies to the trivialisations $\widehat{\psi}_i$.

With respect to the trivialisations $\psi_i,\widehat{\psi}_i$, the restriction $\gamma|_{q^{-1}(U_i)}$ can be represented by an automorphism of the trivial gerbe, hence by a line bundle $L_i \to q^{-1}(U_i)$. Since over $U_i \cap U_j$, the trivialisations $\psi_i, \psi_j$ differ by a trivial line bundle and similarly $\widehat{\psi}_i, \widehat{\psi}_j$ differ by a trivial line bundle, it follows that $L_i|_{U_{ij}} \cong L_j|_{U_{ij}}$. Also since $\sigma^*(\psi_i) = \psi_{\sigma(i)}^{-1}$ and $\sigma^*(\widehat{\psi}_i) = \widehat{\psi}_{\sigma(i)}^{-1}$, it follows that $\sigma^*(L_i) \cong L_{\sigma(i)}^{-1}$.  It follows that the collection of isomorphism classes $\{ [L_i] \}_{i \in I}$ forms an element 
\[
l = [L_i] \in H^0_{\mathbb{Z}_2}( B ; R^2 q_* \mathbb{Z}_-) \cong H^0_{\mathbb{Z}_2}(B ; R^1 q_* \mathcal{C}_C(S^1)) = E_2^{0,1}.
\]
Now by a straightforward adaptation of \cite[Proposition 3.8]{bar2} to the Real setting, we have that 
\begin{equation}\label{equ:kkhat}
p^*(k) + d_2(l) = \widehat{p}^*(\widehat{k}).
\end{equation}

Observe that $R^1 q_* \mathbb{Z} \cong \Lambda_\rho^* \oplus \Lambda_{\widehat{\rho}}^*$, thus
\begin{align*}
E_2^{0,1} &= H^0_{\mathbb{Z}_2}(B ; (\wedge^2 (\Lambda_\rho \oplus \Lambda_{\widehat{\rho}} )^*)_-) \\
&\cong H^0_{\mathbb{Z}_2}(B ; (\wedge^2 \Lambda_\rho^*)_-) \oplus H^0_{\mathbb{Z}_2}(B ; \Lambda_{\widehat{\rho}}^* \otimes (\Lambda_\rho)^*_- ) \oplus H^0_{\mathbb{Z}_2}(B ; (\wedge^2 \Lambda_{\widehat{\rho}}^*)_- ).
\end{align*}
In this isomorphism, $\Lambda_{\widehat{\rho}}^* \otimes (\Lambda_\rho)^*_-$ is regarded as a subsheaf of $(\wedge^2 (\Lambda_\rho \oplus \Lambda_{\widehat{\rho}})^*)_-$ by
\[
\lambda_1 \otimes \lambda_2 \mapsto \lambda_1 \wedge \lambda_2.
\]

Under this decomposition we can write $l = ( p^*(\alpha) , \beta , \widehat{p}^*(\widehat{\alpha}))$. In fact, from the definition of $l$ it is clear that the component $\beta \in H^0_{\mathbb{Z}_2}(B ; \Lambda_{\widehat{\rho}}^* \otimes (\Lambda_\rho)^*_- )$ is given by $\beta = \delta(\gamma)$. Then Equation (\ref{equ:kkhat}) can be re-written as
\begin{equation}\label{equ:kkhat2}
p^*(k) + d_2( p^*(\alpha)) = d_2(\delta) = \widehat{p}^*( \widehat{k})  - d_2( \widehat{p}^*(\widehat{\alpha})).
\end{equation}
Note that $d_2( p^*(\alpha) ) = p^*( d_2(\alpha))$, where the differential $d_2$ on the right hand side is the differential in the Leray--Serre spectral sequence for $\pi : X \to B$ with values in the sheaf $\mathcal{C}_X(S^1)_-$. Similarly $d_2(\widehat{p}^*(\widehat{\alpha})) = \widehat{p}^*( d_2(\widehat{\alpha}))$. Set $k' = k + d_2(\alpha)$ and $\widehat{k}' = \widehat{k} - d_2(\widehat{\alpha})$. Then $k',\widehat{k}$ are lifts of the ungraded Dixmier--Douady classes of $\mathcal{G}$ and $\widehat{\mathcal{G}}$ to $H^2_{\mathbb{Z}_2}(B ; \mathcal{A}(X)_-)$ and $H^2_{\mathbb{Z}_2}( B ; \mathcal{A}(\widehat{X})_-)$ and Equation (\ref{equ:kkhat2}) gives
\[
p^*(k') + d_2(\delta) = \widehat{p}^*(\widehat{k}').
\]
Applying $\lambda : H^2_{\mathbb{Z}_2}( B ; \mathcal{A}(C)_-) \to H^2_{\mathbb{Z}_2}( B ; \Lambda_\rho^* \oplus \Lambda_{\widehat{\rho}}^*)$ to this equality, we have
\[
p^*(\lambda(\mathcal{G})) + \lambda( d_2(\delta) ) = \widehat{p}^*(\lambda(\widehat{\mathcal{G}})).
\]
If we can show that $\lambda( d_2( \delta ) ) = (- \delta \widehat{c} , \delta^t c)$, then we must have $\lambda(\mathcal{G}) = \delta \widehat{c}$, $\lambda( \widehat{\mathcal{G}}) = \delta^t c$ and the proof will be complete.

To evaluate $\lambda( d_2(\delta) )$ we will make use of The Leray--Serre spectral sequence for the map $q : C \to B$ and the constant sheaf $\mathbb{Z}$. Denote this spectral sequence by ${E'}_r^{p,q}$. We have ${E_2'}^{p,q} = H^p_{\mathbb{Z}_2}(B ; R^q q_* \mathbb{Z}_- )$. From \cite{bar3} (see also \cite[Theorem 3.11]{bar2}) the short exact sequence 
\begin{equation}\label{equ:ses}
0 \to \mathbb{Z}_- \to \mathcal{C}_C(\mathbb{R})_- \to \mathcal{C}_C(S^1)_- \to 0
\end{equation}
induces a morphism $D_r : E_r^{p,q} \to {E'_r}^{p,q+1}$ of Leray--Serre spectral sequences associated to $\mathcal{C}_C(S^1)_-$ and $\mathbb{Z}_-$. The short exact sequence (\ref{equ:ses}) induces a long exact sequence of equivariant sheaves on $B$:
\[
\cdots \to R^j q_* \mathbb{Z}_- \to R^j q_* \mathcal{C}_C(\mathbb{R})_- \to R^j q_* \mathcal{C}(S^1)_- \buildrel \delta \over \longrightarrow R^{j+1} q_* \mathbb{Z}_- \to \cdots
\]
On the $E_2$ pages $D_r : E_2^{p,q} \to {E_2'}^{p,q}$ is given by the maps
\[
D_r : H^p_{\mathbb{Z}_2}( B ; R^q q_* \mathcal{C}_C(S^1)_-) \to H^p_{\mathbb{Z}_2}(B ; R^{q+1} q_* \mathbb{Z}_- )
\]
in cohomology induced by the coboundary maps $\delta : R^q q_* \mathcal{C}_C(S^1)_- \to R^{q+1} q_* \mathbb{Z}_-$ in the above long exact sequence. In particular, for the case $r=2$, $(p,q) = (2,0)$, the morphism $D_2 : H^2_{\mathbb{Z}_2}(B ; q_* \mathcal{C}_C(S^1)_-) \to H^2_{\mathbb{Z}_2}(B ; R^1 q_* \mathbb{Z}_-) \cong H^2_{\mathbb{Z}}(B ; \Lambda_\rho \oplus \Lambda_{\widehat{\rho}})$ is precisely the map $\lambda$. Since $D$ is a morphism of spectral sequences, it commutes with the differentials of the spectral sequence and thus
\[
\lambda( d_2(\delta) ) = D( d_2(\delta) ) = d_2( D(\delta) ).
\]
Here $\delta = \delta(\gamma)$ is the isomorphism $\delta : \Lambda_{\widehat{\rho}} \to (\Lambda_\rho)^*_-$ induced by $\gamma$. Then $D(\delta) \in {E_2'}^{0,2}$ is just $\delta$ regarded as an element of ${E_2'}^{0,2} = H^0_{\mathbb{Z}_2}( B ; ( \wedge^2 (\Lambda_{\rho} \oplus \Lambda_{\widehat{\rho}})^*)_- )$. By a straightforward extension of \cite[Proposition 3.3]{bar1} to the Real setting, the differential $d_2 : {E_2'}^{p,q} \to {E_2'}^{p+2,q-1}$ is given by the cup product with the Chern class of $C \to B$, which is $(c , \widehat{c} \, )$, followed by the contraction
\[
\iota : ( \Lambda_\rho \oplus \Lambda_{\widehat{\rho}}) \otimes (\wedge^q (\Lambda_\rho \oplus \Lambda_{\widehat{\rho}})^*)_- \to (\wedge^{q-1} (\Lambda_\rho \oplus \Lambda_{\widehat{\rho}})^*)_-
\]
In particular, it follows that $d_2( D(\delta) ) = \iota_{(c , \widehat{c} \, )} D(\delta) = ( -\delta \widehat{c} , \delta^t c)$, and the proof is complete.
\end{proof}

\begin{theorem}\label{thm:rtd}
Let $\Lambda_\rho, \Lambda_{\widehat{\rho}}$ be $\mathbb{Z}_2$-equivariant local systems on $B$ each with coefficient group a finite rank free abelian group and let $\delta : \Lambda_{\widehat{\rho}} \to (\Lambda_\rho^*)_-$ be an isomorphism. Let $X \to B$ be a Real affine torus bundle on $B$ with monodromy $\rho$ and let $\mathcal{G}$ be a Real affine gerbe on $X$. Then there exists a Real $T$-dual $5$-tuple $(X , \mathcal{G} , \widehat{X} , \widehat{\mathcal{G}} , \gamma)$ such that $\delta(\gamma) = \delta$.
\end{theorem}
\begin{proof}
Choose a lift of the ungraded Dixmier--Douady class of $\mathcal{G}$ to a class $k \in H^2_{\mathbb{Z}_2}( X ; \mathcal{A}(X)_-)$. Define $\widehat{c} \in H^2_{\mathbb{Z}_2}(B ; \Lambda_{\widehat{\rho}})$ to be $\widehat{c} = \delta^{-1} \lambda(k)$. The pair $(\widehat{\rho} , \widehat{c})$ determines a Real affine torus bundle $\widehat{\pi} : \widehat{X} \to B$. By Proposition \ref{prop:cup}, we have $\langle c , \widehat{c} \, \rangle = 0 \in H^4_{\mathbb{Z}_2}(B ; \mathbb{Z}_-)$. Then by the long exact sequence associated to $0 \to \mathcal{C}_B(S^1)_- \to \mathcal{A}(\widehat{X})_- \to (\Lambda_{\widehat{\rho}})^*_- \to 0$ and a second application of Proposition \ref{prop:cup}, $\langle \, \widehat{c} , c \rangle = \langle c , \widehat{c} \, \rangle = 0$ implies that there is a class $\widehat{k}_0 \in H^2_{\mathbb{Z}_2}( B ; \mathcal{A}(\widehat{X})_-)$ such that $\lambda(\widehat{k}_0) = c$. As in the proof of Proposition \ref{prop:tdualc} we let $\{ E_r^{p,q} , d_r \}$ denote the Leray--Serre spectral sequence for $q : C \to B$ with coefficients in $\mathcal{C}_C(S^1)_-$. Notice that $\delta : \Lambda_{\widehat{\rho}} \to (\Lambda_\rho^*)_-$ defines a class
\[
\delta \in H^0_{\mathbb{Z}_2}(B ; \Lambda_{\widehat{\rho}}^* \otimes (\Lambda_\rho)^*_- )  \subseteq H^0_{\mathbb{Z}_2}(B ; (\wedge^2 (\Lambda_\rho \oplus \Lambda_{\widehat{\rho}} )^*)_-) = E_2^{0,1}.
\]
Let us define $\omega \in E_2^{2,0}$ by
\[
\omega = p^*(k) + d_2(\delta) - \widehat{p}^*(\widehat{k}_0).
\]
Then by the same calculation used in the proof of Proposition \ref{prop:tdualc}, we have $\lambda(\omega) = 0$. Hence by the long exact sequence for $0 \to \mathcal{C}_B(S^1)_- \to \mathcal{A}(C)_- \to (\Lambda_\rho \oplus \Lambda_{\widehat{\rho}})^*_- \to 0$, we have that $\omega = q^*(\omega_0)$ for some $\omega_0 \in H^2_{\mathbb{Z}_2}( B ; \mathcal{C}_B(S^1)_-)$. Setting $\widehat{k} = \widehat{k}_0 + q^*(\omega_0)$, we have
\begin{equation}\label{equ:kkhat3}
p^*(k) + d_2(\delta) = \widehat{p}^*(\widehat{k}).
\end{equation}
Since $\mathcal{G}$ is a Real affine graded gerbe its grading class is of the form $\pi^*(e)$ for some $e \in H^1_{\mathbb{Z}_2}( B ; \mathbb{Z}_2)$ (in fact, the Leray--Serre spectral sequence for $\pi : X \to B$ with coefficients in $\mathbb{Z}_2$ implies that $\pi^* : H^1_{\mathbb{Z}_2}(B ; \mathbb{Z}_2) \to H^1_{\mathbb{Z}_2}(X ; \mathbb{Z}_2)$ is injective, so $e$ is uniquely determined). Let $\widehat{\mathcal{G}}$ be the graded Real affine gerbe on $\widehat{X}$ whose grading class is $\widehat{\pi}^*(e)$ and whose ungraded Dixmier--Douady class is the image of $\widehat{k}$ under the natural map $H^2_{\mathbb{Z}_2}(B ; \mathcal{A}(\widehat{X})_-) \to H^3_{\mathbb{Z}_2}( \widehat{X} ; \mathbb{Z}_-)$. Then by construction $\widehat{\mathcal{G}}$ is an affine gerbe and Equation (\ref{equ:kkhat}) implies the existence of an isomorphism $\gamma : p^*(\mathcal{G}) \to \widehat{p}^*(\widehat{\mathcal{G}})$ such that $\delta(\gamma) = \delta$ (this last condition follows by a straightforward extension of \cite[Proposition 3.9]{bar2} to the Real setting).
\end{proof}

\subsection{Real $T$-duality over a point}\label{sec:rtdpt}

In this section we examine in more detail the case of Real $T$-duality when the base $B$ is a point. Let $(X,\sigma_X)$ be an Real affine torus bundle over a point. Write $X = V/\Lambda$ where $\Lambda$ is a free, finite rank module over $\mathbb{Z}_2 = \langle \sigma \rangle$ and $V = \Lambda \otimes_{\mathbb{Z}} \mathbb{R}$. The action of $\sigma$ on $\Lambda$ induces a linear action on $X$ which we denote by $\sigma$. The affine action $\sigma_X : X \to X$ is then of the form $\sigma_X(x) = \sigma(x) + t$ where $t \in X$ satisfies $t + \sigma(t) = 0$. Then $t$ represents a class $[t] \in H^1_{\mathbb{Z}_2}( pt ; X )$ and the first Chern class $c \in H^2_{\mathbb{Z}_2}( pt ; \Lambda)$ is given by $c = \delta [t]$, where $\delta$ is the coboundary map for the short exact sequence $0 \to \Lambda \to V \to X \to 0$. If we wish to specify the dependence of $X$ on $\Lambda$ and $c$ we will write $X = X(\Lambda,c)$.

Let $\mathcal{G}$ be a graded Real affine gerbe on $X$. We will say $\mathcal{G}$ is a point gerbe if it is a pullback of a graded Real gerbe under the map $X \to pt$. The grading class of an graded Real affine gerbe is an element of $H^1_{\mathbb{Z}_2}(pt ; \mathbb{Z}_2) \cong \mathbb{Z}_2$, so is manifestly a pullback of a grading over a point. For the purpose of $T$-duality such gradings are uninteresting and so throughout this section we work only with gerbes with trivial grading. Any non-trivially graded $T$-dual pair of gerbes is obtained from a trivially graded $T$-dual pair by tensoring both sides by a point gerbe corresponding to the generator of $H^1_{\mathbb{Z}_2}(pt ; \mathbb{Z}_2)$.

Let $\mathcal{A}(X)$ denote the $\mathbb{Z}_2$-module given by the space of affine $S^1$-valued functions on $X$. Then (ungraded) Real affine gerbes are given by classes in $H^2_{\mathbb{Z}_2}(pt ; \mathcal{A}(X)_-)$. From the short exact sequence
\[
0 \to S^1_- \to \mathcal{A}(X)_- \to \Lambda_- \to 0
\]
we get a long exact sequence
\[
\cdots \to H^2_{\mathbb{Z}_2}(pt ; S^1_-) \to H^2_{\mathbb{Z}_2}( pt ; \mathcal{A}(X)_- ) \to H^2_{\mathbb{Z}_2}(pt ; \Lambda^*_-) \to H^3_{\mathbb{Z}_2}(pt ; S^1_-) \to \cdots
\]
We have $H^3_{\mathbb{Z}_2}(pt ; S^1_-) \cong H^4_{\mathbb{Z}_2}(pt ; \mathbb{Z}_- ) = 0$ and $H^2_{\mathbb{Z}_2}(pt ; S^1_-) \cong H^3_{\mathbb{Z}_2}(pt ; \mathbb{Z}_-) \cong \mathbb{Z}_2$. Thus we get an exact sequence
\[
\mathbb{Z}_2 \to H^2_{\mathbb{Z}_2}( pt ; \mathcal{A}(X)_- ) \buildrel \lambda \over \longrightarrow H^2_{\mathbb{Z}_2}(pt ; \Lambda^*_-) \to 0.
\]
The first map $\mathbb{Z}_2 \to H^2_{\mathbb{Z}_2}( pt ; \mathcal{A}(X)_- )$ corresponds to pulling back a point gerbe.

\begin{proposition}\label{prop:affgrb}
Let $\widehat{c} \in H^2_{\mathbb{Z}_2}(pt ; \Lambda^*_-)$. Then there exists a Real affine gerbe $\mathcal{G}$ on $X$ with $\lambda(\mathcal{G}) = \widehat{c}$ (and with trivial grading). If $c = 0$ then there are exactly two such gerbes, which differ by tensoring by a point gerbe. If $c \neq 0$ then $\mathcal{G}$ is unique.
\end{proposition}
\begin{proof}
In the case $c = 0$, the exact sequence $0 \to S^1_- \to \mathcal{A}(X)_- \to \Lambda_- \to 0$ splits, giving
\[
H^2_{\mathbb{Z}_2}( pt ; \mathcal{A}(X)_- ) \cong H^2_{\mathbb{Z}_2}(pt ; S^1_-) \oplus H^2_{\mathbb{Z}_2}(pt ; \Lambda_-) \cong \mathbb{Z}_2 \oplus H^2_{\mathbb{Z}_2}(pt ; \Lambda_-).
\]
Thus for given $\widehat{c}$, there are exactly two (ungraded) Real affine gerbes on $X$ with $\lambda(\mathcal{G}) = \widehat{c}$ and these two gerbes differ by a point gerbe.

Now suppose $c \neq 0$. Let $u$ denote the generator of $H^2_{\mathbb{Z}_2}(pt ; S^1_-) \cong \mathbb{Z}_2$. We will show that the image of $u$ in $H^2_{\mathbb{Z}_2}( pt ; \mathcal{A}(X)_-)$ is trivial. This implies uniqueness of $\mathcal{G}$ such that $\lambda(\mathcal{G}) = \widehat{c}$.

Recall that for any $\mathbb{Z}_2$-module $M$, we have $H^2_{\mathbb{Z}_2}(pt ; M ) \cong \{ m \in M \; | \; \sigma(m) = m\}/\{m = n + \sigma(n), n \in M\}$. For the case $M = S^1_- = \mathbb{R}/\mathbb{Z}$ with action $\sigma(t) = -t$, the class $u$ is represented by $t = 1/2$.

Recall from Example \ref{ex:Bpt} that $\Lambda$ is isomorphic to a direct sum $\Lambda \cong \mathbb{Z}^t \oplus \mathbb{Z}^c_- \oplus R^r$ of copies of the trivial, cyclotomic and regular representations. One finds $H^2_{\mathbb{Z}_2}(pt ; \Lambda) \cong \mathbb{Z}_2^t$ and that any class in $H^2_{\mathbb{Z}_2}(pt ; \Lambda)$ can be represented by an element of the form $( v_1 , v_2 , \dots , v_t , 0 , \dots 0)$. In particular, $c = (c_1 , \dots , c_t , 0 , \dots , 0)$ for some integers $c_1, \dots , c_t$. Since $c \neq 0$, we must have that $c_i$ is odd for some $i$. Recall that $\sigma_X(x) = \sigma(x) + t$ where $t \in T$ satisfies $c = \delta [t]$. Let $\widetilde{t} \in V$ be a lift of $t$ to $v$. Then $c = \delta [t]$ means that $c = [ \widetilde{t} + \sigma(\widetilde{t})]$. Writing $\widetilde{t} = ( \widetilde{t}_1 , \dots , \widetilde{t}_t , \dots )$, we have that $c_i = 2\widetilde{t}_i$. Since $c_i$ is odd, this means that $\widetilde{t}_i$ is half an odd integer.

Elements of $\mathcal{A}(X)_-$ can be identified with pairs $\Lambda^* \times S^1$ where $(\lambda , u) \in \Lambda^* \times S^1$ corresponds to the affine function $x \mapsto f(x) = \lambda(x) + u \; ({\rm mod} \; \mathbb{Z})$. The action of $\sigma$ on $\mathcal{A}(X)_-$ is by $f \mapsto \overline{\sigma^*(f)}$. In terms of $(\lambda , u)$ this action is given by $(\lambda , u) \mapsto (-\sigma^*(\lambda) , -\lambda(\widetilde{t}) - u)$. Hence
\[
H^2_{\mathbb{Z}_2}(pt ; \mathcal{A}(X)_-) \cong \frac{\{ (\lambda , u ) \in \Lambda^* \times S^1 \; | \; \sigma^*(\lambda) = -\lambda, 2u = -\lambda(\widetilde{t}) \}}{ \{ (\alpha - \sigma^*\alpha , -\alpha(\widetilde{t}) ),  \alpha \in \Lambda^*  \} }.
\]
Let $\alpha \in \Lambda^*$ be given by $\alpha( x_1 , \dots , x_t , \dots ) = -x_i$. Then $\alpha - \sigma^*\alpha = 0$ and $-\alpha( \widetilde{t} ) = \widetilde{t}_i = c_i/2 = 1/2 \; ({\rm mod} \; \mathbb{Z})$. This shows that the image of $u$ in $H^2_{\mathbb{Z}_2}(pt ; \mathcal{A}(X)_-)$, which equals $[(0 , 1/2)]$ is trivial.
\end{proof}

\begin{proposition}\label{prop:tdpt}
Let $\Lambda, \Lambda^{\vee}$ be free, finite rank $\mathbb{Z}_2$-modules and let $\delta : \Lambda^{\vee} \to \Lambda^*_-$ be an isomorphism. Let $c \in H^2_{\mathbb{Z}_2}(pt ; \Lambda)$, $\widehat{c} \in H^2_{\mathbb{Z}_2}(pt ; \Lambda^{\vee})$. Let $X = X(\Lambda , c)$ and let $\mathcal{G}$ be a Real affine gerbe on $X$ with trivial grading and with $\lambda(\mathcal{G}) = \delta \widehat{c}$. Let $\widehat{X} = X(\Lambda^{\vee} , c^{\vee})$. Then there exists a Real affine gerbe $\widehat{\mathcal{G}}$ with trivial grading on $\widehat{X}$ and an isomorphism $\gamma : p^*(\mathcal{G}) \to \widehat{p}^*(\widehat{\mathcal{G}})$ such that $\delta(\gamma) = \delta$. Thus $(X , \mathcal{G})$ and $(\widehat{X} , \widehat{\mathcal{G}})$ are Real $T$-dual. Moreover, we have:
\begin{itemize}
\item[(1)]{If $c=0$ or $\widehat{c} \neq 0$, then $\widehat{\mathcal{G}}$ is uniquely determined by the above requirements.}
\item[(2)]{If $c \neq 0$ and $\widehat{c} = 0$, then there are two possible isomorphism classes for $\widehat{\mathcal{G}}$ which differ by a point gerbe.}
\end{itemize}
\end{proposition}
\begin{proof}
The existence of $\widehat{\mathcal{G}}$ and $\gamma$ with the stated properties with immediate from Theorem \ref{thm:rtd}. It remains to prove assertions (1) and (2). Since $(X , \mathcal{G})$ and $(\widehat{X} , \widehat{\mathcal{G}})$ are $T$-dual, we have that $\lambda( \widehat{\mathcal{G}} ) = \delta^{t}(c)$. If $\widehat{c} \neq 0$, then there is a unique such gerbe up to isomorphism satisfying $\lambda( \widehat{\mathcal{G}} ) = \delta^{t}(c)$ by Proposition \ref{prop:affgrb}. So now assume that $\widehat{c} = 0$. In this case, there are exactly two isomorphism classes of gerbes on $\widehat{X}$ satisfying $\lambda( \widehat{\mathcal{G}} ) = \delta^{t}(c)$ and these two classes differ by a point gerbe. Denote them by $\widehat{\mathcal{G}}_1, \widehat{\mathcal{G}}_2$. If $c \neq 0$, then both $\widehat{\mathcal{G}}_1$ and $\widehat{\mathcal{G}}_2$ must be $T$-dual to $\mathcal{G}$, because $\mathcal{G}$ is the unique Real affine gerbe on $X$ with $\lambda( \mathcal{G}) = \delta (\widehat{c})$ by Proposition \ref{prop:affgrb}. Lastly, if $c$ and $\widehat{c}$ are both zero then both $\mathcal{G}$ and $\widehat{\mathcal{G}}$ are point gerbes. Clearly in this case if $\mathcal{G} \cong \widehat{\mathcal{G}}$, then they are $T$-dual. Conversely, if $\mathcal{G}$ and $\widehat{\mathcal{G}}$ are $T$-dual, then the existence of the isomorphism $\gamma : p^*(\mathcal{G}) \to \widehat{p}^*(\widehat{\mathcal{G}})$ means that $\mathcal{G}$ and $\widehat{\mathcal{G}}$ are isomorphic when pulled back to $C = X \times \widehat{X}$. But since $c$ and $\widehat{c}$ are zero, the involutions $\sigma_X$ and $\sigma_{\widehat{X}}$ have fixed points. Hence so does the product involution $\sigma_X \times \sigma_{\widehat{X}}$. Pulling back $\gamma$ to a fixed point of $\sigma_X \times \sigma_{\widehat{X}}$, we see that $\mathcal{G} \cong \widehat{\mathcal{G}}$. In particular, if $c$ and $\widehat{c}$ are both zero then $\widehat{\mathcal{G}}$ is uniquely determined.
\end{proof}

\section{The Real Fourier--Mukai transform}\label{sec:rfmt}

Let $B$ be a compact smooth manifold and let $\sigma : B \to B$ be a Real structure given by a smooth involution. Let $(X , \mathcal{G}), (\widehat{X} , \widehat{\mathcal{G}})$ be Real $T$-dual pairs over $B$ where $X,\widehat{X}$ have $n$-dimensional fibres. Let $\gamma : p^*(\mathcal{G}) \to \widehat{p}^*(\widehat{\mathcal{G}})$ be an isomorphism satisfying the Poincar\'e property. The {\em Real Fourier--Mukai transform} associated to this data is the group homomorphism
\[
\Phi : KR^*(X , \mathcal{G} \otimes L(V_\rho) ) \to KR^{*-n}(\widehat{X} , \widehat{\mathcal{G}})
\]
given by the composition
\begin{small}
\[
\xymatrix@C-1pc{
KR^*(X , \mathcal{G} \otimes L(V_\rho) ) \ar[r]^-{p^*} & KR^*( C , p^*(\mathcal{G}) \otimes L(V_\rho) ) \ar[r]^-{\gamma} & KR^*(C , \widehat{p}^*(\widehat{\mathcal{G}}) \otimes L(V_\rho)) \ar[r]^-{\widehat{p}_*} & KR^{*-n}(\widehat{X} , \widehat{\mathcal{G}})
}
\]
\end{small}

Our main result is:
\begin{theorem}\label{thm:rfm}
The Real Fourier--Mukai transform $\Phi : KR^*(X , \mathcal{G} \otimes L(V_\rho) ) \to KR^{*-n}(\widehat{X} , \widehat{\mathcal{G}})$ is an isomorphism for any Real $T$-dual pair $(X , \mathcal{G}), (\widehat{X} , \widehat{\mathcal{G}})$ and any isomorphism $\gamma : p^*(\mathcal{G}) \to \widehat{p}^*(\widehat{\mathcal{G}})$ satisfying the Poincar\'e property.
\end{theorem}

In the course of proving this result, we will also need to make use of the Fourier--Mukai transform for affine torus bundles over a non-compact base. In such cases we use compactly supported $K$-theory, so the Real Fourier--Mukai transform takes the form of a map
\[
\Phi : KR^*_c(X , \mathcal{G} \otimes L(V_\rho) ) \to KR^{*-n}_c(\widehat{X} , \widehat{\mathcal{G}}).
\]

In this section, we prove Theorem \ref{thm:rfm}. The proof consists of the following steps. In Section \ref{sec:redtor} we use a Mayer--Vietoris type argument to reduce to the case that the base space $B$ is a point. Thus $X$ and $\widehat{X}$ are dual tori. We consider this case in Section \ref{sec:gentor} and we show that if $(X , \sigma)$ factors as a product of Real tori $(X_1 , \sigma_1) \times \cdots \times (X_k , \sigma_k)$ then the Fourier--Mukai transform similarly factors into a composition of partial Fourier--Mukai transforms. Then using another Mayer--Vietoris type argument we reduce to the case that $(X,\sigma)$ is indecomposable (i.e. does not factor into a product of lower-dimensional Real tori). Finally in Section \ref{sec:indtor} we prove the Real Fourier--Mukai transform is an isomorphism in the case of indecomposable tori.

\subsection{Reduction to torus case}\label{sec:redtor}

In this section we reduce the problem of proving that the Real Fourier--Mukai transform is an isomorphism for affine torus bundles over an arbitrary base $B$ to the case where the base is a point and hence $X$, $\widehat{X}$ are dual tori.

The general strategy is as follows. We allow the base to be non-compact, but in such cases the Real Fourier--Mukai transform will be defined for compactly supported $KR$-theory. Suppose $(X,\mathcal{G})$, $(\widehat{X} , \widehat{\mathcal{G}})$ are $T$-dual Real affine torus bundles over $B$ and $\gamma : p^*(\mathcal{G}) \to \widehat{p}^*(\widehat{\mathcal{G}})$ is an isomorphism satisfying the Poincar\'e property. Then we get a Real Fourier--Mukai transform
\[
\Phi = \Phi(X , \mathcal{G} , \widehat{X} , \widehat{\mathcal{G}} , \gamma ) : KR^*_c( X , L(V_\rho) \otimes \mathcal{G}) \to KR^{*-m}_c( \widehat{X} , \widehat{\mathcal{G}})
\]
where $m$ is the dimension of the fibres. We will say that $B$ is {\em RFM} if $\Phi$ is an isomorphism for all such tuples $(X , \mathcal{G} , \widehat{X} , \widehat{\mathcal{G}} , \gamma )$ defined over $B$.

Given a smooth manifold $B$, suppose that we have a decomposition $B = B_- \cup_Y \cup B_+$ where $B_+,B_-$ are smooth manifolds with boundary $Y$ and $Y \to B$ is proper. Assume $B_-$ has a collar neighbourhood $(-2 , 0] \times Y$, $B_+$ has a collar neighbourhood $[0,2) \times Y$, set $U = B_- \cup_Y [0,1) \times Y$, $V = (-1,0] \times Y \cup_Y B_+$ so that $\{U,V\}$ is an open cover of $B$. Assume $\sigma$ preserves the decomposition $B = B_+ \cup_Y \cup B_-$ and the action on the collar $(-2,2) \times Y$ is of the form $\sigma(t,y) = (t , \sigma_Y(y))$, where $\sigma_Y = \sigma|_Y$. By restriction we get affine torus bundles over $U$ denotes $X|_U, \widehat{X}|_V$ and similarly for $V, U \cap V$. The tuple $(X , \mathcal{G} , \widehat{X} , \widehat{\mathcal{G}} , \gamma )$ can be restricted to $U$ and we let $\Phi|_U$ denote the corresponding Fourier--Mukai transform over $U$. Similarly we have $\Phi|_V$ and $\Phi|_{U \cap V}$. Consider the following diagram:

\begin{small}
\[
\xymatrix@C-1.2pc{
\cdots \ar[r] & KR^a_c( X|_{U \cap V} , L(V_\rho) \otimes \mathcal{G}) \ar[rr]^-{(i_* , -j_*)} \ar[d]^-{\Phi|_{U \cap V}} & & {\begin{matrix}KR^a_c(X|_U , L(V_\rho) \otimes \mathcal{G}) \\ \oplus \\ KR^a_c(X|_V , L(V_\rho) \otimes \mathcal{G})\end{matrix}} \ar[rr]^-{k_* + l_*} \ar[d]^-{ \Phi|_U \oplus \Phi|_V } & & KR^a_c(X , \mathcal{G}) \ar[r]^-{\delta} \ar[d]^-{\Phi} & \cdots \\
\cdots \ar[r] & KR^{a-m}_c( \widehat{X}|_{U \cap V} , \widehat{\mathcal{G}}) \ar[rr]^-{(i_* , -j_*)} & & {\begin{matrix}KR^{a-m}_c(\widehat{X}|_U , \widehat{\mathcal{G}}) \\ \oplus \\ KR^{a-m}_c(\widehat{X}|_V , \widehat{\mathcal{G}}) \end{matrix}}  \ar[rr]^-{k_* + l_*} & & KR^{a-m}_c(\widehat{X} , \widehat{\mathcal{G}}) \ar[r]^-{\delta} & \cdots
}
\]
\end{small}
In this diagram, the rows are Mayer--Vietoris sequences, as in Proposition \ref{prop:mv}. Commutativity easily follows from the definition of the Fourier--Mukai transform and the definition of $\delta$ given in Proposition \ref{prop:mv}. By the five-lemma, if $\Phi|_U, \Phi|_V$ and $\Phi|_{U \cap V}$ are isomorphisms, then so is $\Phi$. Therefore, if $U,V$ and $U \cap V$ are RFM, then so is $B$.

Now the idea is to prove that a general compact smooth manifold $B$ with Real structure $\sigma$ is RFM by decomposing it into simpler pieces, each of which is RFM and using the above Mayer--Vietorus argument. We will also need to make use of a few other reduction steps which we state as lemmas.

\begin{lemma}\label{lem:thomrfm}
Let $B_0$ be a compact smooth manifold with smooth involution and let $r : B \to B_0$ be a Real vector bundle over $B_0$. If $B_0$ is RFM, then so is $B$.
\end{lemma}
\begin{proof}
By the classification of Real affine torus bundles, any Real affine torus bundle $X \to B$ is isomorphic to a pullback $X \cong r^*(X_0)$ of a Real affine torus bundle $X_0 \to B_0$, namely $X_0$ is the restriction of $X$ to the zero section. Similary any graded Real gerbe $\mathcal{G}$ on $X$ is isomorphic to the pullback $\mathcal{G} \cong r^*(\mathcal{G}_0)$, where $\mathcal{G}_0$ is the restriction of $\mathcal{G}$ to the zero section. If $(X,\mathcal{G}), (\widehat{X}, \widehat{\mathcal{G}})$ are Real $T$-dual pairs on $B$ and $\gamma : p^*(\mathcal{G}) \to \widetilde{p}^*(\widetilde{\mathcal{G}})$ is an isomorphism satisfying the Poincar\'e property, then $(X , \mathcal{G}) \cong r^*( X_0 , \mathcal{G}_0)$, $(\widehat{X} , \widehat{\mathcal{G}}) \cong r^*( \widehat{X}_0 , \widehat{\mathcal{G}}_0)$ are isomorphic to pullbacks and under these identifications, $\gamma$ is equivalent to the pullback of an isomorphism $\gamma_0 : p^*(\mathcal{G}_0) \to \widetilde{p}^*(\widetilde{\mathcal{G}}_0)$ satisfying the Poincar\'e property. Thus any tuple $(X , \mathcal{G} , \widehat{X} , \widehat{\mathcal{G}} , \gamma )$ on $B$ can be identified with the pullback of a tuple $(X_0 , \mathcal{G}_0 , \widehat{X}_0 , \widehat{\mathcal{G}}_0 , \gamma_0 )$ on $B_0$. If we let $\Phi$ and $\Phi_0$ denote the Real Fourier--Mukai transforms with respect to these tuples then it is easily seen that we have a commutative diagram
\[
\xymatrix{
KR_c^*(X , L(V_\rho) \otimes \mathcal{G}) \ar[r]^-{r_*} \ar[d]^-{\Phi} & KR^{*-n}( X_0 , L(V_\rho) \otimes L(B)^{-1} \otimes \mathcal{G}_0) \ar[d]^-{\Phi_0} \\
KR_c^{*-n}(\widehat{X} , \widehat{\mathcal{G}}) \ar[r]^-{r_*} & KR^{*-n-m}(\widehat{X}_0 , L(B)^{-1} \otimes \mathcal{G}_0)
}
\]
where $L(B)$ denotes the lifting gerbe of $B \to B_0$ (thought of as a vector bundle over $B_0$), $n$ is dimension of the fibres of $B \to B_0$, $m$ is the dimension of the fibres of $X \to B$. The horizontal maps are isomorphisms (by the Thom isomorphism) and $\Phi_0$ is an isomorphism by the assumption that $B$ is RFM. Hence $\Phi$ is also an isomorphism.
\end{proof}

\begin{lemma}\label{lem:sphere}
Suppose that $B = \{pt\}$ is RFM. Then $RFM$ holds for any sphere $S^n$, equipped with trivial involution.
\end{lemma}
\begin{proof}
We use induction on $n$. The case $n=0$ holds because $S^0$ is the disjoint union of two points. For $n > 0$, write $S^n = B_+ \cup_{S^{n-1}} B_-$, where $B_+,B_-$ are $n$-dimensional closed discs. The interiors of $B_+, B_-$ are RFM by Lemma \ref{lem:thomrfm} and $S^{n-1}$ is RFM by the inductive hypothesis. Thus $S^n$ is RFM by a Mayer--Vietoris argument.
\end{proof}

\begin{lemma}\label{lem:rfmtriv}
Suppose that $\{pt\}$ is RFM. Then any compact smooth manifold $B$ with trivial involution is RFM.
\end{lemma}
\begin{proof}
Choose a Morse function $f : B \to \mathbb{R}$ with critical points $b_1, \dots , b_k$ and with $c_1 < c_2 < \cdots < c_k$, where $c_i = f(b_i)$. Choose real numbers $a_0, \dots , a_k$ with $a_0 < c_1 < a_1 < c_2 < \cdots < c_k  < a_k$ and set $B_i = f^{-1}(-\infty , a_i )$. Then each $B_i$ is a smooth manifold, $B_0$ is empty, $B_k = B$ and $B_i$ is obtained from $B_{i-1}$ by attaching a handle of some index $\lambda_i$. It follows that $B_i$ admits an open cover given by $U = B_{i-1}$, $V = D^{\lambda_i} \times D^{n-\lambda_i}$ and $U \cap V \cong S^{\lambda_i-1} \times (-1,1) \times D^{n-\lambda_i}$, where $D^j$ denotes the open unit disc in $\mathbb{R}^j$ and $n = dim(B)$. Both $V$ and $U \cap V$ are RFM by Lemmas \ref{lem:thomrfm} and \ref{lem:sphere}. Applying a Mayer--Vietoris type argument to the cover $U,V$, we see that if $B_{i-1}$ is RFM, then so is $B_i$. Hence by induction on $i$, $B$ is RFM.
\end{proof}

\begin{lemma}\label{lem:free}
Let $B$ be a compact smooth manifold and let $\sigma$ be a smooth involution on $B$ which acts freely. Then $(B , \sigma)$ is RFM.
\end{lemma}
\begin{proof}
Set $Q = B/\langle \sigma \rangle$. Since $\sigma$ acts freely, the quotient map $\rho : B \to Q$ is an unbranched double cover. Suppose that $B \to Q$ is a trivial covering, so $B = Q \cup Q$ and $\sigma$ swaps the two copies. Let $X \to B$ be a Real affine torus bundle over $B$. Since the involution $\sigma_X$ on $X$ covers $\sigma$, we have that $X = X_Q \cup X_Q$ where $X_Q = X|_Q$ and $\sigma_X$ swaps the two copies. In this case the twisted $KR$-theory of $(X, \mathcal{G})$ is isomorphic to the twisted $K$-theory of $(X|_Q , \mathcal{G}|_{X_Q})$. One easily checks that in such a case Real $T$-duality of a pair $(X , \mathcal{G})$, $(\widehat{X} , \widehat{\mathcal{G}})$ on $B$ just reduces to ordinary $T$-duality of their restrictions over $Q$. Moreover the Real Fourier--Mukai transform just corresponds to the ordinary Fourier--Mukai transform in complex $K$-theory $K^*(X_Q , L(V_\rho)|_{Q} \otimes \mathcal{G}_Q) \to K^{*-m}( \widehat{X}_Q , \widehat{\mathcal{G}}|_{\widehat{X}_Q})$, which is already known to be an isomorphism.

In the more general case where the covering $\rho : B \to Q$ is non-trivial, we consider a Morse function $f : Q \to \mathbb{R}$ exactly as in the proof of Lemma \ref{lem:rfmtriv}. We get a sequence of open subsets $Q_0 , Q_1 , \dots , Q_k$, where $Q_0$ is empty, $Q = Q_k$ and $Q_i$ is obtained from $Q_{i-1}$ by attaching a handle of index $\lambda_i$. So $Q_i = U \cup V$ where $U = Q_{i-1}$, $V = D^{\lambda_i} \times D^{n - \lambda_i}$, $U \cap V \cong S^{\lambda_i - 1} \times (-1,1) \times D^{n-\lambda_i}$. We get a similar decomposition of $B$ by taking preimages $B_i = \rho^{-1}(Q_i)$. Thus $B_i = \widetilde{U} \cup \widetilde{V}$ where $\widetilde{U} = B_{i-1}$, $V = \rho^{-1}(V)$, $\widetilde{U} \cap \widetilde{V} = \rho^{-1}(U \cap V)$. Since $V$ is contractible, the double cover $\widetilde{V} \to V$ is trivial and so $V$ is RFM. If $\lambda_i \neq 2$, then $U \cap V$ is simply-connected, so the double covering $\widetilde{U} \cap \widetilde{V} \to U \cap V$ is trivial and $\widetilde{U} \cap \widetilde{V}$ is RFM. In the case $\lambda_i = 2$, if the double cover $\widetilde{U} \cap \widetilde{V} \to U \cap V$ is trivial, then $\widetilde{U} \cap \widetilde{V}$ is RFM. Suppose instead that the double covering is non-trivial. Since $U \cap V \cong \mathbb{R}^m \times S^1$ (where $m = dim(B) +1 - \lambda_i$), we have that $\widetilde{U} \cap \widetilde{V} \to U \cap V$ must be $\mathbb{R}^m \times \widetilde{S}^1$, where $\widetilde{S}^1 \to S^1$ is the unique non-trivial double covering of the circle. Using Lemma \ref{lem:thomrfm}, we will have that $\widetilde{U} \cap \widetilde{V}$ is RFM provided $\widetilde{S}^1$ is RFM. Assuming this for the time being, a Mayer--Vietoris argument will then imply that $B_i$ is RFM provided $B_{i-1}$ is RFM. Then induction on $i$ gives us that $B$ is RFM.

It remains to prove that $S^1$ equipped with a free involution is RFM. Wrrite $S^1 = \mathbb{R}/\mathbb{Z}$ with involution $\sigma(t) = t + 1/2$. Consider the open cover $U = (0,1/2) \cup (1/2 , 1)$, $V = (-1/4,1/4) \cup (1/4 , 3/4)$. Clearly $\sigma$ preserves $U$ and $V$ and permutes the two connected components of $U$ and $V$. Thus $U$ and $V$ are RFM. Furthermore $U \cap V = (-1/4 , 0) \cup (0,1/4) \cup (1/4 , 1/2) \cup (1/2 , 3/4)$ is made up of four components which are swapped by $\sigma$ in pairs. Thus $U \cap V$ is RFM. The Mayer--Vietoris argument then gives us that $(S^1 , \sigma)$ is RFM.
\end{proof}

\begin{lemma}\label{lem:free2}
Let $B$ be the interior of a compact smooth manifold $W$ with boundary. Let $\sigma_W$ be a smooth involution on $W$ and let $\sigma_B$ be the restriction of $\sigma_W$ to $B$. If $\sigma_W$ acts freely, then $(B , \sigma_B)$ is RFM.
\end{lemma}
\begin{proof}
Let $Y$ be the boundary of $W$. Note that $Y$ is compact and that $\sigma_W$ acts freely on $Y$. Let $D_W = W \cup_Y W$ be the double of $W$. That is, $D_W$ is the compact manifold obtained by taking two copies of $W$ and identifying their boundaries. The involution $\sigma_W$ extends to $D_W$ in the obvious way and the resulting involution $\sigma_{D_W}$ is free. Let $X \to B$ be a Real affine torus bundle over $B$. By the classification of Real affine torus bundles, it is clear that such a bundle extends to a Real affine torus bundle $X_W \to W$. Furthermore, $X_W$ extends to $D_W$ by taking two copies of $X_W$ and identifying them over $Y$. A similar doubling procedure works for any Real affine gerbe $\mathcal{G}$ on $X$. Thus any pair $(X , \mathcal{G})$ can be doubled to produce a pair $(D_X , D\mathcal{G})$ on $W$. If $(X , \mathcal{G})$ and $(\widehat{X} , \widehat{\mathcal{G}})$ are Real $T$-duals over $B$, then it is easily seen that their doubles $(D_X , D\mathcal{G})$, $(D_{\widehat{X}} , D\widehat{\mathcal{G}})$ are Real $T$-duals over $W$. Indeed given an isomorphism $\gamma : p^*(\mathcal{G}) \to \widehat{p}(\widehat{\mathcal{G}})$ satisfying the Poincar\'e property, we can construct a doubled isomorphism $D_\gamma$ which also satisfies the Poincar\'e property. By Lemma \ref{lem:free}, $W$ and $Y$ are RFM, so $\Phi(D_\gamma)$ and $\Phi(\gamma|_Y)$ are isomorphism. Now we appy a Mayer--Vietoris argument to the decomposition $D_W = W \cup_Y W$ to see that $\Phi(\gamma)$ is an isomorphism. Hence $B$ is RFM.
\end{proof}

\begin{lemma}\label{lem:trivrfm}
Let $B$ be a compact, smooth manifold and $\sigma$ a smooth involution. If the fixed point set $B^\sigma$ is RFM (with trivial involution), then $(B , \sigma)$ is RFM.
\end{lemma}
\begin{proof}
Consider the covering $U = B \setminus B^{\sigma}$, $V = \nu B^\sigma$, a tubular neighbourhood of $B^\sigma$. Since $\sigma$ acts freely on $U$ and $U \cap V$, we have that $U$ and $U \cap V$ are RFM, by Lemma \ref{lem:free2}. Further, $V$ is the total space of a Real vector bundle over $B^\sigma$. So if $B^\sigma$ is RFM, then so is $V$ by Lemma \ref{lem:thomrfm}. By the Mayer--Vietoris argument, if $U,V$ and $U \cap V$ are RFM, then so is $B$.
\end{proof}

Putting all of the above results together, we have:

\begin{lemma}\label{lem:redpt}
Suppose that $\{ pt \}$ is RFM. Then $(B , \sigma)$ is RFM for any compact smooth manifold $B$ and any smooth involution $\sigma$.
\end{lemma}
\begin{proof}
By Lemma \ref{lem:trivrfm}, it suffices to prove the result in the case that $\sigma$ acts trivially. Then the result follows from Lemma \ref{lem:rfmtriv}.
\end{proof}

\subsection{Torus case}\label{sec:gentor}

By the results of Section \ref{sec:redtor} we are reduced to establishing the Real Fourier--Mukai transform is an isomorphism in the case that $B = \{pt\}$. Thus $X,\widehat{X}$ are tori. In this section we will reduce the problem of establishing that the Real Fourier--Mukai transform is an isomorphism for arbitrary tori down to the case of proving it for indecomposable tori, meaning that $(X,\sigma)$ can not be expressed as a product of lower dimensional tori. As in Section \ref{sec:rtdpt}, it suffices to restrict attention to the case that all gerbes have a trivial grading.

Let $(X , \sigma)$ be a Real affine torus. Suppose that $(X , \sigma)$ is decomposable into a product $(X , \sigma) = (X_1 , \sigma_1) \times (X_2 , \sigma_2)$. If we write $X = V/\Lambda$ then the decomposition $X = X_1 \times X_2$ corresponds to a decomposition $\Lambda = \Lambda_1 \oplus \Lambda_2$ preserved by $\sigma$ and a corresponding decomposition $V = V_1 \oplus V_2$. 

\begin{lemma}\label{lem:grbdec}
If $(X,\sigma)$ decomposes into a product $(X_1 , \sigma_1) \times (X_2 , \sigma_2)$, then any graded Real affine gerbe $\mathcal{G}$ on $X$ factors into a product $\mathcal{G} \cong \pi_1^*(\mathcal{G}_1) \otimes \pi_2^*(\mathcal{G}_2)$ of graded Real affine gerbes on $X_1,X_2$ ($\pi_1,\pi_2$ are the projections to $X_1,X_2$).
\end{lemma}
\begin{proof}
Let $\mathcal{A}(X)$ denote the $\mathbb{Z}_2$-module given by the space of affine $S^1$-valued functions on $X$. Then graded Real affine gerbes are given by classes in $H^2_{\mathbb{Z}_2}(pt ; \mathcal{A}(X)_-)$. Similarly we have the $\mathbb{Z}_2$-modules $\mathcal{A}(X_1)$, $\mathcal{A}(X_2)$. Any affine function on $X = X_1 \times X_2$ can be written as a product of an affine function on $X_1$ and an affine function on $X_2$. This gives rise to a short exact sequence
\[
0 \to S^1_- \to \mathcal{A}(X_1)_- \oplus \mathcal{A}(X_2)_- \to \mathcal{A}(X)_-
\]
and a corresponding long exact sequence
\[
\cdots \to H^2_{\mathbb{Z}_2}(pt ; S^1_-) \to \begin{matrix}H^2_{\mathbb{Z}_2}(pt ; \mathcal{A}(X_1)_-) \\ \oplus \\ H^2_{\mathbb{Z}_2}(pt ; \mathcal{A}(X_2)_-) \end{matrix} \to H^2_{\mathbb{Z}_2}(pt ; \mathcal{A}(X)_-) \to H^3_{\mathbb{Z}_2}(pt ; S^1_-) \to \cdots
\]
We have $H^3_{\mathbb{Z}_2}(pt ; S^1_-) \cong H^4_{\mathbb{Z}_2}(pt ; \mathbb{Z}_-) = 0$, so the map $H^2_{\mathbb{Z}_2}(pt ; \mathcal{A}(X_1)_-) \oplus H^2_{\mathbb{Z}_2}(pt ; \mathcal{A}(X_2)_-)\to H^2_{\mathbb{Z}_2}(pt ; \mathcal{A}(X)_-)$ is surjective.
\end{proof}

\begin{proposition}\label{prop:tdfac}
Suppose that $(X , \mathcal{G})$ and $(\widehat{X} , \widehat{\mathcal{G}})$ are Real $T$-dual pairs and that $\gamma : p^*(\mathcal{G}) \to \widehat{p}^*(\widehat{\mathcal{G}})$ is an isomorphism satisfying the Poincar\'e property. Suppose that $(X , \sigma_X)$ factors into a product $(X , \sigma_X) = (X_1 , \sigma_{X_1}) \times (X_2 , \sigma_{X_2})$. Then there exists:
\begin{itemize}
\item[(1)]{A factorisation $(\widehat{X} , \sigma_{\widehat{X}}) \cong (\widehat{X}_1 , \sigma_{\widehat{X}_1}) \times (\widehat{X}_2 , \sigma_{\widehat{X}_2})$}
\item[(2)]{Real (trivially graded) affine gerbes $\mathcal{G}_1, \mathcal{G}_2, \widehat{\mathcal{G}}_1, \widehat{\mathcal{G}}_2$ on $X_1,X_2,\widehat{X}_1,\widehat{X}_2$ such that $\mathcal{G} \cong \mathcal{G}_1 \otimes \mathcal{G}_2$, $\widehat{\mathcal{G}} \cong \widehat{\mathcal{G}}_1 \otimes \widehat{\mathcal{G}}_2$}
\item[(3)]{Isomorphisms $\gamma_i : p_i^*(\mathcal{G}_i) \to \widehat{p}_i^*(\widehat{\mathcal{G}}_i)$ on $X_i \times \widehat{X}_i$ satisfying the Poincar\'e property and such that the isomorphisms $\gamma$ and $\gamma_1 \otimes \gamma_2$ differ by the pullback of Real line bundles on $X$ and $\widehat{X}$.}
\end{itemize}

\end{proposition}
\begin{proof}
Write $X = V/\Lambda$, $\widehat{X} = V^{\vee}/\Lambda^{\vee}$. Set $\delta = \delta(\gamma) : \Lambda^{\vee} \to \Lambda^*_-$. Note that $\delta$ is an isomorphism of $\mathbb{Z}_2$-modules. The factorisation $X = X_1 \times X_2$ corresponds to a $\sigma$-invariant deomposition $\Lambda = \Lambda_1 \oplus \Lambda_2$. This gives rise to a corresponding decomposition $\Lambda^*_- = (\Lambda_1^*)_- \oplus (\Lambda_2^*)_-$. Set $\Lambda^{\vee}_i = \delta^{-1}( (\Lambda_i^*)_-)$. Then we have a $\mathbb{Z}_2$-invariant decomposition $\Lambda^{\vee} = \Lambda^{\vee}_1 \oplus \Lambda^{\vee}_2$ with the property that $\delta$ sends $\Lambda^{\vee}_i$ to $(\Lambda_i)^*_-$. Corresponding to this decomposition of $\Lambda^{\vee}$ we get a decomposition $\widehat{X} = \widehat{X}_1 \times \widehat{X}_2$. Let $\delta_i : \Lambda^{\vee}_i \to (\Lambda^*_i)_-$ be the restriction of $\delta$ to $\Lambda^{\vee}_i$.

Choose a factorisation $\mathcal{G} \cong \mathcal{G}_1 \otimes \mathcal{G}_2$ as in Lemma \ref{lem:grbdec}. From Proposition \ref{prop:tdpt}, there exists Real affine gerbes $\widehat{\mathcal{G}}_1$, $\widehat{\mathcal{G}}_2$ on $\widehat{X}_1, \widehat{X}_2$ and isomorphisms $\gamma_i : p_i^*(\mathcal{G}_i) \to \widehat{p}_i^*(\widehat{\mathcal{G}}_i)$ on $X_i \times \widehat{X}_i$ on $X_i \times \widehat{X}_i$ satisfying $\delta(\gamma_i) = \delta_i$. In particular, $\gamma_i$ satisfies the Poincar\'e property. Tensoring these isomorphisms together we get an isomorphism $\gamma_1 \otimes \gamma_2 : p^*(\mathcal{G}) \to \widehat{p}^*( \widehat{\mathcal{G}}_1 \widehat{\mathcal{G}}_2 )$ such that $\delta(\gamma_1 \otimes \gamma_2) = \delta_1 \oplus \delta_2 = \delta$. Thus $(X , \mathcal{G})$ and $(\widehat{X} , \widehat{\mathcal{G}}_1 \widehat{\mathcal{G}}_2)$ are $T$-dual pairs. From Proposition \ref{prop:tdpt}, we either have $\widehat{\mathcal{G}} \cong \widehat{\mathcal{G}}_1 \otimes \widehat{\mathcal{G}}_2$, or $\widehat{\mathcal{G}}$ and $\widehat{\mathcal{G}}_1 \otimes \widehat{\mathcal{G}}_2$ differ by a point gerbe which we denote by $\mathcal{H}$. In the latter case it follows from Proposition \ref{prop:tdpt} that $\widehat{c} = 0$ and $c \neq 0$. The vanishing of $\widehat{c}$ means that the Chern classes $\widehat{c}_1, \widehat{c}_2$ of $\widehat{X}_1, \widehat{X}_2$ both vanish. On the other hand, the non-vanishing of $c$ implies that one of the Chern classes $c_1, c_2$ of $X_1, X_2$ is non-vanishing. Without loss of generality, assume $c_1 \neq 0$. Applying Proposition \ref{prop:tdpt} to $(X_1 , \mathcal{G}_1)$, we see that $(\widehat{X}_1 , \widehat{\mathcal{G}}_1 \otimes \mathcal{H})$ is also $T$-dual to $(X , \mathcal{G}_1)$. So, after replacing $\widehat{\mathcal{G}}_1$ by $\widehat{\mathcal{G}}_1 \otimes \mathcal{H}$, we can assume that $\widehat{\mathcal{G}} \cong \widehat{\mathcal{G}}_1 \otimes \widehat{\mathcal{G}}_2$.

To finish, we note that since $\delta( \gamma_1 \otimes \gamma_2) = \delta = \delta(\gamma)$, the isomorphisms $\gamma$ and $\gamma_1 \otimes \gamma_2$ must differ by the pullback of Real line bundles on $X$ and $\widehat{X}$.
\end{proof}

Let $(X, \mathcal{G}) , (\widehat{X} , \widehat{\mathcal{G}})$ and $\gamma$ be as in Proposition \ref{prop:tdfac}. Then associated to the isomorphisms $\gamma$ and $\gamma_1 \otimes \gamma_2$ are Fourier--Mukai transforms
\[
\Phi(\gamma), \Phi(\gamma_1 \otimes \gamma_2) : KR^*( X , L(V) \otimes \mathcal{G}) \to KR^{*-m}( \widehat{X} , \widehat{\mathcal{G}} )
\]
where $m = dim(X)$. Since, according to Proposition \ref{prop:tdfac}, $\gamma$ and $\gamma_1 \otimes \gamma_2$ differ by tensoring by line Real bundles $N,\widehat{N}$ on $X$ and $\widehat{X}$, it follows that we have an equality
\begin{equation}\label{equ:fmcomp}
\Phi(\gamma) = t_{\widehat{N}} \circ \Phi(\gamma_1 \otimes \gamma_2) \circ t_N
\end{equation}
where $t_N : KR^*(X , L(V) \otimes \mathcal{G}) \to KR^*(X , L(V) \otimes \mathcal{G})$ is the automorphism given by tensoring by $N$ and $t_{\widehat{N}}$ is similarly defined. Since $t_N$ and $t_{\widehat{N}}$ are isomorphisms, it follows that $\Phi(\gamma)$ is an isomorphism if and only if $\Phi(\gamma_1 \otimes \gamma_2)$ is an isomorphism.

Next we observe that $X = X_1 \times X_2$ and $X_1 \times \widehat{X}_2$ are Real affine torus bundle over $X_1$ with fibres $X_2, \widehat{X}_2$. Clearly $(X_1 \times X_2 , \mathcal{G}_1 \otimes \mathcal{G}_2)$ and $(X_1 \times \widehat{X}_2 , \mathcal{G}_1 \otimes \widehat{\mathcal{G}}_2)$ are Real $T$-duals via the isomorphism $\gamma_2$. Similarly $(X_1 \times \widehat{X}_2 , \mathcal{G}_1 \otimes \widehat{\mathcal{G}}_2)$ and $(\widehat{X}_1 \times \widehat{X}_2 , \widehat{\mathcal{G}}_1 \otimes \widehat{\mathcal{G}}_2)$ thought of as torus bundles over $\widehat{X}_2$ are Real $T$-dual via the isomorphism $\gamma_1$. Corresponding to the $T$-dual pairs are Fourier--Muakai transforms
\[
\Phi(\gamma_2) : KR^*( X_1 \times X_2 , L(V_1) \otimes L(V_2) \otimes \mathcal{G}_1 \otimes \mathcal{G}_2) \to KR^{*-m_2}(X_1 \times \widehat{X}_2 , L(V_1) \otimes \mathcal{G}_1 \otimes \widehat{\mathcal{G}}_2)
\]
and
\[
\Phi(\gamma_1) : KR^*( X_1 \times \widehat{X}_2 , L(V_1) \otimes \mathcal{G}_1 \otimes \widehat{\mathcal{G}}_2 ) \to KR^{*-m_1}(\widehat{X}_1 \times \widehat{X}_2 , \widehat{\mathcal{G}}_1 \otimes \widehat{\mathcal{G}}_2 )
\]
where $m_i = dim(X_i)$.

\begin{lemma}\label{lem:fmfactor}
We have $\Phi(\gamma_1 \otimes \gamma_2 ) = \Phi(\gamma_1 ) \circ \Phi(\gamma_2)$.
\end{lemma}
\begin{proof}
The proof is a fairly straightward calculation. The only difficulty is keeping track of all the different maps involved. For this purpose, we introduce the following notation. Let $X_{1^{\vee}} = \widehat{X}_1$, $X_{2^{\vee}} = \widehat{X}_2$. If $I \subseteq \{ 1 , 2 , 1^{\vee} , 2^{\vee} \}$, let $X_I = \prod_{i \in I} X_i$. If $J \subseteq I \subseteq \{1,2,1^{\vee}, 2^{\vee} \}$, let $\pi^I_J : X_I \to X_J$ be the projection map. Then we have:
\begin{align*}
& \Phi(\gamma_1 \otimes \gamma_2)(x) \\
&\quad =(\pi^{1 1^{\vee} 2 2^{\vee} }_{1^\vee 2^\vee})_*(  (\pi^{11^{\vee}22^{\vee}}_{11^{\vee}})^*(\gamma_1) \otimes (\pi^{11^{\vee}22^{\vee}}_{2 2^{\vee}})^*(\gamma_2) \otimes (\pi^{121^{\vee} 2^{\vee}}_{12})^*(x) ) \\
&\quad = (\pi^{11^{\vee} 2^{\vee}}_{1^\vee 2^\vee})_* (\pi^{1 1^{\vee} 2 2^{\vee}}_{1 1^{\vee} 2^{\vee}})_*( (\pi^{11^{\vee}22^{\vee}}_{11^{\vee}})^*(\gamma_1) \otimes (\pi^{11^{\vee}22^{\vee}}_{2 2^{\vee}})^*(\gamma_2) \otimes (\pi^{11^\vee 22^\vee}_{12 2^\vee})^*( \pi^{12 2^\vee}_{12})^*(x) ) \\
&\quad = (\pi^{11^{\vee} 2^{\vee}}_{1^\vee 2^\vee})_* ( (\pi^{11^\vee 2^\vee}_{11^\vee})^*(\gamma_1) \otimes (\pi^{1 1^{\vee} 2 2^{\vee}}_{1 1^{\vee} 2^{\vee}})_*(  (\pi^{11^{\vee}22^{\vee}}_{2 2^{\vee}})^*(\gamma_2) \otimes (\pi^{11^\vee 22^\vee}_{1 2 2^\vee})^*( \pi^{12 2^\vee}_{12})^*(x) ) ).
\end{align*}

Next, we have
\begin{align*}
& (\pi^{1 1^{\vee} 2 2^{\vee}}_{1 1^{\vee} 2^{\vee}})_*(  (\pi^{11^{\vee}22^{\vee}}_{2 2^{\vee}})^*(\gamma_2) \otimes (\pi^{11^\vee 22^\vee}_{12 2^\vee})^*( \pi^{12 2^\vee}_{12})^*(x) )  \\
&= (\pi^{1 1^{\vee} 2 2^{\vee}}_{1 1^{\vee} 2^{\vee}})_*(\pi^{11^\vee 2 2^\vee}_{122^\vee})^*(   (\pi^{122^{\vee}}_{2 2^{\vee}})^*(\gamma_2) \otimes (\pi^{122^\vee}_{12})^*(x) ) \\
&= (\pi^{11^\vee 2^\vee}_{1 2^\vee})^* (\pi^{12 2^\vee}_{1 2^\vee})_*(   (\pi^{122^{\vee}}_{2 2^{\vee}})^*(\gamma_2) \otimes (\pi^{122^\vee}_{12})^*(x) ) \\
&=  (\pi^{11^\vee 2^\vee}_{1 2^\vee})^* \Phi(\gamma_2)(x).
\end{align*}

Hence
\begin{align*}
\Phi(\gamma_1 \otimes \gamma_2)(x) &= (\pi^{11^\vee 2^\vee}_{1^\vee 2^\vee})_*(  (\pi^{11^\vee 2^\vee}_{11^\vee})^*(\gamma_1) \otimes ( \pi^{11^\vee 2^\vee}_{1 2^\vee})^*( \Phi(\gamma_2)(x) ) ) \\
&= \Phi(\gamma_1)(( \Phi(\gamma_2)(x))).
\end{align*}

\end{proof}

As a consequence of Lemma \ref{lem:fmfactor} and Equation (\ref{equ:fmcomp}), we have that $\Phi(\gamma)$ is an isomorphism if and only of $\Phi(\gamma_1)$ and $\Phi(\gamma_2)$ are ismorphisms.

We have just demonstrated that if a Real affine torus $X$ decomposes as a product $X = X_1 \times X_2$ then the Fourier--Mukai transform $\Phi(\gamma)$ associated to any $T$-dual tuple $(X , \mathcal{G} , \widehat{X} , \widehat{\mathcal{G}} , \gamma)$ can be expressed in terms of isomorphisms $t_N, t_{\widehat{N}}$ and partial Fourier--Mukai transforms $\Phi(\gamma_1), \Phi(\gamma_2)$. Applying the above argument iteratively, if we decompose $X$ into indecomposables $X = X_1 \times X_2 \times \cdots \times X_k$, then $\Phi(\gamma)$ can be expressed as a composition of isomorphisms (of the form $t_N$ for some Real line bundles) and the partial Fourier--Mukai transforms $\Phi(\gamma_1) , \dots , \Phi(\gamma_2)$. In particular, if $\Phi(\gamma_i)$ is an isomorphism for each $i$, then $\Phi(\gamma)$ is an isomorphism. Each of the partial Fourier--Mukai transforms $\Phi(\gamma_i)$ is a special case of the Fourier--Mukai transform for product torus bundles $X_i \times B$, $\widehat{X}_i \times B$ over a base space $B$.

Let $(B , \sigma_B)$ be a smooth manifold with smooth involution $\sigma_B$. We will say $B$ is {\em RFMI} if the Real Fourier--Mukai transform is an isomorphisms for all affine torus bundles over $B$ of the form $(X , \sigma_X) = (T , \sigma_T) \times (B , \sigma_B)$, where $(T , \sigma_T)$ is an indecomposable Real torus. By the discussion above, we have proven that if every compact smooth $(B , \sigma_B)$ is RMFI, then the Real Fourier--Mukai transform is an isomorphism whenever the base space is a point. Hence by Lemma \ref{lem:redpt}, we will have shown that the Fourier--Mukai transform is an isomorphism for any compact smooth base space $B$ and any smooth involution $\sigma_B$.

The next lemma reduces us to the problem of showing $\{ pt\}$ is RFMI. That is, that the Real Fourier--Mukai transform is an isomorphism in the case of indecomposable tori.

\begin{lemma}
Suppose that $\{ pt \}$ is RFMI. Then $(B , \sigma)$ is RFMI for any compact smooth manifold $B$ and any smooth involution $\sigma$.
\end{lemma}
\begin{proof}
This follows by an identical argument to the proof of Lemma \ref{lem:redpt} given in Section \ref{sec:redtor}.
\end{proof}

\subsection{Indecomposable tori}\label{sec:indtor}

In this section we restrict to the case that $B = \{pt\}$ and $X$ is indecomposable. By the results of Section \ref{sec:gentor}, if the Real Fourier--Mukai transform is an isomorphism in this case then it is an isomorphism for all affine torus bundles over any compact smooth base and so the proof of Theorem \ref{thm:rfm} will be complete.

Assume that $X$ is an indecomposable Real torus. More precisely, we have $X = T = V/\Lambda$, $\Gamma = \mathbb{Z}_2 = \langle \sigma \rangle$. Then $\rho : \mathbb{Z}_2 \to GL(\Lambda)$ makes $\Lambda$ into a $\mathbb{Z}_2$-module. We have that $\rho$ is indecomposable. As discussed in Example \ref{ex:Bpt}, there are three possible cases for $\rho$: 
\begin{itemize}
\item[(1)]{Trivial: $\Lambda = \mathbb{Z}$, $\rho(x) = x$.}
\item[(2)]{Cyclotomic: $\Lambda = \mathbb{Z}$, $\rho(x) = -x$.}
\item[(3)]{Regular: $\Lambda = \mathbb{Z}^2$, $\rho(x,y) = (y,x)$.}
\end{itemize}

Let $\rho^{\vee}$ be the $T$-dual monodromy representation on $\Lambda^*$, namely $\rho^{\vee}(x) = -\rho^*(x)$. Observe that $\rho \mapsto \rho$ exchanges the trivial and cyclotomic cases and sends the regular case to itself.

Let $c \in H^2_{\mathbb{Z}_2}(pt ; \Lambda_\rho)$ denote the first Chern class of $X$. As in Example \ref{ex:Bpt}, we have that $c=0$ in the cyclotomic and regular cases. In the trivial case $
\Lambda_\rho = \mathbb{Z}$, we have $H^2_{\mathbb{Z}_2}(pt ; \Lambda_\rho) \cong \mathbb{Z}_2$, so there are two choices depending on whether $c$ is zero or non-zero.

By assumption $\mathcal{G}$ is an affine gerbe, so its Dixmier--Douady class lifts to $H^0_{\mathbb{Z}_2}(pt ; \mathbb{Z}_2) \oplus H^2_{\mathbb{Z}_2}(pt ; \mathcal{A}(X)_- )$. The first summand $H^0_{\mathbb{Z}_2}(pt ; \mathbb{Z}_2) \cong \mathbb{Z}_2$ is the grading class of the $\mathcal{G}$. Twisting $KR$-theory by this class just results in a grading shift. The second term $H^2_{\mathbb{Z}_2}(pt ; \mathcal{A}(X)_-)$ is more interesting. We calculate this group for the four possible cases of $(\rho , c)$. First of all note that any affine $S^1$-valued function on $X = V/\Lambda$ can be identified with a pair $( \lambda , u ) \in \Lambda^* \times S^1$ via $(\lambda , u ) \leftrightarrow f$, where $f(t) = \lambda(t) + u$. To compute $H^2_{\mathbb{Z}_2}(pt ; \mathcal{A}(X)_-)$, we note that in general for a $\mathbb{Z}_2$-module $A$, we have
\[
H^2_{\mathbb{Z}_2}(pt ; A) = \frac{\{ a \in A \; | \; \sigma(a) = a \}}{\{ b+\sigma(b) \; | \; b \in A\} }.
\]

{\bf Case 1: $\Lambda_\rho = \mathbb{Z}$, $c=0$.} Then $\sigma : X \to X$ is the identity map. The $\mathbb{Z}_2$-action on $\mathcal{A}(X)_- \cong \mathbb{Z} \times S^1$ is $\sigma(\lambda , u) = (-\lambda , -u)$. Hence $H^2_{\mathbb{Z}_2}(pt ; \mathcal{A}(X)_-) \cong \mathbb{Z}_2$. Moreover, the pullback map 
\[
\pi^* : H^2_{\mathbb{Z}_2}( pt ; S^1_-) \to H^2_{\mathbb{Z}_2}( pt ; \mathcal{A}(X)_-)
\]
is an isomorphism. Hence $\mathcal{G}$ is the pullback of a Real gerbe on $B = \{pt\}$.

{\bf Case 2: $\Lambda_\rho = \mathbb{Z}$, $c \neq 0$.} In this case $X = \mathbb{R}/\mathbb{Z}$ with $\sigma(x) = x+1/2$. The $\mathbb{Z}_2$-action on $\mathcal{A}(X)_- \cong \mathbb{Z} \times S^1$ is $\sigma(\lambda , u ) = (-\lambda , -u - \lambda/2)$. One finds that $H^2_{\mathbb{Z}_2}(pt ; \mathcal{A}(X)_-) = 0$.

{\bf Case 3: $\Lambda_\rho = \mathbb{Z}_-$.} In this case $X = \mathbb{R}/\mathbb{Z}$ with $\sigma(x) = -x$. The $\mathbb{Z}_2$-action on $\mathcal{A}(X)_- \cong \mathbb{Z} \times S^1$ is $\sigma(\lambda , u ) = (\lambda , -u)$. One finds that $H^2_{\mathbb{Z}_2}(pt ; \mathcal{A}(X)_-) \cong \mathbb{Z}_2 \oplus \mathbb{Z}_2$. Observe that the action of $\sigma$ on $X$ has two fixed points $x_0 = 0$, $x_1 = 1/2$. By restricting to these points we get a homomorphism
\[
H^3_{\mathbb{Z}_2}(X ; \mathbb{Z}_-) \to H^3_{\mathbb{Z}_2}(x_0 ; \mathbb{Z}_-) \oplus H^3_{\mathbb{Z}_2}(x_1 ; \mathbb{Z}_-) \cong \mathbb{Z}_2 \oplus \mathbb{Z}_2.
\]
One finds that this map is an isomorphism, using localisation in equivariant cohomology. The map $H^2_{\mathbb{Z}_2}(pt ; \mathcal{A}(X)_-) \to H^3_{\mathbb{Z}_2}(X ; \mathbb{Z}_-)$ from affine Real gerbes to ordinary Real gerbes is also seen to be an isomorphism. Therefore, the isomorphism class of $\mathcal{G}$ is detected by its restrictions $\mathcal{G}|_{x_0}$, $\mathcal{G}|_{x_1}$. If $\mathcal{G}|_{x_0} \cong \mathcal{G}|_{x_1}$, then $\mathcal{G}$ is the pullback of a class in $H^3_{\mathbb{Z}_2}(pt ; \mathbb{Z}_-)$.

{\bf Case 4: $\Lambda_\rho = R$, the regular representation.} In this case $X = \mathbb{R}^2/\mathbb{Z}^2$ with $\sigma(x,y) = (y,x)$. One finds that $\mathcal{A}(X)_- \cong R \times S^1_-$ as $\mathbb{Z}_2$-modules and hence $H^2_{\mathbb{Z}_2}( pt ; \mathcal{A}(X)_-) \cong H^2_{\mathbb{Z}_2}(pt ; S^1_-) \cong \mathbb{Z}_2$. This also implies that $\mathcal{G}$ is a pullback from $pt$.

By a {\em point gerbe} on $X$, we mean a Real graded gerbe on $X$ obtained by pulling back a Real graded gerbe $\mathcal{H}$ over a point $\{pt\}$ under the map $X \to \{pt\}$. Let us say that two affine gerbes $\mathcal{G}, \mathcal{G}'$ on $X$ are {\em equivalent modulo point gerbes} if $\mathcal{G}' \cong \mathcal{G} \otimes \mathcal{H}$, where $\mathcal{H}$ is a point gerbe. If $(\widehat{X} , \widehat{\mathcal{G}})$ is Real $T$-dual to $(X , \mathcal{G})$, then it follows easily that $(\widehat{X} , \widehat{\mathcal{G}} \otimes \mathcal{H})$ is Real $T$-dual to $(X , \mathcal{G}')$. Moreover, the effect of twisting by a point gerbe $\mathcal{H}$ on $KR$-theory is just a grading shift. It follows that if the Real Fourier--Mukai transform for the pair $(X , \mathcal{G}), (\widehat{X} , \widehat{\mathcal{G}})$ is an isomorphism, then it is also an isomorphism for the pair $(X , \mathcal{G}') , (\widehat{X} , \widehat{\mathcal{G}} \otimes \mathcal{H})$. Therefore, it is enough to check that the Real Fourier--Mukai transform is an isomorphism for one particular representative of an equivalence class of affine gerbes modulo point gerbes. Looking again at Cases 1-4, we see that in Cases 1,2 and 4, $\mathcal{G}$ is equivalent modulo point gerbes to a trivial gerbe, whereas in Case 3 there are two equivalence classes: the trivial case $\mathcal{G}|_{x_0} \cong \mathcal{G}|_{x_1}$ and the non-trivial case $\mathcal{G}|_{x_0} \ncong \mathcal{G}|_{x_1}$. Also the gradings on $\mathcal{G}$ and $\widehat{\mathcal{G}}$ must be equal to each other and must be a pullback from $pt$. So modulo point gerbes, we can assume that $\mathcal{G}$ and $\widehat{\mathcal{G}}$ are trivially graded. It follows that modulo point gerbes there are exactly five types:
\begin{itemize}
\item[($T_1$)]{$X = \mathbb{R}/\mathbb{Z}$, $\sigma(x) = x$, $\mathcal{G}$ is trivial.}
\item[($T_2$)]{$X = \mathbb{R}/\mathbb{Z}$, $\sigma(x) = x+1/2$, $\mathcal{G}$ is trivial.}
\item[($T_3$)]{$X = \mathbb{R}/\mathbb{Z}$, $\sigma(x) = -x$, $\mathcal{G}$ is trivial.}
\item[($T_4$)]{$X = \mathbb{R}/\mathbb{Z}$, $\sigma(x) = -x$, $\mathcal{G}|_{x_0} \ncong \mathcal{G}|_{x_1}$.}
\item[($T_5$)]{$X = \mathbb{R}^2/\mathbb{Z}^2$, $\sigma(x,y) = (y,x)$, $\mathcal{G}$ is trivial.}
\end{itemize}

Let us work out which pairs are Real $T$-duals. By fixing $X$ and $\mathcal{G}$, the isomorphism class of $\widehat{X}$ is completely determined and $\widehat{\mathcal{G}}$ is determined up to tensoring with a point gerbe (with trivial grading). If $\sigma$ has fixed points on both $X$ and $\widehat{X}$, then the pullback $H^3_{\mathbb{Z}_2}(pt ; \mathbb{Z}_-) \to H^3_{\mathbb{Z}_2}(X \times_B \widehat{X} ; \mathbb{Z}_-)$ is injective and therefore the existence of an isomorphism $\gamma : p^*(\mathcal{G})  \to \widehat{p}^*(\widehat{\mathcal{G}})$ implies that $\widehat{\mathcal{G}}$ is uniqely determined by $(X , \mathcal{G})$. On the other hand, if $\sigma$ acts freely on $X$ or on $\widehat{X}$, then either $H^3_{\mathbb{Z}_2}(X ; \mathbb{Z}_-) = 0$ or $H^3_{\mathbb{Z}_2}(\widehat{X} ; \mathbb{Z}_-) = 0$. It follows that the pullback map $H^3_{\mathbb{Z}_2}(pt ; \mathbb{Z}_-) \to H^3_{\mathbb{Z}_2}(X \times_B \widehat{X} ; \mathbb{Z}_-)$ is zero, because it factors through $H^3_{\mathbb{Z}_2}(X ; \mathbb{Z}_-)$ and also through $H^3_{\mathbb{Z}_2}(\widehat{X} ; \mathbb{Z}_-)$. Hence in this case, if $(\widehat{X} , \widehat{\mathcal{G}})$ is Real $T$-dual to $(X , \mathcal{G})$, then so is $(\widehat{X} , \widehat{\mathcal{G}} \otimes \widehat{\pi}^*(\mathcal{H}))$ for any point gerbe $\mathcal{H}$ (with trivial grading).

From the above discussion it follows that, modulo point gerbes and modulo swapping the roles of $(X , \mathcal{G})$ and $(\widehat{X} , \widehat{\mathcal{G}})$, the Real $T$-dual pairs are:
\begin{itemize}
\item[(1)]{$X = \mathbb{R}/\mathbb{Z}$, $\sigma(x) = x$, $\widehat{X} = \mathbb{R}/\mathbb{Z}$, $\sigma(x) = -x$, $\mathcal{G}$ and $\widehat{\mathcal{G}}$ are trivial.}
\item[(2)]{$X = \mathbb{R}/\mathbb{Z}$, $\sigma(x) = x+1/2$, $\mathcal{G}$ is trivial. $\widehat{X} = \mathbb{R}/\mathbb{Z}$, $\sigma(x) = -x$, $\widehat{\mathcal{G}}$ is any Real gerbe with trivial grading and with $\mathcal{G}|_{x_0} \ncong \mathcal{G}|_{x_1}$.}
\item[(3)]{$X = \widehat{X} = \mathbb{R}^2/\mathbb{Z}^2$, $\sigma(x,y) = (y,x)$, $\mathcal{G}$ and $\widehat{\mathcal{G}}$ are trivial.}
\end{itemize}

Now let $(X,\mathcal{G})$ be one of the indecomposable types $T_1$-$T_5$, let $(\widehat{X} , \widehat{\mathcal{G}})$ be a Real $T$-dual and let $\gamma : p^*(\mathcal{G}) \to \widehat{p}^*(\widehat{\mathcal{G}})$ be an isomorphism satisfying the Poincar\'e property. Reversing the roles of $X$ and $\widehat{X}$, we have that $\gamma^{-1} : \widehat{p}^*(\widehat{\mathcal{G}}) \to p^*(\mathcal{G})$ also satisfies the Poincar\'e property and so we get a pair of Fourier--Mukai transforms
\[
\Phi : KR^*(X , \mathcal{G} ) \to KR^{*-n}(\widehat{X} , L(V_\rho)^{-1} \otimes \widehat{\mathcal{G}})
\]
and
\[
\Phi' : KR^{*-n}(\widehat{X} , L(V_\rho)^{-1} \otimes \widehat{\mathcal{G}}) \to KR^{*-2n}(X , L(V_\rho)^{-1} \otimes L(V_{\rho^{\vee}})^{-1} \otimes \mathcal{G})
\]
where $n = dim(X)$. We observe that in all cases $T_1$-$T_5$, we have $L(V_{\rho}) \otimes L(V_{\rho^\vee}) \cong L(V_\rho \oplus V_{\rho^{\vee}}) \cong L( \mathbb{R}^n \oplus \mathbb{R}^n_{-})$. Hence
\[
KR^{*-2n}(X , L(V_\rho)^{-1} \otimes L(V_{\rho^{\vee}})^{-1} \otimes \mathcal{G}) \cong KR^{*-2n}(X , L(\mathbb{R}^n \oplus \mathbb{R}^n_{-}) \otimes \mathcal{G} ) \cong KR^*(X , \mathcal{G}).
\]
Therefore the composition of $\Psi = \Phi' \Phi$ is an endomorphism
\[
\Psi : KR^*(X , \mathcal{G}) \to KR^*(X , \mathcal{G}).
\]
In a similar manner, the composition $\Psi' = \Phi \Phi'$ is an endomorphism
\[
\Psi' : KR^*(\widehat{X} , L(V_\rho)^{-1} \otimes \widehat{\mathcal{G}}) \to KR^*(\widehat{X} , L(V_\rho)^{-1} \otimes \widehat{\mathcal{G}}).
\]
Since $L(V_\rho)^{-1}$ is a point gerbe, it just alters the twisted $KR$-theory by a degree shift. So we can identify $\Psi$ with an endomorphism
\[
\Psi' : KR^*(\widehat{X} ,  \widehat{\mathcal{G}}) \to KR^*(\widehat{X} ,  \widehat{\mathcal{G}}).
\]
We will show that $\Psi, \Psi'$ are isomorphisms. This implies that $\Phi$ and $\Phi'$ are isomorphisms: injectivity of $\Psi$ implies injectivity of $\Phi$. Surjectivity of $\Psi'$ implies surjectivity of $\Phi$, so $\Phi$ is an isomorphism. Similarly $\Phi'$ is an isomorphism. Furthermore, since exchanging the roles of $(X , \mathcal{G})$ and $(\widehat{X} , \widehat{\mathcal{G}})$ exchanges $\Psi$ and $\Psi'$, we just need to show that $\Psi$ is an isomorphism in each of the five cases $T_1-T_5$.

To show that $\Psi : KR^*(X , \mathcal{G}) \to KR^*(X , \mathcal{G})$ is an isomorphism, we will make use of two key properties. First of all $\Psi$ is a morphism of $KR^*(pt)$-modules. This is clear because $\Psi$ is the composition of Fourier--Mukai transforms $\Phi,\Phi'$ which themselves are compositions of pullbacks, pushforwards and gerbe isomorphisms. Each of these operations are morphisms of $KR^*(pt)$-modules. The second key property of $\Psi$ is compatibility with the forgetful map $\varrho : KR^*(X , \mathcal{G}) \to K^*(X , \mathcal{G})$ to complex $K$-theory. Namely, we have a commutative diagram
\[
\xymatrix{
KR^*(X , \mathcal{G}) \ar[r]^-{\Psi} \ar[d]^-{\varrho} & KR^*(X , \mathcal{G}) \ar[d]^-{\varrho} \\
K^*(X , \mathcal{G}) \ar[r]^-{\Psi_K} & K^*(X , \mathcal{G})
}
\]
where $\Psi_K : K^*(X , \mathcal{G}) \to K^*(X , \mathcal{G})$ is defined identically to $\Psi$ but using complex $K$-theory. Note that $\mathcal{G}$ is trivial as a gerbe without Real structure, so that $K^*(X , \mathcal{G})$ is isomorphic to the untwisted $K$-theory of $X$. Since the Fourier--Mukai transform is known to be an isomorphism in complex $K$-theory, we have that $\Psi_K$ is an isomorphism.

We now proceed to the verification that $\Psi$ is an isomorphism in types $T_1$-$T_5$. In the calculations that follow, we will regard $KR$-theory as being $\mathbb{Z}_8$-graded, that is $KR^{j+8} = KR^j$ is an equality, not just an isomorphism. Let $R^* = KR^*(pt)$ be the $KR$-theory of a point. Then $R^* \cong \mathbb{Z}[\eta , h]/(2\eta , \eta^2 , \eta h, h^2 - 4)$, where $deg(\eta) = -1$, $deg(h) = 4$. Let $R_K^* = K^*(pt)$ regarded as $8$-periodic. Thus $R_K^* \cong \mathbb{Z}[\xi]/(\xi^4-1)$, where $deg(\xi) = -2$. The forgetful map $\varrho : R^* \to R^*_K$ makes $R^*_K$ into an $R^*$-module and is given by $\varrho(\eta) = 0$, $\varrho(h) = 2\xi^2$.

{\bf Type $T_1$.} $X = S^1$ with trivial involution and trivial gerbe. From the Thom isomorphism, we have $KR^*(X) \cong R^* \oplus R^* \tau$ where $deg(\tau) = 1$. Thus $KR^*(X)$ is a free $R^*$-module with basis $1,\tau$. Similarly $K^*(X) \cong R^*_K \oplus R^*_K \tau_K$ is a free $R^*_K$ with basis $1,\tau_K$, $deg(\tau_K) = 1$. The forgetful map is given by $\varrho(1) = 1$, $\varrho(\tau) = \tau_K$. Since $\Psi_K$ is an isomorphism, we have $\Psi_K(\tau_K) = \pm \tau_K$. Then the compatibility $\varrho \Psi = \Psi_K \varrho$ implies that $\Psi(\tau) = \pm \tau$. Since $KR^0(X) \cong \mathbb{Z} \oplus \mathbb{Z}_2 (\eta \tau)$, we have $\Psi(1) = a + b \eta \tau$ for some $a \in \mathbb{Z}$, $b \in \mathbb{Z}_2$. The compatibility of $\Psi$ and $\Psi_K$ gives $a = \pm 1$, hence $\Psi(1), \Psi(\tau)$ is a basis for $KR^*(X)$ as an $R^*$-module and thus $\Psi$ is an isomorphism.

{\bf Type $T_2$.} $X = S^1 = \mathbb{R}/\mathbb{Z}$ with involution $\sigma(x) = x+1/2$ and with trivial gerbe. Regard $S^1$ as the boundary of the closed unit disc $D \subset \mathbb{C}$ equipped with the involution $\sigma(z) = -z$. Then we have a long exact sequence associated to the pair $(D ,  X )$:
\[
\cdots \to KR^i(D,X) \to KR^i(D) \to KR^*(X) \to \cdots
\]
The Thom isomorphism gives $KR^i(D,X) \cong R^* \tau$ where $deg(\tau) = -2$ (the degree is $-2$ instead of $+2$ because $\sigma$ acts as $-1$ on $D$). Contractibility of $D$ implies $KR^*(D) \cong KR^*(pt)$. Thus the long exact sequence takes the form
\[
\cdots \to KR^i(pt) \to KR^i(X) \to KR^{i+3}(pt) \to \cdots
\]
From this we find $KR^0(X) \cong \mathbb{Z}$, $KR^1(X) \cong \mathbb{Z}$, $KR^2(X) \cong 0$, $KR^3(X) \cong \mathbb{Z}_2$. Furthermore, $X$ admits a Quaternionic line bundle, namely the trivial line bundle $S^1 \times \mathbb{C}$ equipped with the Quaternionic structure $(z , t ) \mapsto (-z , \overline{z}t)$ (here $S^1$ is taken to be the unit circle in $\mathbb{C}$). This gives a class $q \in KR^4(X)$ with the property that $q^2 = 1$. Hence tensoring by $q$ gives an isomorphism $KR^{i}(X) \cong KR^{i+4}(X)$ for all $i$. Let $1 \in KR^0(X)$ denote the identity element and let $\omega \in KR^1(X) \cong \mathbb{Z}$ be a generator. Then our calculations show that
\begin{align*}
KR^0(X) &\cong \mathbb{Z}1 &  KR^4(X) &\cong \mathbb{Z}q \\
KR^1(X) &\cong \mathbb{Z}\omega & KR^5(X) &\cong \mathbb{Z} q\omega \\
KR^2(X) &\cong 0 & KR^6(X) &\cong 0 \\
KR^3(X) &\cong \mathbb{Z}_2 \eta q & KR^7(X) &\cong \mathbb{Z}_2 \eta
\end{align*}

Since the forgetful map $\varrho : KR^{-2}(D , X) \to K^{-2}(D,X) \cong K^2(D,X)$ sends the Thom class $\tau \in KR^{-2}(D,X)$ to $\xi^2$ times the Thom class $\tau_K \in K^2(D,X)$, we have $\varrho(\tau) = \xi^2 \tau_K$. Then compatibility of the forgetful maps from $KR$ $K$ with the long exact sequences for $(D,X)$ implies that $\varrho(\omega)$ is a generator of $K^1(X) \cong \mathbb{Z}$. Let us denote this generator by $\omega_K$. Then the complex $K$-theory of $X$ is given by 
\begin{align*}
K^0(X) &\cong \mathbb{Z}1 &  K^4(X) &\cong \mathbb{Z}\xi^2 \\
K^1(X) &\cong \mathbb{Z}\omega_K & K^5(X) &\cong \mathbb{Z} \xi^2 \omega_K \\
K^2(X) &\cong \mathbb{Z}\xi^3 & K^6(X) &\cong \mathbb{Z} \xi \\
K^3(X) &\cong \mathbb{Z} \xi^3 \omega_K & K^7(X) &\cong \mathbb{Z} \xi \omega_K
\end{align*}

Since $h^2 = 4$ and $q^1 = 1$, it follows that $h = 2q$. Then since $\varrho(h) = 2\xi^2$, we must have $\varrho(q) = \xi^2$. So $\varrho(1) = 1$, $\varrho(\omega) = \omega_K$, $\varrho(q) = \xi^2$, $\varrho(q \omega) = \xi^2 \omega_K$. The compatibility relation $\varrho \Psi = \Psi_K \varrho$ together with the fact that $\Psi_K$ is an isomorphism implies that $\Psi(1) = \pm 1$, $\Psi(\omega) = \pm \omega$, $\Psi(q) = \pm q$, $\Psi( q \omega ) = \pm q \omega$. Lastly, since $\Psi$ is a morphism of $R^*$-modules, $\Psi(\eta) = \eta \Psi(1) = \eta$ and $\Psi(\eta q) = \eta \Psi(q) = \eta q$. Thus $\Psi$ is an isomorphism.

{\bf Type $T_3$.} $X = S^1 = \mathbb{R}/\mathbb{Z}$, $\sigma(x) = -x$ with trivial gerbe. The Thom isomorphism gives $KR^*(X) \cong R^* \oplus R^* \tau_-$, where $deg(\tau_-) = -1$. From here the argument is essential the same as the case $T_1$. Compatibility with $\Psi_K$ implies that $\Psi(1) = \pm 1$, $\Psi(\tau_-) = \pm \tau_- + b \eta$ for some $b \in \mathbb{Z}_2$. Thus $\Psi$ is an isomorphism.

{\bf Type $T_4$.} $X = S^1 = \mathbb{R}/\mathbb{Z}$, $\sigma(x) = -x$ with gerbe $\mathcal{G}$ such that $\mathcal{G}$ has trivial grading and $\mathcal{G}|_{x_0} \ncong \mathcal{G}|_{x_1}$, where $x_0 = 0,x_1 = 1/2$ are the fixed points of $\sigma$. Set $A = \{x_0,x_1\}$. Then we have the long exact sequence of the pair $(X , A)$:
\[
\cdots \to KR^*(X,A,\mathcal{G}) \to KR^*(X,\mathcal{G}) \to KR^*(A , \mathcal{G}) \to \cdots
\]
Let $U$ be the image of $(-\epsilon , \epsilon) \cup (1/2 - \epsilon , 1/2 + \epsilon)$ in $X = \mathbb{R}/\mathbb{Z}$ where $\epsilon$ is sufficiently small. Excision gives
\begin{align*}
KR^*(X,A,\mathcal{G}) &\cong KR^*(X , U , \mathcal{G}) \\
& \cong KR^*(X \setminus A , U \setminus A , \mathcal{G}) \\
& \cong K^*( [\epsilon , 1/2 - \epsilon] , \{\epsilon , 1/2 - \epsilon\} , \mathcal{G}) \\
& \cong K^*( [0,1] , \{0,1\} ) \\
& \cong \widetilde{K}^*(S^1) \\
& \cong K^*(pt)\tau
\end{align*}
where $deg(\tau) = 1$. Since $\mathcal{G}$ is trivial on one of $x_0,x_1$ and non-trivial on the other, we have
\[
KR^*(A , \mathcal{G}) \cong KR^*(pt)\alpha \oplus KR^{*}(pt) \beta
\]
where $deg(\alpha) = 0$, $deg(\beta) = 4$. Comparing with ordinary $K$-theory, we see that $\delta : KR^*(A,\mathcal{G}) \to KR^{*+1}(X,A , \mathcal{G})$ sends $\alpha$ to $\pm \tau$. From this one deduces $KR^0(X,\mathcal{G}) \cong \mathbb{Z}, KR^1(X ,\mathcal{G}) \cong 0$, $KR^2(X , \mathcal{G}) \cong \mathbb{Z}_2$ and there is a short exact sequence $0 \to \mathbb{Z} \to KR^3(X , \mathcal{G}) \to \mathbb{Z}_2 \to 0$, hence $KR^3(X , \mathcal{G}) \cong \mathbb{Z}$ or $\mathbb{Z} \oplus \mathbb{Z}_2$. Next we observe that $KR^*(X , \mathcal{G})$ is $4$-periodic. To see this, let $f : X \to X$ be the map $f(x) = x+1/2$. Then $f$ commutes with $\sigma$ and $f$ swaps $x_0,x_1$. Hence $f^*(\mathcal{G}) \cong \mathcal{G} \otimes \mathcal{H}$, where $\mathcal{H}$ is the point gerbe corresponding to the generator of $H^3_{\mathbb{Z}_2}(pt ; \mathbb{Z}_-)$. Thus $f^*$ induces an isomorphism
\[
KR^*(X , \mathcal{G}) \cong KR^*(X , \mathcal{G} \otimes \mathcal{H}) \cong KR^{*+4}(X , \mathcal{G}).
\]
We will prove that $KR^3(X , \mathcal{G}) \cong \mathbb{Z}$. Assume to the contrary that $KR^3(X , \mathcal{G}) \cong \mathbb{Z} \oplus \mathbb{Z}_2$. We we derive a contradiction by showing this is incompatible with the Real Fourier--Mukai transform. Let $(\widehat{X} , \widehat{\mathcal{G}})$ be a Real $T$-dual to $(X , \mathcal{G})$. So $(\widehat{X}) , \widehat{\mathcal{G}})$ is of type $T_2$, that is $\widehat{X} = \mathbb{R}/\mathbb{Z}$ with involution $\sigma(x) = x+1/2$ and with trivial gerbe. Then we have Fourier--Mukai transforms
\[
\Phi : KR^j( X , \mathcal{G}) \to KR^{j-1}(\widehat{X} , L(\mathbb{R}_{-})^{-1} \otimes \widehat{\mathcal{G}}) \cong KR^{j+1}(\widehat{X}  , \widehat{\mathcal{G}})
\]
and
\[
\Phi' : KR^j(\widehat{X} , \widehat{\mathcal{G}}) \to KR^{j-1}(X , \mathcal{G}).
\]
Furthermore, we have already established that $\Psi' = \Phi \Phi' : KR^*(\widehat{X} , \widehat{\mathcal{G}}) \to KR^*(\widehat{X} , \widehat{\mathcal{G}})$ is an isomorphism when we studied the $T_2$ case. Consider the commutative diagram
\[
\xymatrix{
\mathbb{Z} \oplus \mathbb{Z}_2 \cong KR^3(X , \mathcal{G}) \ar[r]^-{\Phi} \ar[d]^-{\eta} & KR^4(\widehat{X} , \widehat{\mathcal{G}}) \cong \mathbb{Z} \ar[d]^-{\eta} \\
\mathbb{Z}_2 \cong KR^2(X , \mathcal{G}) \ar[r]^-{\Phi} & KR^3(\widehat{X} , \widehat{\mathcal{G}}) \cong \mathbb{Z}_2
}
\]
Since $\Psi' = \Phi \Phi'$ is an isomorphism, it follows that $\Phi : KR^2(X , \mathcal{G}) \to KR^3(\widehat{X} , \widehat{\mathcal{G}})$ is an isomorphism. Let $e = (0,1) \in \mathbb{Z} \oplus \mathbb{Z}_2 \cong KR^3(X , \mathcal{G})$. Since $e$ is torsion, $\Phi(e) = 0$ and thus $\eta \Phi(e) = \Phi( \eta e) = 0$. But $\Phi : KR^2(X , \mathcal{G}) \to KR^3(\widehat{X} , \widehat{\mathcal{G}})$ is an isomorphism and so $\eta e = 0$. On the other hand, under the map $KR^3(X , \mathcal{G}) \to KR^3(A , \mathcal{G}) \cong \mathbb{Z}_2 \eta \beta$, we must have that $e$ maps to $\eta \beta$, because the kernel of  $KR^3(X , \mathcal{G}) \to KR^3(A , \mathcal{G})$ has no torsion. Then under the map $KR^4(X , \mathcal{G}) \to KR^4(A , \mathcal{G}) \cong \mathbb{Z} h\alpha \oplus \mathbb{Z}_2 \eta^2 \beta$, we have that $\eta e$ maps to $\eta^2 \beta \neq 0$, hence $\eta e \neq 0$. This is a contradiction and so the assumption that $KR^3(X , \mathcal{G}) \cong \mathbb{Z} \oplus \mathbb{Z}_2$ was false. So we must instead have $KR^3(X ,\mathcal{G}) \cong \mathbb{Z}$. From this we deduce that $KR^*(X ,\mathcal{G})$ has the form
\begin{align*}
KR^0(X,\mathcal{G}) &\cong \mathbb{Z} \chi_0 &  KR^4(X,\mathcal{G}) &\cong \mathbb{Z}\chi_4 \\
KR^1(X,\mathcal{G}) &\cong 0 & KR^5(X,\mathcal{G}) &\cong 0 \\
KR^2(X,\mathcal{G}) &\cong \mathbb{Z}_2 \eta \chi_3 & KR^6(X,\mathcal{G}) &\cong \mathbb{Z}_2 \eta \chi_7 \\
KR^3(X,\mathcal{G}) &\cong \mathbb{Z} \chi_3 & KR^7(X,\mathcal{G}) &\cong \mathbb{Z} \chi_7
\end{align*}
for some classes $\chi_0, \chi_4, \chi_3, \chi_7$. By comparing the long exact sequences in $KR$ and $K$-theory for $(X , A)$ one find that the forgetful map $\varrho : KR^*(X , \mathcal{G}) \to K^*(X , \mathcal{G}) \cong K^*(X)$ is injective modulo torsion. Compatibility of the Fourier--Mukai transforms in $KR$ and $K$-theory then implies that $\Psi : KR^*(X , \mathcal{G}) \to KR^*(X , \mathcal{G})$ is an isomorphism in degrees $0,1,3,4,5,7$. In particular, $\Psi(\chi_3) = \pm \chi_3$ and $\Psi(\chi_7) = \pm \chi_7$. Hence $\Psi(\eta \chi_3) = \eta \Psi(\chi_3) = \eta \chi_3$ and similarly $\Psi(\eta \chi_7) = \eta \chi_7$. Thus $\Psi$ is an isomorphism.

{\bf Type $T_5$.} We have $X = \mathbb{R}^2/\mathbb{Z}^2$, $\sigma(x,y) = (y,x)$ and the gerbe is trivial. Consider the standard cellular structure on $X$ with one $0$-cell, two $1$-cells (swapped by $\sigma$) and one $2$-cell. Let $X_0 \subset X_1 \subset X_2 = X$ be the skeleta. Then $X_0 = \{ x_0\}$ is a point, $X_1 = S^1 \vee S^1$ where $\sigma$ swaps the circles and $X_2 = X$. We use $x_0$ as a basepoint and set $\widetilde{KR}^i(X) = KR^i(X,x_0)$. The long exact sequence for $(X , \{x_0\})$ splits, so $KR^i(X) \cong \widetilde{KR}^i(X) \oplus KR^i(pt)$. Now to compute the reduced $KR$-theory of $X$ we consider the long exact sequence of the triple $(X_2 , X_1 , X_0)$:
\[
\cdots \buildrel \delta \over \longrightarrow KR^*(X , X_1) \to \widetilde{KR}^*(X) \to KR^*(X_1 , X_0) \buildrel \delta \over \longrightarrow \cdots
\]
By the Thom isomorphism, we have
\[
KR^*(X,X_1) \cong \widetilde{KR}^*(X/X_1) \cong \widetilde{KR}^*( (\mathbb{R} \oplus \mathbb{R}_-)^+ ) \cong R^* \tau,
\]
where $deg(\tau) = 0$. We also have 
\[
KR^*(X_1,X_0) = \widetilde{KR}^*( S^1 \vee S^1 , x_0) \cong \widetilde{K}^*(S^1) \cong R^*_K \mu
\]
where $deg(\mu) = 1$. So the long exact sequence can be re-written as
\[
\cdots \buildrel \delta \over \longrightarrow R^* \tau \to \widetilde{KR}^*(X) \to R^*_K \mu \buildrel \delta \over \longrightarrow \cdots
\]
Carrying out a similar computation in $K$-theory yields a long exact sequence
\[
\cdots\buildrel \delta \over \longrightarrow R^*_K \tau_K \to K^*(X) \to R^*_K \gamma_1 \oplus R^*_K \gamma_2 \buildrel \delta \over \longrightarrow \cdots
\]
where $deg(\tau_K) = 2$, $deg(\gamma_1) = \deg(\gamma_2) = 1$. However in this case we know $K^*(X)$ is isomorphic to $\mathbb{Z}$ in even degrees and $\mathbb{Z}^2$ in odd degrees, so the connecting homomorphisms $\delta$ must be zero. The forgetful map $\varrho$ from $KR$ to $K$-theory yields a commutative diagram
\[
\xymatrix{
\cdots \ar[r]^-{\delta} & R^* \tau \ar[r] \ar[d]^-{\varrho} & \widetilde{KR}^*(X) \ar[r] \ar[d]^-{\varrho} & R^*_K \mu \ar[r]^-{\delta} \ar[d]^-{\varrho} & \cdots \\
\cdots \ar[r]^-{0} & R_K^* \tau_K \ar[r] & \widetilde{K}^*(X) \ar[r] & R^*_K \gamma_1 \oplus R^*_K \gamma_2 \ar[r]^-{0} & \cdots
}
\]
Since $\varrho(\tau) = \xi \tau_K$, $\varrho(\mu) = \gamma_1 - \gamma_2$ (for suitably chosen generators $\tau_K,\gamma_1,\gamma_2$) one deduces that the coboundary maps $\delta$ are zero in $KR$-theory and thus we have a short exact sequence
\[
0 \to R^* \tau \to \widetilde{KR}^*(X) \to R^*_K \mu \to 0.
\]
Furthermore, since $R^*_K$ has no torsion, the short exact sequence splits (as abelian groups), giving
\begin{align*}
KR^0(X) &\cong \mathbb{Z}1 \oplus \mathbb{Z}\tau &  KR^4(X) &\cong \mathbb{Z}h \oplus \mathbb{Z}h\tau \\
KR^1(X) &\cong \mathbb{Z}\mu_1 & KR^5(X) &\cong \mathbb{Z}\mu_5 \\
KR^2(X) &\cong 0 & KR^6(X) &\cong \mathbb{Z}_2 \eta^2 \oplus \mathbb{Z}_2 \eta^2 \tau \\
KR^3(X) &\cong \mathbb{Z} \mu_3 & KR^7(X) &\cong \mathbb{Z}_2 \eta \oplus \mathbb{Z}_2 \eta \tau \oplus \mathbb{Z}\mu_7
\end{align*}
for some classes $\mu_1,\mu_3,\mu_5,\mu_7,\tau$. The action of $R^*$ on $1,\tau$ is evident from the generators given above. The action on the $\mu_i$ is less straightforward. Comparing with the complex $K$-theory of $X$, we see that the forgetful map $\varrho : KR^*(X) \to K^*(X)$ is injective modulo torsion. The Real Fourier--Mukai transform $\Psi : KR^*(X) \to KR^*(X)$ is a group homomorphism to sends torsion to torsion. So we get a commutative diagram
\[
\xymatrix{
0 \ar[r] & Tors( KR^*(X) ) \ar[d]^-{\Psi_{tors}} \ar[r] & KR^*(X) \ar[r] \ar[d]^-{\Psi} & KR^*(X)_{tf} \ar[r] \ar[d]^-{\Psi_{tf}} & 0 \\
0 \ar[r] & Tors( KR^*(X) ) \ar[r] & KR^*(X) \ar[r] & KR^*(X)_{tf} \ar[r] & 0
}
\]
where $Tors(KR^*(X))$ is the torsion subgroup of $KR^*(X)$ and $KR^*(X)_{tf}$ is the torsion-free quotient. Compatibility of $\Psi$ and $\Psi_K$ together with injectivity of $\varrho : KR^*(X)_{tf} \to K^*(X)$ implies that $\Psi_{tf}$ is an isomorphism. To show that $\Psi$ is an isomorphism, it remains only to show that $\Psi_{tors}$ is an isomorphism. Let $a,b,c,d \in \mathbb{Z}$ be such that $\Psi(1) = a + b\tau$, $\Psi(\tau) = c + d\tau$. Since $\Psi_{tf}$ is an isomorphism, it follows that $\Psi : KR^0(X) \to KR^0(X)$ is invertible and so $ad-bc = \pm 1$. Now $Tors(KR^*(X))$ is concentrated in degrees $6,7$. $Tors(KR^7(X)) = \mathbb{Z}_2 \eta \oplus \mathbb{Z}_2 \eta \tau$, $Tors(KR^6(X)) = \mathbb{Z}_2 \eta^2 \oplus \mathbb{Z}_2 \eta^2 \tau$. We have $\Psi(\eta) = \eta \Psi(1) = a\eta + b \eta \tau$, $\Psi(\eta \tau) = \eta \Psi(\tau) = c \eta + d \eta \tau$. Since $ad-bc = \pm 1$, we have that $\Psi_{tors} : Tors( KR^7(X)) \to Tors( KR^7(X))$ is invertible. The same follows for $\Psi_{tors}$ in degree $6$. Thus $\Psi_{tors}$ is an isomorphism and hence $\Psi$ itself is an isomorphism.

\section{The index of Real Dirac operators}\label{sec:indr}

In this section we consider an application of $KR$-theory and the Real Fourier--Mukai transform to the index theory of Real Dirac operators. We introduce the notion of Real spin$^c$-structures of type $k$ and show that the index of a spin$^c$-Dirac operator lifts from complex $K$-theory to Real $K$-theory in the presence of such a structure. We also study the families index for families of Real Dirac operators obtained by coupling to line bundles. We show in such cases the families index is given by the Real Fourier--Mukai transform.

Let $X$ be a topological space with a Real structure $\sigma$. Let $V \to X$ be a real rank $n$ orthogonal vector bundle. Suppose an involutive lift of $\sigma$ to $V$ is given. Let $P \to X$ denote the principal $O(n)$-frame bundle of $V$. The the action of $\sigma$ on $V$ induces an involution $\sigma : P \to P$ satisfying $\sigma(pg) = \sigma(p)g$ for all $p \in P$, $g \in O(n)$. An orientation on $V$ is equivalent to a reduction of structure of $P$ to a principal $SO(n)$-bundle. Alternatively an orientation can be viewed as a continuous function $\epsilon_P : P \to \mathbb{Z}/2\mathbb{Z}$ satisfying $\epsilon_P(pg) = \epsilon_P(p) + \epsilon(g)$, for all $p \in P$, $g \in O(n)$, where $\epsilon : O(n) \to \mathbb{Z}_2$ is the determinant homomorphism. This is the point of view that we will use. A $Pin^c$-structure on $V$ is a lift $\widetilde{P} \to P$ of the frame bundle of $V$ to a principal $Pin^c(n)$-bundle $\widetilde{P}$. A $Spin^c$-structure on $V$ may then be defined as a pair $( \widetilde{P} , \epsilon_P)$, where $\widetilde{P}$ is a $Pin^c$-structure and $\epsilon_P$ is an orientation.

\begin{definition}\label{def:spinck}
Let $X$ be a topological space with a Real structure $\sigma$. Let $V \to X$ be a real rank $n$ orthogonal vector bundle. Suppose an involutive lift of $\sigma$ to $V$ is given. A {\em Real spin$^c$-structure} of type $k \in \mathbb{Z}/4\mathbb{Z}$ consists of: a $Spin^c$-structure $( \widetilde{P} , \epsilon_P)$ and a lift $\widetilde{\sigma} : \widetilde{P} \to \widetilde{P}$ of $\sigma$ to a principal bundle automorphism of $\widetilde{P}$ satisfying:
\begin{itemize}
\item[(1)]{$\epsilon_P( \sigma(p) ) = \epsilon_P( \sigma(p) ) + k \; ({\rm mod} \; 2)$, that is, $\sigma : V \to V$ is orientation preserving or reversing according to the parity of $k$, and}
\item[(2)]{$\widetilde{\sigma}^2 = (-1)^{ \binom{k+1}{2} }$.}
\end{itemize}

\end{definition}

\begin{remark}\label{rem:trivk}
The trivial vector bundle $\mathbb{R}_-^k \to \{pt\}$ has a canonical Real spin$^c$-structure of type $k$. Indeed, the frame bundle is $P = O(k)$. The orientation is $\epsilon_P(g) = \epsilon(g) = det(g)$. The lift to $Pin^c(k)$ is $\widetilde{P} = Pin^c(k)$. The induced action of $\sigma$ on $P$ is $\sigma(p) = up$, where $u \in O(k)$ is given by $u = diag(-1,-1,\dots , -1)$. If $e_1, \dots , e_k$ denotes the standard basis of $\mathbb{R}^k$, then a lift of $u$ to $Pin^c(k)$ is given by $\widetilde{u} = e_1 e_2 \cdots e_k$. Then $\widetilde{\sigma} : \widetilde{P} \to \widetilde{P}$ may be taken to be $\widetilde{\sigma}(p) = \widetilde{u}p$. Observe that $\epsilon_P( \sigma(p) ) = \epsilon_P( up) = \epsilon(u) + \epsilon(p) = \epsilon(p) + k$ and $\widetilde{\sigma}^2 = \widetilde{u}^2 = (-1)^{\binom{k+1}{2}}$.
\end{remark}

The significance of Definition \ref{def:spinck} is given by the following:
\begin{proposition}\label{prop:realspinck}
Let $X$ be a topological space with a Real structure $\sigma$. Let $V \to X$ be a real rank $n$ orthogonal vector bundle. Suppose an involutive lift of $\sigma$ to $V$ is given. Let $(\widetilde{P} , \epsilon_P , \widetilde{\sigma})$ be a Real spin$^c$-structure of type $k$. Then $(\widetilde{P} , \epsilon_P , \widetilde{\sigma})$ determines an isomorphism of lifting gerbes $L(V) \cong L( \mathbb{R}_-^k )$.
\end{proposition}
\begin{proof}
First recall the definition of the lifting gerbe $L(V)$. We take $Y = P$, the $O(n)$-frame bundle of $V$, we take the circle bundle $C \to Y^{[2]}$ to be $C = P \times Pin^c(n)$, which maps to $Y^{[2]}$ via $(p,h) \mapsto (p , p \pi(h))$ (where $\pi : Pin^c(n) \to O(n)$ is the projection) and $\lambda : C_{12} \otimes C_{23} \to C_{13}$ is given by $\lambda( (p_1,h_1) , (p_2 , h_2) ) = (p_1 , h_1 h_2)$. The grading on $C$ is given by $\epsilon( p,h ) = \widetilde{\epsilon}(h)$ (where $\widetilde{\epsilon}(h) = \epsilon( \pi(h) ) = det( \pi(h))$). The Real structure on $C$ is given by $\sigma(p,h) = (\sigma(p) , c(h) )$ (where $c$ is the involution on $Pin^c(n)$ which covers the identity on $O(n)$ and acts as complex conjugation on $S^1$).

It will be convenient to define $\lambda , \mu \in \mathbb{Z}/2\mathbb{Z}$ by taking $\lambda = k \; ({\rm mod} \; 2)$ and $\mu = \binom{k+1}{2} \; ({\rm mod} \; 2)$. Let $e_{ij} \in \mathbb{Z}/2\mathbb{Z}$ be defined by $e_{ij} = 0$ if $i = j$, $e_{ij} = 1$ if $i \neq j$. Let $m_{ijk} \in S^1$ be defined by $m_{010} = m_{101} = -1$, $m_{ijk} = 1$ otherwise. Then $[e_{ij}], [m_{ijk}]$ are generators for $H^1_{\mathbb{Z}_2}(pt ; \mathbb{Z}_2)$ and $H^3_{\mathbb{Z}_2}(pt ; \mathbb{Z}_-)$. Consider a point gerbe $\mathcal{U}_k$ with graded Dixmier--Douady class equal to $( \lambda e_{ij} , m_{ijk}^{\mu} )$. Specifically, we take $\mathcal{U}_k$ to be the gerbe defined with respect to the cover $\{0,1\} \to \{pt\}$ with circle bundle $C(\mathcal{U}) = \bigcup_{i,j} C(\mathcal{U})_{ij}$, $C(\mathcal{U})_{ij} = S^1$, grading $\epsilon = \lambda e_{ij}$ on $C(\mathcal{U})_{ij}$, Real structure $\sigma(z) = \overline{z}$ for $z \in C(\mathcal{U})_{ij}$ and multiplication $\lambda_{ijk}( z_1 , z_2 ) = m_{ijk}^{\mu} z_1 z_2$.

From the data $(\widetilde{P} , \epsilon_P , \widetilde{\sigma})$, we will construct a canonical isomorphism $L(V) \to \mathcal{U}_k$. In light of Remark \ref{rem:trivk}, the same construction applied to $\mathbb{R}_-^k$ in place of $V$ gives a canonical isomorphism $L(\mathbb{R}_-^k) \to \mathcal{U}_k$. Combining these isomorphisms, we get a canonical isomorphism $L(V) \to L(\mathbb{R}_-^k)$, as claimed.

To define the isomorphism $L(V) \to \mathcal{U}_k$, we must first take the pullbacks of $L(V)$ and $\mathcal{U}_k$ to the common cover $Z = Y \times \{0,1\} = Y_0 \cup Y_1$, where $Y_i = Y \times \{i\} = P \times \{i\}$. Note that $Z^{[2]}$ is four copies of $Y^{[2]}$ indexed by pairs $i,j \in \{0,1\}$. We write the components as $Y_{ij}^{[2]}$.

On $Z = Y_0 \cup Y_1$ define a Real circle bundle $N \to Z$ by taking $N_0 = N_1 = \widetilde{P}$. We give $N$ a grading by declaring the grading on $N_i$ to be $\epsilon_P + \lambda i$. We define a Real structure $\sigma_N$ on $N$ by letting $\sigma_N : N_0 \to N_1$ equal $\widetilde{\sigma}$ and letting $\sigma_N : N_1 \to N_0$ be $(-1)^\mu \widetilde{\sigma}$. Since $\widetilde{\sigma}^2 = (-1)^\mu$, we have $\sigma_N^2 = 1$, so $\sigma_N$ makes $N$ into a Real line bundle. Moreover the grading on $N$ is defined in such a way that $\sigma_N$ is grading preserving, so $(N , \sigma_N)$ is a graded Real line bundle.

We will construct a strict isomorphism $\varphi$ from $(\delta N) \otimes \mathcal{U}_k$ to $L(V)$ (where $\mathcal{U}_k$ and $L(V)$ are pulled back to $Z$). Then $(N , \varphi)$ defines an isomorphism $\mathcal{U}_k \to L(V)$. First note that
\[
(\delta N)_{ij} = \{ (u_1,u_2) \in \widetilde{P} \times \widetilde{P} \; | u_2 = u_1 h \text{ for some } h \in \widetilde{G} \}/\!\sim
\]
where $(u_1,u_2) \sim (u'_1 , u'_2)$ if $(u'_1 , u'_2) = (u_1 z , u_2 z)$ for some $z \in S^1$.

Since $u_1$ and $h$ completely determine $u_2$ via $u_2 = u_1 h$, we can identify $(\delta N)_{ij}$ with pairs $(u_1 , h) \in \widetilde{P} \times \widetilde{G}$ modulo $(u_1 , h) \sim (u_1 z , h)$, $z \in S^1$. Thus $(\delta N)_{ij} \cong P \times \widetilde{G} = L_{ij}$. So we have a manifest isomorphism of circle bundles $(\delta N)_{ij} = L_{ij}$. However this does not define an isomorphism of gerbes $(\delta N) \otimes \mathcal{U} \to L$ because it does not respect the Real structures or multiplication.

The Real structure on $(\delta N)_{ij}$ is given by $\sigma( u_1 , u_2 ) = ( (-1)^{\mu i} \widetilde{\sigma}(u_1) , (-1)^{\mu j} \widetilde{\sigma}(u_2) )$. In terms of the identification $(\delta N)_{ij} = L_{ij} = P \times \widetilde{G}$, this is given by $\sigma( p , h ) = ( \sigma(p) , (-1)^{\mu(i+j)}c(h)  )$.

The multiplication
\[
\left( (\delta N)_{ij} \otimes \mathcal{U}_{ij} \right) \otimes \left( (\delta N)_{jk} \otimes \mathcal{U}_{jk} \right) \to (\delta N)_{ik} \otimes \mathcal{U}_{ik}
\]
is given as follows (where $p_2 = p_1h_1$, $p_3 = p_2 h_2$):
\begin{align*}
& \left( ( p_1 , h_1 ) \otimes (1)_{ij} \right) \otimes \left( (p_2 , h_2 ) \otimes (1)_{jk} \right)  \\
& \quad = (-1)^{(\varepsilon(h_2) + \lambda e_{jk}) \lambda e_{ij}} (p_1 , h_1) \otimes (p_2 , h_2 ) \otimes (1)_{ij} \otimes (1)_{jk} \\
& \quad \mapsto (-1)^{(\varepsilon(h_2) + \lambda e_{jk}) \lambda e_{ij}} m_{ijk}^{\mu} (p_1 , h_1 h_2).
\end{align*}

Now, we define the map $\varphi_{ij} : (\delta N)_{ij} \otimes \mathcal{U}_{ij} \to L_{ij}$ to be:
\[
\varphi_{ij}(  (p , h) \otimes (1)_{ij} ) = (p , (-1)^{\lambda \varepsilon(h) i }  \chi_{ij} h)
\]
where we set $\chi_{00} = \chi_{11} = 1$, $\chi_{01} = \chi_{10} = e^{\pi i\mu/2}$ (to make sense of $e^{\pi i \mu/2}$, we regard $\mu$ as an element of $\{0,1\}$ rather than as an integer mod $2$). One finds that with this choice $\varphi_{ij}$ respects Real structures.

With these definitions, it is straightforward to check that $\varphi$ defines a strict isomorphism from $(\delta N) \otimes \mathcal{U}_k$ to $L(V)$, as claimed. The isomorphism is canonical since it depends only on $(\widetilde{P} , \epsilon_P , \widetilde{\sigma})$ and no additional choices were made in the construction.
\end{proof}

If $X$ is a smooth $n$-manifold and $\sigma_X$ is a smooth Real structure on $X$, then we define a {\em Real spin$^c$-structure on $X$ of type $k$} to be a Real spin$^c$-structure of type $k$ on the tangent bundle $TX$. Let $X,Y$ be smooth, compact manifolds of dimensions $n$ and $m$. Let $\sigma_X, \sigma_Y$ be smooth Real structures on $X,Y$ and let $f : X \to Y$ be an equivariant smooth map from $X$ to $Y$ ($ f \circ \sigma_X = \sigma_Y \circ f$). Then we have a pushforward map
\[
f_* : KR^j( X , L(TX) \otimes f^*(\mathcal{G}_Y) ) \to KR^{j-n+m}(Y , L(TY) \otimes \mathcal{G}_Y )
\]
where $\mathcal{G}_Y$ is any graded Real gerbe on $Y$. Suppose now that $X,Y$ have Real spin$^c$-structrues of types $k_X,k_Y \in \mathbb{Z}/4\mathbb{Z}$. Hence by Proposition \ref{prop:realspinck}, we have canonical isomorphisms $L(TX) \cong L(\mathbb{R}_-^{k_X})$, $L(TY) \cong L(\mathbb{R}_-^{k_Y})$. As explained in Section \ref{sec:tkr}, twisting by $L( \mathbb{R}_-^k)$ has the effect on twisted $KR$-theory as a grading shift by $-2k$. Hence in this case, the pushforward can be regarded as a map
\[
f_* : KR^{j}( X , f^*(\mathcal{G}_Y) ) \to KR^{j-n+m+2k_X-2k_Y}(Y , \mathcal{G}_Y ).
\]
In particular, if $\mathcal{G}_Y$ is taken to be trivial, we get
\[
f_* : KR^{j}( X ) \to KR^{j-n+m+2k_X-2k_Y}(Y ).
\]
In the case $Y = \{pt \}$ is a point, we get an index map for twisted Real spin$^c$-structures:
\[
ind = f_* : KR^j(X) \to KO^{j-n+2k_X}(pt).
\]
Notably, $f_*(1) \in KO^{2k_X-n}(pt)$ is the index of the Dirac operator associated to the spin$^c$-structure. Under the forgetful map $KO^{2k_X-n}(pt) \to K^{2k_X-n}(pt) = K^{-n}(pt)$, $f_*(1)$ equals the usual index $ind(D) \in \mathbb{Z}$ of the Dirac operator associated to the spin$^c$-structure (which is zero if $n$ is odd). From this, we obtain the following observations:
\begin{proposition}
Let $X$ be a compact smooth $n$-manifold with smooth Real structure $\sigma_X$. Let $\mathfrak{s}$ be a spin$^c$-structure and let $ind(D) \in \mathbb{Z}$ denote the index of the spin$^c$-Dirac operator associated to $\mathfrak{s}$. Suppose that $\mathfrak{s}$ admits the structure of a Real spin$^c$-structure of type $k \in \mathbb{Z}/4\mathbb{Z}$.
\begin{itemize}
\item[(1)]{If $n = 2k+4 \; ({\rm mod} \; 8)$ then $ind(D)$ is even.}
\item[(2)]{If $n = 2k \pm 2 \; ({\rm mod} \; 8)$ then $ind(D) = 0$.}
\end{itemize}
\end{proposition}
\begin{proof}
We have seen that $ind(D) \in K^{-n}(pt)$ admits a lift to a class in $KO^{2k-n}(pt)$. If $n = 2k + 4 \; ({\rm mod} \; 8)$, then $ind(D)$ is even, because the image of $KO^4(pt) \to K^0(pt)$ is $2\mathbb{Z}$. If $n = 2k \pm 2 \; ({\rm mod} \; 8)$, then $ind(D) = 0$, because the map $KO^{2k-n}(pt) = KO^{\pm 2}(pt) \to K^{0}(pt)$ is zero.
\end{proof}

\begin{remark}
If $n = 2k+1$ or $2k+2$ mod $8$, then we have a mod $2$ index $f_*(1) \in KO^{2k-n}(pt) \cong \mathbb{Z}_2$ which is not detected by complex $K$-theory.
\end{remark}

In what follows we will be intetested in the index not of a single Dirac operator $D$, but of a family $\{D_L\}_{L}$ of Dirac operators parametrised by the Jacobian torus $Jac(X) = H^1(X ; \mathbb{R})/H^1(X ; \mathbb{Z})$. The resulting index will be a class in the $KR$-theory of $Jac(X)$.

Let $X$ be a compact, oriented smooth $n$-manifold with smooth Real structure $\sigma_X$. Let $\mathfrak{s}$ be a spin$^c$-structure on $X$ and suppose that $\mathfrak{s}$ admits the structure of a Real spin$^c$-structure of type $k$. Let $D$ denote the spin$^c$-Dirac operator associated to $\mathfrak{s}$. Let $Jac(X) = T_X = V/\Lambda$ be the Jacobian of $X$, where $V = H^1(X ; \mathbb{R})$, $\Lambda = H^1(X ; \mathbb{Z})$. We define the {\em Albanese torus} $Alb(X) = A_X$ of $X$ to be the dual torus
\[
A_X = \frac{ Hom( H^1(X ; \mathbb{R}) , \mathbb{R}) }{Hom( H^1(X ; \mathbb{Z}) , \mathbb{Z})} = V^*/\Lambda^*.
\]
Define the {\em Albanese map} $a : X \to A_X$ as follows. Choose a basepoint $x_0 \in X$ and a $b_1(X)$-dimensional subspace $\widetilde{V} \subset \Omega^1_{cl}(X)$ of closed $1$-forms such that the projection $\widetilde{V} \to V = H^1(X ; \mathbb{R})$ to cohomology is an isomorphism. For $x \in X$, $a(x) \in A_X$ is defined as follows. Choose a smooth path $\gamma$ from $x_0$ to $x$. Then we get an element $\int_{\gamma} \in V^* = Hom( V , \mathbb{R})$ defined by $[\lambda] \mapsto \int_{\gamma} \lambda$, where $\lambda \in \widetilde{V}$ is the unique element of $\widetilde{V}$ representing $[\lambda] \in V$. Different choices of path from $x_0$ to $x$ differ by a class in $H_1(X ; \mathbb{Z})$, hence if $\gamma_1,\gamma_2$ are two such paths then $\int_{\gamma_1} - \int_{\gamma_2}$ belongs to $\Lambda^* \subset V^*$. Thus the image of $\int_{\gamma} \in V^*/\Lambda^* = A_X$ depends only on the endpoints $x_0,x$ and not on the choice of path $\gamma$. We set $a(x) = \int_{\gamma} \; ({\rm mod} \; \Lambda^*)$ where $\gamma$ is a path from $x_0$ to $x$. The Albanese map $a : X \to A_X$ depends on the choice of lift $\widetilde{V}$ of $V$ and also on the choice of basepoint $x_0$, but these choices do not change the homotopy class of $a$.

Over $X \times Jac(X)$ we have a line bundle with connection $\mathcal{P}_X \to X \times Jac(X)$ with the property that for each $[L] \in Jac(X)$, $\mathcal{P}_X |_{X \times \{[L]\}}$ is isomorphic to $L$. To be specific, we will take $\mathcal{P}_X = ( a \times id_{Jac(X)} )^*( \mathcal{P} )$, where $\mathcal{P} \to A_X \times T_X$ is the Poincar\'e line bundle. By coupling the Dirac operator $D$ to the line bundle $\mathcal{P}_X$, we get a family of Dirac operators over $T_X$. Taking the families index of this family gives a class $ind(D) \in K^{-n}( T_X )$. Letting $\pi : X \times T_X \to T_X$ denote the projection to the second factor, it follows that $ind(D) = \pi_*( \mathcal{P}_X )$. Note that the spin$^c$-structure on $X$ determines a spin$^c$-structure on the fibres of $\pi : X \times T_X \to T_X$ and hence the push-forward $K^0(X \times T_X) \to K^{-n}(T_X)$ is defined.

Next we wish to promote $ind(D) \in K^{-n}(T_X)$ to a class in $KR$-theory. The involution $\sigma_X$ induces an involution $(\sigma_X)_*$ on $H_1(X ; \mathbb{Z})$ and hence also on the torsion-free quotient $H^1(X ; \mathbb{Z})_{tf}$. Noting that $\Lambda^* = Hom(H^1(X ; \mathbb{Z}) , \mathbb{Z}) \cong H_1(X ; \mathbb{Z})_{tf}$, we get an involution on $\Lambda^*$ denoted by $\sigma_*$, hence also an involution $\sigma_*$ on $V^* = \Lambda^* \otimes_{\mathbb{Z}} \mathbb{R}$ and an involution $\sigma_{A_X}$ on $A_X$. On $\Lambda = H^1(X ; \mathbb{Z})$ we take the involution $-(\sigma_X^*)$ which is {\em minus} the pullback of $\sigma_X$. This induces an involution $-\sigma^*$ on $V$ and an in turn an involution $\sigma_{T_X}$ on $T_X$. One motivation for using minus the pullback is that if $T_X = Jac(X)$ is identified with flat unitary line bundles, then minus the pullback corresponds to $L \mapsto \sigma^*( \overline{L})$. In the case $X$ is a complex manifold and $\sigma_X$ is anti-holomorphic, this is the natural action of $\sigma$ on holomorphic line bundles. Another motivation for using minus the pullback is that it makes $(T_X , \sigma_{T_X}) , (A_X , \sigma_{A_X})$ into a Real $T$-dual pair. Lastly and perhaps most crucially, the minus sign makes the Poincar\'e line bundle $\mathcal{P} \to A_X \times T_X$ into a Real line bundle. Following \cite[\textsection 5.2]{bar4} (see also \cite{fk}), we can explicitly construct $\mathcal{P}$ as follows. Take the trivial line bundle $\widetilde{\mathcal{P}} = V^* \times V \times \times \mathbb{C}$ on $V^* \times V$ and let $\Lambda^* \times \Lambda$ act on $\widetilde{P}$ by
\[
(\mu , \lambda) \cdot ( w , v , z) = (w + \mu , v + \lambda , e^{2\pi i \langle \mu , v \rangle } z).
\]
Then $\mathcal{P}$ may be defined as the quotient $\mathcal{P} = \widetilde{\mathcal{P}}/(\Lambda^* \times \Lambda)$. The Real structure on $\mathcal{P}$ is induced the the Real structure on $\widetilde{\mathcal{P}}$ given by
\[
\sigma( w , v , z ) = ( \sigma_*(w) , -\sigma^*(v) , \overline{z}).
\]
Note that the presence of the minus sign is needed to ensure that the Real structure commutes with the action of $\Lambda^* \times \Lambda$. The Real Poincar\'e line bundle $\mathcal{P}$, regarded as an isomorphism of trivial gerbes clearly satisfies the Poincar\'e property (indeed the Poincar\'e property was defined precisely so that this happens) and therefore associated to $\mathcal{P}$ is a Fourier--Mukai transform $\Phi : KR^*(A_X , L(A_X) ) \to KR^{* - b_1(X)}(T_X)$. As shown below, the twist by $L(A_X)$ will simply result in a grading shift.

We will assume that the action of $\sigma_X$ on $X$ is not free and we will take $x_0$ to be a fixed point of $\sigma_X$. We will also choose the subspace $\widetilde{V}$ to be $\sigma$-invariant (that is, $\sigma^*( \widetilde{V} ) = \widetilde{V}$). This can always be done, for example we could let $\widetilde{V}$ be the space of $g$-harmonic $1$-forms for a $\sigma$-invariant metric $g$. It then follows that the Albanese map $a : X \to A_X$ is equivariant in the sense that $a \circ \sigma_X = \sigma_{A_X} \circ a$. Then $\mathcal{P}_X = (a \times id_{T_X})^*(\mathcal{P})$ is a Real line bundle on $X$ and so defines a class $[\mathcal{P}_X] \in KR^0(X \times T_X)$. The projection $\pi : X \times T_X \to T_X$ is equivariant and the vertical tangent bundle has a Real spin$^c$-structure of type $k$ (since by assumption, the spin$^c$-structure $\mathfrak{s}$ admits a Real spin$^c$-structure of type $k$). Taking the families index of $D$ coupled to $\mathcal{P}_X$ then yields a lift of $ind(D) \in K^{-n}(T_X)$ to a class $ind_R(D) \in KR^{2k-n}(T_X)$ given by
\[
ind_R(D) = \pi_*( \mathcal{P}_X ) \in KR^{2k-n}(T_X).
\]
We call this the {\em Real families index of $D$}.

We note that the tangent bundle of $A_X$ is canonically isomorphic to $V^*$ and the tangent bundle of $T_X$ is canonically isomorphic to $V_- = V \otimes_{\mathbb{R}} \mathbb{R}_-$. Let $b = b_1(X)$ and let $b_+,b_-$ denote the dimensions of the $\pm 1$-eigenspaces of $\sigma^*$ on $H^1(X ; \mathbb{R})$. Then $V^* \cong \mathbb{R}^{b_+} \oplus \mathbb{R}^{b_-}_-$ and $V_- \cong \mathbb{R}^{b_+}_- \oplus \mathbb{R}^{b_-}$. In particular, the translation invariant spin$^c$-structure on $A_X$ has a Real structure of type $b_+$ and the translation invariant spin$^c$-structure on $T_X$ has a Real structure of type $b_-$. It follows that the pushforward with respect to the Albanese map $a : X \to A_X$ takes the form
\[
a_* : KR^*(X) \to KR^{*-n+b_+ - b_- +2k}(A_X)
\]
and the Real Fourier--Mukai transform for $(A_X,T_X , \mathcal{P})$ takes the form
\[
\Phi : KR^*(A_X) \to KR^{*-b_+ + b_-}(T_X).
\]

\begin{proposition}\label{prop:realdir}
We have
\[
ind_R(D) = \Phi( \alpha )
\]
where
\[
\Phi : KR^*( A_X ) \to KR^{* - b_+ + b_-}( T_X )
\]
is the Real Fourier--Mukai transform and
\[
\alpha = a_*(1) \in KR^{-n+b_+ - b_- + 2k}(A_X)
\]
is the pushforward of $1 \in KR^0(X)$ under the Albanese map $a : X \to A_X$.
\end{proposition}
\begin{proof}
Let $p_1 : A_X \times T_X \to A_X$ and $p_2 : A_X \times T_X \to T_X$ be the projections. Then $\Phi( x ) = (p_2)_*( \mathcal{P} \otimes p_1^*(x) )$. Using the commutative diagram
\[
\xymatrix{
X \times T_X \ar[rr]^-{a \times id_{T_X}} \ar[d]^-{\pi} & & A_X \times T_X \ar[d]^-{p_2} \\
T_X \ar@{=}[rr] & & T_X
}
\]
we deduce that
\begin{align*}
ind_R(D) &= \pi_*( \mathcal{P}_X ) \\
&= (p_2)* \circ (a \times id_{T_X})_* ( (a \times id_{T_X})^*(\mathcal{P}) ) \\
&= (p_2)_*( \mathcal{P} \otimes (a \times id_{T_X})_*(1) ) \\
&= (p_2)_*( \mathcal{P} \otimes (p_1)^*( a_*(1) ) ) \\
&= \Phi( a_*(1) ).
\end{align*}

\end{proof}

\begin{remark}
In the course of proving Proposition \ref{prop:realdir} we needed to assume that $\sigma_X$ has a fixed point, so that the Albanese map $a : X \to A_X$ was equivariant in the sense that $a \circ \sigma_X = \sigma_{A_X} \circ a$. If $\sigma_X$ does not have fixed points, then we can still define an involution $\sigma'_{A_X}$ on $A_X$ for which $a$ is equivariant. This involution will have the form $\sigma'_{A_X}(t) = \sigma_{A_X}(t) + u$, where $u \in A_{X}$ is given by $u([\lambda]) = \int_{\gamma} \lambda$ and $\gamma$ is a path from $\sigma(x_0)$ to $x_0$. From the point of view of Real affine torus bundles, this means $(A_X , \sigma_X)$ carries a non-trivial Chern class. We still have a Fourier--Mukai transform, but the $T$-dual of $(A_X , \sigma'_{A_X})$ will be of the form $( T_X , \sigma_{T_X} , \mathcal{G})$ for some non-trivial Real gerbe on $T_X$. We can still define a Real index $ind_R(D) = \Phi( a_*(1) )$ which is a lift of $ind(D)$, but now $ind_R(D)$ will be valued in a twisted $KR$-theory group $KR^*( T_X , \mathcal{G})$.
\end{remark}


\bibliographystyle{amsplain}

\end{document}